\documentclass[12pt,reqno]{amsart}

\usepackage[text={400pt,660pt},centering]{geometry}
\usepackage{color}
\usepackage{comment}
\usepackage{esint,amssymb}
\usepackage{graphicx}
\usepackage{mathtools}
\usepackage{xpatch}
\usepackage[colorlinks=true, pdfstartview=FitV, linkcolor=blue, citecolor=blue, urlcolor=blue,pagebackref=false]{hyperref}
\usepackage{microtype}

\usepackage{bm}
\usepackage{xfrac}
\usepackage{scalerel} %for scaling the boxes 
\usepackage{dsfont}
\usepackage{mathrsfs}
\usepackage{eufrak}
\usepackage[font={footnotesize}]{caption}
\usepackage{comment}
\usepackage[shortlabels]{enumitem}
\usepackage[normalem]{ulem}
\usepackage[foot]{amsaddr}

\newtheorem{proposition}{Proposition}
\newtheorem{theorem}[proposition]{Theorem}
\newtheorem{lemma}[proposition]{Lemma}
\newtheorem{corollary}[proposition]{Corollary}

\theoremstyle{remark}
\newtheorem{remark}[proposition]{Remark}

\theoremstyle{definition}
\newtheorem{definition}[proposition]{Definition}

\numberwithin{equation}{section}
\numberwithin{proposition}{section}

\newcommand{\dist}{\text{dist}}

\newcommand{\NN}{\mathbb{N}}
\newcommand{\PP}{\mathbb{P}}

\newcommand{\RR}{\mathbb{R}}

\newcommand{\EE}{\mathbb{E}}

\newcommand{\ep}{\epsilon}
\newcommand{\ve}{\varepsilon}

\DeclareMathOperator*{\limsups}{limsup^{\ast}}
\DeclareMathOperator*{\liminfs}{liminf_{\ast}}

\newcommand{\al}{\alpha}

\newcommand{\ga}{\gamma}

\newcommand{\la}{\lambda}

\newcommand{\W}{\mathbf{X}}
\newcommand{\B}{\mathbf{B}}

\newcommand{\Ga}{\Gamma}

\newcommand{\La}{\Lambda}
\newcommand{\X}{\Xi}

\newcommand{\norm}[1]{\|#1\|}
\newcommand{\tr}{\mbox{tr}}
\newcommand{\di}{\mbox{div}} 
\newcommand{\id}{\mbox{Id}} 
\newcommand{\rst}[1]{\ensuremath{{\mathbin\upharpoonright}%
\raise-.5ex\hbox{$#1$}}}

\let\emptyset\varnothing

\newcommand{\avg}{\mu}
\newcommand{\bmo}{W^{1}}
\newcommand{\bmd}{W^{\mathbf{d}}}

\newcounter{jscan}
\newcounter{aprimescan}
\newcounter{gscan}
\newcounter{bwscan}
\newcounter{cscan}
\newcounter{hscan}
\newcounter{hprimescan}
\newcounter{fscan}
\newcounter{pscan}
\newcounter{sscan}
\newcounter{iscan}
\newcounter{rscan}
\newcounter{ascan}
\newcounter{rrscan}
\newcounter{fpscan}
\newcounter{goscan}
\newcounter{gfscan}
\newcounter{bscan}

\renewcommand{\thegfscan}{G5}
\renewcommand{\thegoscan}{G0}
\renewcommand{\thejscan}{J\arabic{jscan}}
\renewcommand{\thegscan}{G\arabic{gscan}}

\renewcommand{\thecscan}{$\mathcal{C}$\arabic{cscan}}

\renewcommand{\theascan}{A\arabic{ascan}}
\renewcommand{\theaprimescan}{A\arabic{aprimescan}'}

\numberwithin{equation}{section}

\newcommand\restr[2]{{% we make the whole thing an ordinary symbol
  \left.\kern-\nulldelimiterspace % automatically resize the bar with \right
  #1 % the function
  \vphantom{\big|} % pretend it's a little taller at normal size
  \right|_{#2} % this is the delimiter
  }}

\newcommand{\bP}{\mathbb{P}}

\newcommand{\E}{\mathbb{E}}
\newcommand{\N}{\mathbb{N}}

\newcommand{\V}{\mathbb{V}}
\newcommand{\Z}{\mathbb{Z}}
\newcommand{\R}{\mathbb{R}}

\newcommand{\indc}{\mathds{1}}

\newcommand{\eps}{\epsilon}

\newcommand{\cF}{\mathcal{F}}
\newcommand{\cS}{\mathcal{S}}

\newcommand{\cU}{\mathcal{U}}

\renewcommand{\X}{\mathbf{X}}
\def\restrict#1{\raise-.5ex\hbox{\ensuremath|}_{#1}}

\begin{document}

\title[Generalized Front Propagation for Stochastic Models]{Generalized Front Propagation for Spatial Stochastic Population Models}

\author{Thomas Hughes$^\dagger$}
\address{$^\dagger$Department of Mathematical Sciences, University of Bath, Claverton Down, Bath, United Kingdom BA2 7AY}
\email{$^\dagger$th2275@bath.ac.uk}

\author{Jessica Lin$^\ast$}
\address{$^\ast$Department of Mathematics, McGill University, Burnside Hall, 805 Sherbrooke Street West, Montreal, QC, Canada H3A 0B9}
\email{$^\ast$jessica.lin@mcgill.ca}


\maketitle

\begin{abstract}
We present a general framework which can be used to prove that, in an annealed sense, rescaled 
spatial stochastic population models converge to generalized propagating fronts. Our work is motivated by recent results of Etheridge, Freeman, and Penington \cite{EFP2017} and Huang and Durrett \cite{DH2021}, who proved convergence to classical mean curvature flow (MCF) for certain spatial stochastic processes, up until the first time when singularities of MCF form. Our arguments rely on the level-set method and the abstract approach to front propagation introduced by Barles and Souganidis \cite{BS}. 

This approach is amenable to stochastic models equipped with moment duals which satisfy certain general and verifiable properties. Our main results improve the existing results in several ways, first by removing regularity conditions on the initial data, and second by establishing convergence beyond the formation of singularities of MCF. In particular, we obtain a general convergence theorem which holds globally in time. This is then applied to all of the models considered in \cite{EFP2017} and \cite{DH2021}. 
\end{abstract}

\section{Introduction}

\subsection{Broad Ideas and Informal Discussion of the Main Results} 
We are interested in spatial stochastic population models which, asymptotically, exhibit phase separation in an annealed sense; by this, we mean that the expectation of the population concentrates to two stable states, and we aim to characterize how the states evolve with time. Phase separation occurs in numerous models arising in mathematical biology and mathematical physics which exhibit bistability. The approach presented here is general in nature, but the particular problems we consider model evolutionary and ecological phenomena in biological systems.

Let $(w_{t}, t\geq 0)$ denote a generic stochastic process taking values in $[0,1]^E$, where $E$ is some spatial domain which will usually be $\R^{\mathbf{d}}$ or $\Z^{\mathbf{d}}.$ Under a suitable scaling by a parameter $\eps>0$, we consider the rescaled process $(w^{\eps}_{t}, t\geq 0)_{\eps>0}$. Let us suppose that 0 and 1 are so-called stable states, where the notion of stability will be made precise later. Let $p: \RR^{\mathbf{d}}\rightarrow [0,1]$ be such that there exists a non-empty open set $\Theta_{0}\subseteq \R^{\mathbf{d}}$ satisfying 
\begin{equation*}
\Theta_{0}=\left\{p(\cdot)<\sfrac{1}{2}\right\} \quad\text{and}\quad \overline{\Theta_{0}}^{c}=\left\{p(\cdot)>\sfrac{1}{2}\right\}.
\end{equation*}
We then let $\Ga_{0}:=\partial \Theta_{0}$. Suppose that $w^{\eps}_{0}=p$, and we write $\bP^\epsilon_{p}$ to denote the law of the process $w^{\eps}$, started from initial condition $p$, and similarly we denote the associated expectation operator by $\E^\epsilon_{p}$. 

In broad terms, we prove the following: there exist open sets $\left\{\Theta^{0}_{t}\right\}_{t>0}, \left\{\Theta^{1}_{t}\right\}_{t>0}\subseteq \R^{\mathbf{d}}$ such that, 
\begin{equation}\label{e.genericconv}
\lim_{\eps\to 0}\EE^{\eps}_{p}[w^{\eps}_{t}(x)]=\begin{cases} 0&\text{locally uniformly in $\bigcup_{t>0} \left\{t\right\}\times \Theta^{0}_{t}$,}\\
1&\text{locally uniformly in $\bigcup_{t>0} \left\{t\right\}\times \Theta^{1}_{t}$}. 
\end{cases}
\end{equation}
The sets $(\Theta^{0}_{t},t>0)$ and $(\Theta^{1}_{t},t>0)$ can be completely characterized in terms of the viscosity solution of an associated PDE, known as the mean-curvature equation, which we will later describe in detail. In particular, there exists a function $u: [0, \infty)\times \R^{\mathbf{d}}\rightarrow \R$ solving this PDE with an initial condition defined in terms of $p$, such that
\begin{equation}\label{e.ilevelsets}
\Theta^{0}_{t}:=\left\{u(t, \cdot)<0\right\}\quad\text{and}\quad \Theta^{1}_{t}:=\left\{u(t, \cdot)>0\right\}.
\end{equation}
This characterization of the sets is referred to as generalized mean-curvature flow (MCF). 

The first such result of this nature, for stochastic spatial models, can be found in the works of Katsoulakis and Souganidis \cite{KS1,KS2, KS3} and Barles and Souganidis \cite{BS}. In \cite{KS1}, the authors studied a spin system subject to Glauber-Kawasaki stirring dynamics, while in \cite{BS, KS2, KS3}, they considered stochastic Ising models with long range interactions and general spin-flip dynamics. The approach taken in these papers is based on comparing relevant annealed quantities of the system with special solutions (namely traveling/standing waves) of an associated reaction-diffusion equation. %

Recently, there have been a series of works on different spatial stochastic processes modelling biological systems, which also exhibit asymptotic phase separation evolving according to \emph{classical} MCF. For instance, this was examined in the work of Etheridge, Freeman, and Penington \cite{EFP2017} for a ternary branching Brownian motion (BBM) subject to a majority voting mechanism, and also for a version of the Spatial-Lambda-Fleming-Viot (SLFV)-process. This version of the SLFV was introduced to model the behaviour of hybrid zones arising in populations exhibiting selection against heterozygosity (see \cite{EFP2017} for details.)
Huang and Durrett \cite{DH2021} further analyzed several interacting particle systems modelling ecological phenomena. These included a sexual (bi-parental) reproduction model subject to fast stirring dynamics and two perturbations of the voter model: the Lotka-Volterra perturbation, introduced by Neuhauser and Pacala \cite{NeuhauserPacala} as an individual-based model for competition between species, and the nonlinear voter model of Molofsky et al. \cite{molo}, which was introduced to study co-existence of competing species.

 The proofs in \cite{EFP2017, DH2021} did not rely at all on PDE techniques, and instead proposed a more ``probabilistic approach'' for analyzing the large-scale dynamics of these models. The convergence results in these papers take place under certain regularity assumptions on $p$, and hold up until the first time when the MCF forms singularities. Our main motives in this paper are as follows: (1) to prove convergence results under minimal regularity assumptions on the function $p$, (2) to extend these convergence results to be \emph{global in time}, and (3) to demonstrate how one can prove generalized front propagation convergence results using purely probabilistic techniques. In particular, we prove that the biological models from \cite{EFP2017,DH2021} converge to generalized MCF, and thereby obtain information about the phase separation of these models globally in time, including in scenarios when the interfaces develop geometric singularities. We emphasize that global-in-time convergence means that we can determine the limiting behaviour of the model at all times, but that the convergence itself is locally uniformly in \emph{space-time}, as in \eqref{e.genericconv}.

Our work is based on a methodology introduced by Barles and Souganidis \cite{BS}, which provides four sufficient conditions on rescaled phase field models which guarantee convergence to a generalized front propagation. As in \eqref{e.ilevelsets}, the interface evolution is characterized via the level-set method and the theory of viscosity solutions. In \cite{BS}, the authors apply this methodology to several different rescaled PDE models and a class of stochastic Ising models. The framework in \cite{BS} encapsulates the priorly mentioned results of Katsoulakis and Souganidis \cite{KS1,KS2, KS3}. To our knowledge, all prior applications of the method of \cite{BS} rely on verifying these conditions using properties of PDEs; in particular, the existence of front-like solutions (i.e. solutions which are functions of $x\cdot e$ for some direction $e$) has been crucial.

In the same spirit as the ``probabilistic approach'' used in \cite{EFP2017, DH2021}, we verify the general conditions introduced by Barles and Souganidis \cite{BS} using probabilistic arguments. By this, we mean that  we circumvent the use any PDEs or front-like (traveling/standing wave) solutions at the microscopic (rescaled) level. Since the existence of front-like solutions is not always guaranteed, we are hopeful that this will allow for broader applications of the methodology of \cite{BS} in probabilistic problems. In this present work, we only consider examples in which the interface evolves according to MCF, but the framework we introduce allows for more general flows.

While a version of the result of \cite{BS} can be formulated directly in terms of $\E_{p}^{\eps}[w^{\eps}_{t}(x)]$, we take the perspective of \cite{EFP2017, DH2021}; we work with the \emph{moment dual} of $(w^{\eps}_{t}, t\geq 0)_{\eps>0}$, which is more accessible in terms of verifying the general conditions of \cite{BS}. For the models considered in this work, the dual process is given by a branching process which may or may not have coalescences, and the duality relation is made via a voting algorithm applied to the dual process. Even by working with the dual, for each individual model considered in \cite{EFP2017, DH2021}, verifying the general conditions of \cite{BS} requires considerable effort. 

This was our motivation in formulating our main result, Theorem \ref{t.realmain}, which states that, under several checkable conditions on the dual, $\EE^{\eps}_{p}[w^{\eps}_{t}]$ 
 converges to generalized MCF (i.e. \eqref{e.genericconv} holds). The hypotheses of Theorem \ref{t.realmain} are general enough to capture all of the models considered in \cite{EFP2017, DH2021}, thus allowing us to extend the convergence results of \cite{EFP2017, DH2021} for more general initial conditions defined according to $p$, and past the time of singularities. As discussed before, the limiting phase separation we prove is associated to a generalized notion of MCF defined via the level-set method. This flow is well-defined even when the MCF develops singularities, but it can form pathologies of its own. In particular, it is possible for the zero set defining the interface to have non-empty interior, effectively creating a macroscopic region where the limiting behaviour of our models cannot be determined. Whether or not interior is formed depends on the kind of singularity formed by the MCF, and the classification of singularities for which this does and does not occur is an active area in geometry.

The properties of the dual process and the associated voting algorithm serve as a ``probabilistic proxy'' for the  properties of PDEs that were priorly used in \cite{BS}. For one model, we see a direct connection between the approach of using moment duals and the approach using PDEs. As is shown in \cite[Theorem 2.2]{EFP2017}, in the case of ternary BBM subject to a majority voting procedure, the expected value of the outcome of the voting algorithm exactly solves a rescaled reaction-diffusion equation (notably, the rescaled Allen-Cahn equation). However, it is not in general true that a stochastic quantity (related to the dual) solves a PDE, and thus the arguments of \cite{EFP2017, DH2021}, and the present paper, do not \emph{rely} on any PDE techniques to prove convergence. In reference to the priorly mentioned use of front-like solutions in \cite{KS1, KS2, KS3, BS}, we highlight that the voting algorithm of the dual process can in fact be compared with the voting algorithm applied to a simplified one-dimensional version of the process (see Proposition \ref{prop:onedcouplefulltime}); this comparison between multidimensional and one-dimensional versions of the process is analogous to the one-dimensional nature of front-like solutions. 

\subsection{Further Literature Review and Comparison with Other Works}

As previously mentioned, the present work is most closely related to the papers of Etheridge, Freeman, and Penington \cite{EFP2017} and Huang and Durrett \cite{DH2021}; we extend their results, in the sense that we can handle more general choices of the initial function $p$, and we provide a global in time interpretation of the convergence. Indeed, letting $\mathrm{T}$ denote the first time where MCF forms singularities, for $t<\mathrm{T}$, the convergence in \eqref{e.genericconv} implies convergence to classical MCF (see Section \ref{s.mcf} for a discussion). On the other hand, we point out that they \cite{EFP2017, DH2021} prove quantitative convergence rates for times $t<\mathrm{T}$, whereas our convergence results do not have a rate. 

Very recently, Becker, Etheridge, and Letter \cite{BEL} have explored rescaled branching stable processes which converge to classical MCF, adapting the approach developed in \cite{EFP2017} to the setting of jump processes. Our methods could likely be extended to this setting, in the same spirit as the long-range interaction results of \cite{BS,KS2,KS3}, but we do not pursue this here. 

For a more PDE-based approach to the study of interacting particle systems, in addition to the priorly discussed works of Katsoulakis and Souganidis \cite{KS1, KS2, KS3}, we also mention some of the more recent developments in this direction. Kettani et al. \cite{Funakietal} consider Glauber and zero-range interacting particle systems, and they prove that classical MCF arises as the hydrodynamical limit in this setting. This was extended past the nearest neighbor case by Funaki et al. in \cite{Funakietal2}. All of these results are roughly based on comparing the interacting particle system to a special solution of an associated PDE and using properties of these solutions. More generally, there have been many works relating the study of interacting particle systems with MCF; see for example \cite{ANG, GBA, Funaki}, and references therein. 

Since the publication of \cite{EFP2017}, there has also been avid interest in further understanding the relationship between voting algorithms applied to branching processes and reaction-diffusion equations. We mention the very recent works of \cite{AHR, KRZ, koskela2024bernoulli}, which explore different models related to this question. 

\subsection{Outline of the Paper}
In Section \ref{s.prelims}, we establish the foundations needed in order to state our main results. This includes: a general framework for the spatial stochastic process, the statement of all of our hypotheses, and an introduction to generalized flows in the spirit of \cite{BS}. We state our main result, Theorem~\ref{t.realmain}, in Section \ref{s:abstract}. In Section \ref{s.BS}, we present the four conditions of \cite{BS} expressed in terms of the the voting algorithm and the dual process (which we refer to as \eqref{j.1}-\eqref{j.4}). It further contains the statement of Theorem~\ref{t.generaldual} which yields a general convergence result (for arbitrary flows) under the assumptions of \eqref{j.1}-\eqref{j.4} and the approximate dual property \eqref{eq:approxdualdef}.

In Section~\ref{s:vote}, we provide a precise framework for the voting algorithm and the verification of \eqref{j.2} and \eqref{j.4} under our main hypotheses. In Section \ref{s.J1} and Section \ref{s:genthm}, we complete the proof of Theorem \ref{t.realmain} by verifying the two most technical conditions, \eqref{j.1} and \eqref{j.4}, under the given hypotheses. Section \ref{s.bbmslfv} and Section \ref{s.ips} consist of verifying that the models considered in \cite{EFP2017} and \cite{DH2021} respectively satisfy the hypotheses of Theorem \ref{t.realmain}, and hence exhibit convergence to generalized MCF. Finally in the Appendix (Section \ref{s.app}), we collect some background from \cite{BS} related to generalized flows and generalized front propagation, which may be of interest to readers who are unfamiliar with that material.

\section{Preliminaries and Statement of the Main Results}\label{s.prelims}
\subsection{Spatial Stochastic Markov processes and Duals} \label{s.modelsduals}
The results of this paper are presented under a very general framework which encompasses many problems of interest. We devote this section to describing the full generality of this framework, and introduce \eqref{eq:approxdualdef}, one of the necessary hypotheses for our approach.

We consider a spatial stochastic process $(w_t, t \geq 0)$. In general, $w_t$ may be discrete, e.g. an interacting particle system taking values in the state space $\{0,1\}^{\Z^\mathbf{d}}$, or continuous, in which case we will generally assume $w_t \in \mathcal{B}(\R^{\mathbf{d}})$, the space of Borel measurable functions on $\R^{\mathbf{d}}$. For the present discussion, we will use the latter, but the main points carry through {\it mutatis mutandum} for discrete state spaces. In all cases, we will consider a deterministic function $p: \R^{d}\rightarrow [0,1]$ which \emph{generates} the initial condition. By this, we mean that either 
\begin{equation}\label{e.pdef}
w_{0}(\cdot)=p(\cdot)\quad\text{or}\quad \PP[w_{0}(\cdot)=1]=p(\cdot).
\end{equation}
We let $\PP_{p}$ denote the corresponding law of $(w_{t}, t\geq 0)$ and $\EE_{p}$ the expectation; in the special case when $p\equiv b$, for $b\in \R$ a constant, we simply write $\PP_{b}$ and $\EE_{b}$. We assume throughout that $w_t(\cdot) \in [0,1]$, and $w_{t}(x)$ represents information about the population at location $x$ and at time $t$. In general, the models we consider are driven by some combination of motion, dispersion, and interaction.

Our results rely on moment duality. Recall that in general, the dual process, which we denote by $(\hat{X}_{t}, t\geq 0)$, associated to $(w_{t}, t\geq 0)$ allows us to compute $\E_{p}[w_{t}(x)]$. We set $(\hat{\W}_{t}, t\geq 0)$ to be the associated dual historical process, which is defined as the backwards in time process which describes the spatial locations of the ancestors of an individual, as well as their historical trajectories. For each $x\in \R^\mathbf{d}$, we let $\hat{Q}_{x}$ denote a probability measure under which $(\hat{\W}_{t}, t\geq 0)$ has the law of the historical process started at $x$.

The duality is established via a voting algorithm $\V(\cdot; p)$ defined on the historical process $\hat{\W}_{t}$, which also depends on the initial function $p$. $\V(\cdot;p)$ can be thought of as an algorithm via which ancestral information is passed  through the ancestral graph of the historical dual process. The general form of the duality relationship, which can be established for many models of interest, is that for any $t>0$ and $x \in \R^{\mathbf{d}}$, 
\begin{equation} \label{eq.dualdef}
\E_{p}[w_t(x)] = \hat{Q}_x [\V(\hat{\W}_t; p) = 1].
\end{equation}

For the techniques imposed in this paper, we rely on the hypothesis that upon rescaling, the dual process can be approximated with high probability, by what we refer to as a \emph{pure branching} process, which we will denote by $(X_{t}, t\geq 0)$. For easy reference in future discussions, we give a precise definition which encodes the meaning of a pure branching process in this paper. 
\begin{definition}\label{d.defprocess}
The trajectory of $(X_{t}, t \geq 0$) is defined according to the following two mechanisms: 
\begin{itemize}
\item {\bf Spatial motion.} In between branch times, individuals evolve in space as independent copies of a Hunt process $(Y_t, t \geq 0)$  on $\R^{\mathbf{d}}$, with paths in the Skorokhod space $\mathbb{D}([0, \infty), \R^{\mathbf{d}})$. The law and expectation are denoted by $P^{Y}_x$ and $E^{Y}_x$ when $Y_0 = x$, and the map $x \mapsto P^{Y}_x$ is assumed to be measurable.
\item {\bf Branching.} With branch rate $\gamma > 0$, individuals branch into $N_0 \in \N$ individuals, independently of one another. After giving birth, an individual is removed from the population. At the branch time, if the location of the parent is $y$, then the offspring displacements have joint law $\mu_y$; thus, for a branching event with parent location $y$, the offspring location vector is $(y+\xi_1,\dots, y+\xi_{N_0})$, where $(\xi_1,\dots,\xi_{N_0})$ is sampled from $\mu_y$. The map $y \mapsto \mu_y$ is assumed to be measurable.
\end{itemize}
\end{definition}

If one ignores the spatial information, the above is simply a continuous-time branching process with $N_0$-ary branching at rate $\gamma$. Thus, it has a naturally defined (time-labelled) tree structure. This is discussed in detail in Section~\ref{s:vote}, where we describe voting algorithms acting on trees.

Equipped with this definition, we now describe the rescaled processes, and the sense in which the previously mentioned approximation holds. We will consider a family of stochastic spatial models $(w^\eps_t, t \geq 0)_{\eps>0}$ rescaled and tuned according to a parameter $\eps>0$. The dual processes are scaled accordingly, with all prior quantities denoted with an $\eps$-dependence. We will always consider a scaling regime in which the branch rate of the dual is $\gamma_\eps := \gamma \eps^{-2}$ for some fixed $\gamma > 0$. We highlight that the voting algorithm in general does not change with any rescaling of the model. 
 
The primary assumption is that, for each $\eps>0$, there exists a pure branching process $(X^{\eps}_{t}, t\geq 0)$ (in the sense of Definition \ref{d.defprocess}), and corresponding law $Q^{\eps}_{x}$ started at $x$, such that for any compact set $K\subseteq (0,\infty)\times \R^{\mathbf{d}}$,
\begin{equation}\tag{AD}\label{eq:approxdualdef}
\lim_{\epsilon \to 0} \sup_{(t,x) \in K}\left|\E^{\eps}_{p}[w^{\eps}_t(x)] -Q^{\eps}_x [\V(\W^{\eps}_t; p) = 1]\right|=0. 
\end{equation} 
If \eqref{eq:approxdualdef} holds, we will refer to $(\W^{\eps}_{t}, t\geq 0)$ as the \emph{approximate dual} historical process. Observe that \eqref{eq:approxdualdef} does not inherently \emph{require} the existence of a true moment dual. Nevertheless, in most cases, we establish \eqref{eq:approxdualdef} for stochastic spatial models which are equipped with a true moment dual $(\hat{X}^{\eps}_{t}, t\geq 0)$ and instead prove that \begin{equation} \label{eq:dualapproxdual}
\lim_{\epsilon \to 0} \sup_{(t,x) \in K} \left|Q^\eps_x [\V(\hat{\W}^{\eps}_t; p) = 1]-Q^\eps_x [\V(\W^{\eps}_t; p) = 1]\right|= 0, 
\end{equation}
from which \eqref{eq:approxdualdef} is immediate.

As noted after \eqref{eq.dualdef}, we work in regimes in which the dual process $(\hat{X}_t, t\geq 0)$ can be approximated by a pure branching process $(X_{t}, t\geq 0)$ with high probability. The definition of $(X_{t}, t\geq 0)$ is model specific, but typically involves removing collisions/interactions/coalescences from the dynamics of $(\hat{X}_t,t \geq 0)$. Proving \eqref{eq:approxdualdef} is thus a matter of showing that this perturbation has an asymptotically negligible effect.

In addition to \eqref{eq:approxdualdef}, we make two assumptions on the approximate dual process itself. To state these assumptions, we introduce some notation. Throughout the rest of the paper, we reserve the notation $(W^{\mathbf{d}}_{t}, t\geq 0)$ and $(W^{\mathbf{1}}_{t}, t\geq 0)$ to denote respectively $\mathbf{d}$- and  $1$-dimensional standard Brownian motion, with $P^{W^{\mathbf{d}}}_x$, $P^{W^{1}}_x$, $E^{W^{\mathbf{d}}}_x$, and $E^{W^{1}}_x$ denoting their laws and expectations respectively when $W^{\mathbf{d}}_0$ or $W^1_0$ is $x$. We will also use this notation to denote processes whose law, under some probability measure, e.g. $Q^\eps_x$, is a Brownian motion.

Our assumptions on the approximate dual allow us to compare this process to an $N_0$-ary branching Brownian motion. First, as noted above, we will always consider a scaling in which the branch rate of $X^\eps$ is given by $\gamma_\eps = \gamma \eps^{-2}$, with $\gamma > 0$. Let us denote by $P^{Y,\eps}_x$ and $\mu^\eps_y$ the law of the spatial motion and offspring displacement distribution associated to $X^\eps$. We assume that there exist constants $\bar{C}\in [1, \infty)$, $k >1$, and $\bar{c}, \eta, \eps_0 \in (0,1]$ such that the following hold: 
\begin{list}{ (\theascan)}
{
\usecounter{ascan}
\setlength{\topsep}{1.5ex plus 0.2ex minus 0.2ex}
\setlength{\labelwidth}{1.2cm}
\setlength{\leftmargin}{1.5cm}
\setlength{\labelsep}{0.3cm}
\setlength{\rightmargin}{0.5cm}
\setlength{\parsep}{0.5ex plus 0.2ex minus 0.1ex}
\setlength{\itemsep}{0ex plus 0.2ex}
}
\item \label{a.1} \textbf{Lineages converge to Brownian motion.} For every $x$, $(Y_t, t\geq0)$ started at $x$ can be coupled with a Brownian motion $(\bmd_t, t\geq 0)$ started from $x$, such that for all $\eps< \eps_0$,
\begin{equation*}
\sup_x \sup_{s \in (0, \eps^2 |\log \eps|^2]} P^{Y, \eps}_x [ |Y_s - \bmd_s| > \eps^{k+2}]\leq \bar{C}e^{-\bar{c} \eps^\eta}.
\end{equation*} 
\item \label{a.2} \textbf{Offspring dispersion concentration.} For $\eps<\eps_0$, we have 
\begin{equation*}
 \sup_y \mu^\eps_y[ \xi: \max_{1\leq i\leq N_0} |\xi_i| > \eps^{k+2} ] \leq \bar{C}e^{-\bar{c}\eps^\eta}.
\end{equation*}
\end{list}

Because $Y$ and $\xi$ are the building blocks of the spatial part of the dual process, random variables with the same distributions as $Y$ and $\xi$ will arise frequently under $Q^\eps_x$. Hence, in a slight abuse of notation, we will write the above estimates as
\begin{equation*}
\sup_x \sup_{s \in (0, \eps^2 |\log \eps|^2]} Q^{\eps}_x [ |Y_s - \bmd_s| > \eps^{k+2}]\leq \bar{C}e^{-\bar{c} \eps^\eta}.
\end{equation*}
and 
\begin{equation*}
 \sup_y Q^{\eps}_y[  \xi: \max_{1\leq i\leq N_0} |\xi_i| > \eps^{k+2} ] \leq \bar{C}e^{-\bar{c}\eps^\eta}.
\end{equation*}

\subsection{The $g$-function and the Voting Algorithm}\label{s.gint}
We next require hypotheses on the voting algorithm $\mathbb{V}(\W_{t}; p)$ which appears in \eqref{eq:approxdualdef}.  We reserve a precise discussion of the voting algorithm for Section \ref{s:vote}, however we begin with a heuristic description of the voting which allows us to precisely state the main hypotheses for our results.

The voting function relies on the initial distribution $p$, and the historical approximate dual process $(\W_t, t\geq 0)$, which is a pure branching process. As priorly mentioned, the approximate dual process is naturally associated to an $N_{0}$-ary tree. Based on input votes at the leaves, which are determined by $p$, the votes of the parents are determined by the votes of their children. By propagating backwards in time through the tree, this eventually assigns a vote for the root. As was introduced in \cite{DH2021}, the voting in any generation can be described according to a function $g:[0,1]^{N_0} \to[0,1]$, where the $N_{0}$ inputs correspond to information about the votes of $N_{0}$ children at each generation, and the output $g$ yields information about the vote of the parent. In particular, the voting algorithm propagates votes through the tree via a mapping $\Theta:\{0,1\}^{N_0} \to [0,1]$, where the inputs are the votes of a family of $N_0$ siblings and the output is the probability that their parent has vote $1$. The $g$-function, which is defined precisely in Section~\ref{s:gfunction} (see \eqref{e.gdef}), encodes the expected behaviour of $\Theta$ when evaluated on independent Bernoulli inputs. The assumptions we now state on $g$ are essentially the same as those in \cite{DH2021}. 

By definition, $g$ is a multivariate function. We will abuse notation and write a univariate function $g(p)$, understood to mean $g(p,\dots,p)$ for $p \in [0,1]$. An important assumption is the monotonicity property, which relates the multivariate and univariate $g$-functions:  
\begin{list}{(\thegoscan)}
{
\usecounter{goscan}
\setlength{\topsep}{1.5ex plus 0.2ex minus 0.2ex}
\setlength{\labelwidth}{1.2cm}
\setlength{\leftmargin}{1.5cm}
\setlength{\labelsep}{0.3cm}
\setlength{\rightmargin}{0.5cm}
\setlength{\parsep}{0.5ex plus 0.2ex minus 0.1ex}
\setlength{\itemsep}{0ex plus 0.2ex}
}
\item \label{g.0} For any $p_{i}\leq \bar{p}_{i}$, for $i\in [N_{0}]$, 
\begin{equation*}
    g(p_1,\dots p_{i}, \dots ,p_{N_0}) \leq g(p_1,\dots \bar{p}_{i}, \dots ,p_{N_0}).
\end{equation*}
\end{list}
Under the assumption that \eqref{g.0} holds, it easily follows that 
\begin{equation*}
 g(\min p_i) \leq g(p_1,\dots p_{i}, \dots ,p_{N_0}) \leq g(\max p_i), 
\end{equation*}
and consequently, we are able to present the remaining assumptions in terms of the univariate $g$-function.

For the remaining assumptions, we assume that $g \in C^2((0,1)) \cap C([0,1])$, with $g(0)=0$ and $g(1)=1$. We will extend $g$ to all of $\R$ by continuously extending it to be constant outside of $[0,1]$. We assume that $g$ satisfies the following:
\begin{list}{ (\thegscan)}
{
\usecounter{gscan}
\setlength{\topsep}{1.5ex plus 0.2ex minus 0.2ex}
\setlength{\labelwidth}{1.2cm}
\setlength{\leftmargin}{1.5cm}
\setlength{\labelsep}{0.3cm}
\setlength{\rightmargin}{0.5cm}
\setlength{\parsep}{0.5ex plus 0.2ex minus 0.1ex}
\setlength{\itemsep}{0ex plus 0.2ex}
}

\item \label{g.1}$g$ has fixed points $a,\mu,b \in [0,1]$ satisfying the following: $0 \leq a < \avg < b \leq 1$, $\avg$ is unstable, $a$ and $b$ are stable, and $b- \avg = \avg -a$. Moreover, these are the only fixed points of $g$ in $[a,b]$. If $a \neq 0$ (resp. $b \neq 1$), $g$ may have other fixed points in $[0,a)$ (resp. $(b,1]$).

\item \label{g.2} For any $\delta\in (0,\avg  - a$),  $g(\mu + \delta) - \mu = -g(\mu-\delta) - \mu$, i.e. $g$ is anti-symmetric about the fixed point $g(\mu) = \mu$.
\item \label{g.3} $g'(\cdot)>0$, $g'(\avg) > 1$ and $g'(a) = g'(b) < 1$.
\end{list}

\begin{remark}
    In \cite{DH2021}, there is an additional convexity condition on $g$, namely that $g''(r) > 0$ if $r \in (a,\avg)$ and $g''(r) <0$ if $r \in (\avg, b)$. Using this convexity, one can deduce the claim in \eqref{g.1} that the only fixed points of $g$ in $[a,b]$ are $a,b$, and $\mu.$ Our framework is slightly more general; we do not require the convexity condition but instead explicitly impose in  \eqref{g.1} that there are no additional fixed points of $g$. 
\end{remark}

\begin{remark}
    As consequence of \eqref{g.1}-\eqref{g.3}, we have 
 \begin{list}{ (\thegfscan)}
{
\usecounter{gfscan}
\setlength{\topsep}{1.5ex plus 0.2ex minus 0.2ex}
\setlength{\labelwidth}{1.2cm}
\setlength{\leftmargin}{1.5cm}
\setlength{\labelsep}{0.3cm}
\setlength{\rightmargin}{0.5cm}
\setlength{\parsep}{0.5ex plus 0.2ex minus 0.1ex}
\setlength{\itemsep}{0ex plus 0.2ex}
}
\item\label{g.5} There exists $c_0 \in (0,1-g'(a))$ and $\delta_*\in (0,1)$ so that
\begin{equation*}
\max_{r\in [b-\delta_{*}, b+\delta_{*}]}|g'(r)|<1-c_{0}\quad\text{and}\quad \max_{r\in [a-\delta_{*}, a+\delta_{*}]} |g'(r)|<1-c_{0}.
\end{equation*}
\end{list}
Indeed, we first remark that \eqref{g.3} implies the existence of $c_0\in (0,1-g'(a))$. Since $g\in C^{2}((0,1))$, and $g$ is extended to be continuous and constant outside of $(0,1)$, this (and the symmetry guaranteed by \eqref{g.2})  implies the existence of $\delta_{*}=\delta_{*}(c_{0}, \norm{g}_{C^{2}((0,1))})\in (0,1)$ as above. We include \eqref{g.5} as its own hypothesis, because we need it readily available in several proofs.
       \end{remark}

       \subsection{Motion By MCF and the Level-Set Method}\label{s.mcf}
      We give an overview of mean curvature flow (MCF) and generalized MCF. Readers who are experienced with these notions can skip this subsection. MCF is one of the most well-studied geometric flows, and it naturally arises in mathematical models for phase separation, including models of physical and biological systems. We begin with \emph{classical} MCF. We say that a collection of $(\mathbf{d}-1)$-dimensional hypersurfaces $\left\{\Ga_{t}, t\geq 0\right\}\subseteq \R^{\mathbf{d}}$ evolve according to MCF if for each $x\in \Ga_{t}$, the velocity of $\Ga_{t}$ at the point $x$ is given by $V(x)=-2^{-1}H n(x)$ where $H$ denotes the mean curvature of $\Ga_{t}$ at $x$, and $n(x)$ denotes the unit outward normal to $\Ga_{t}$. The additional scaling factor of $2^{-1}$ changes the speed of the flow such that we can work with standard Brownian motions as opposed to Brownian motions run at speed $2$. Our notion of ``outwards'' is based on the premise that as a $(\mathbf{d}-1)$-dimensional hypersurface, we associate an open set $\Theta_{t}\subseteq \RR^{\mathbf{d}}$ such that $\Ga_{t}=\partial \Theta_{t}$, and $n$ points towards $\Theta_{t}^{c}$. 
      
The formation of singularities for surfaces evolving by MCF (i.e. when the velocity vector is ill-defined) is a ubiquitous phenomenon and a topic of major ongoing interest. Surfaces evolving by MCF may develop singularities, even when started from smooth initial data. Let $\mathrm{T}$ denote the first time when a singularity forms. Without a weaker notion of MCF, results on models exhibiting phase separation evolving by MCF can only hold up to time $\mathrm{T}$. 

There have been two main approaches to extending the notion of MCF past the formation of singularities. One such approach is a weak geometric formulation introduced by Brakke \cite{Br}. The other approach is a so-called PDE approach, and this is what we choose to work with in the present paper. 

The framework which we consider is a notion of generalized MCF via the level-set method.  The level-set method was pioneered by Osher and Sethian \cite{OS}, and for the case of MCF, was further studied by Evans and Spruck \cite{ES1}, and independently by Chen, Giga, and Goto \cite{CGG}. Broadly speaking, the idea is to connect the evolution of interfaces with the level sets of a given function $\psi: [0, \infty)\times \R^{\mathbf{d}}\rightarrow \RR$, where the initial interface $(\mathbf{d}-1)$-dimensional interface $\Ga_{0}\subseteq \R^{\mathbf{d}}$ is given by $\Ga_{0}=\left\{\psi(0, \cdot)=0\right\}$.

In order to identify $\psi$ for a given velocity vector field $V$, as in the method of characteristics for first-order PDEs, let $x(t)$ denote the curve, started from $x_{0}\in\Ga_{0}$, such that $\psi(t,x(t))=0$. Upon differentiating this relation with respect to $t$, and by using the velocity field $V$, we have 
\begin{equation*}
\partial_{t}\psi+\dot{x}\cdot D\psi=\partial_{t}\psi+V\cdot D\psi=0. 
\end{equation*}
This is equivalent to the PDE
\begin{equation}\label{e.normvel}
\partial_{t}\psi + V_{n}|D\psi|=0, 
\end{equation}
where the function $V_{n}:=V\cdot n$ is the normal velocity of the surface, and in the case when the surface is the level set of some function $\psi(t, \cdot)$, we may identify $n= \frac{D\psi}{|D\psi|}$, assuming that $|D\psi|\neq 0$ on the level-set $\left\{\psi(t, \cdot)=0\right\}$.

Using the normal velocity as in the setting of classical MCF, \eqref{e.normvel} becomes the equation
\begin{equation*}
\partial_{t}\psi-\frac{1}{2}\text{div}\left(\frac{D\psi}{|D\psi|}\right)|D\psi|=\partial_{t}\psi-\frac{1}{2}\tr\left[\left(\id-\frac{D\psi \otimes D\psi}{|D\psi|^{2}}\right)D^{2}\psi\right]=0. 
\end{equation*}
The above equation is a degenerate parabolic PDE, for which the theory of \emph{viscosity solutions} provides a well-posed theory of weak solutions (see \cite{users}). Of note, for continuous initial data, a unique viscosity solution exists globally in time (the theory can also be extended to discontinuous initial data, see for example \cite{BS} for a discussion).  

We now describe the level-set method precisely. Let $\Theta_{0}\subseteq \R^{\mathbf{d}}$ be a non-empty open set, and let $d(0,\cdot): \R^{d}\rightarrow \R$ denote the signed distance function to $\Ga_{0}:=\partial\Theta_{0}$, taken to be negative in the set $\Theta_{0}$ and positive on the set $\overline{\Theta_{0}}^{c}$. This implies that we can also characterize 
\begin{equation}\label{e.ic}
\Ga_{0}=\left\{d(0, \cdot)=0\right\}=\partial\left\{d(0, \cdot)<0\right\}.
\end{equation}
We consider the unique viscosity solution $u: [0, \infty)\times \R^{\mathbf{d}}\rightarrow \R$ solving
\begin{equation} \label{eq:levelset1}
\begin{cases}  \partial_t u  - \frac{1}{2} \tr \left[\left( \id - \frac{Du\otimes Du}{|Du|^2} \right) D^{2}u\right]=0 &\text{in $(0,\infty) \times \R^\mathbf{d}$},\\
u(0,x) = d(0,x), &\text{in $\R^\mathbf{d}$}.
\end{cases}
\end{equation}
Since the unique viscosity solution of \eqref{eq:levelset1} exists globally in time, we may now \emph{define}
\begin{equation}\label{e.gammatdef}
\Ga_{t}:=\left\{x: u(t,x)=0\right\}
\end{equation}
for all $t \geq 0$. This provides a generalized notion of what it means for sets $\left\{\Ga_t, t \geq 0\right\}$ to evolve by MCF, and throughout the paper, we will say that $\left\{\Gamma_t, t \geq 0\right\}$ evolves by \emph{generalized MCF} if it is defined according to \eqref{eq:levelset1}-\eqref{e.gammatdef}.

It follows from the work of Evans and Spruck \cite{ES1, ES2} that if $\left\{\Ga_t', 0\leq t\leq T^{*}\right\}$ is a family of surfaces evolving according to \emph{classical} MCF up until some time $T^{*}>0$, then $\Ga_{t} = \Ga_t'$ for $t \in [0,T^{*}]$. The choice of initial condition is simply for convenience; as is shown in \cite{BSS}, the sets $\left\{\Ga_{t}, t\geq 0\right\}$ are uniquely defined, even for different initial conditions which share the same $0$-level set. Finally, as we will see, this level-set formulation has yet another equivalent formulation, known as generalized flows (see Definition \ref{d.genflow}), which was introduced in \cite{BS}, and this equivalence is crucial for our approach. 

While generalized MCF offers the advantage that one can have a global in time interpretation for MCF of an interface (i.e. level-set), the method does not guarantee that the interface is necessarily a $(\mathbf{d}-1)$-dimensional hypersurface; indeed, there are pathological scenarios where, past the time that singularities form in the classical flow, the interface $\Ga_t$ fails to be the boundary of an open set, and instead develops interior (or ``fattens''). In the analysis of phase separation, with this notion of generalized MCF, one can discuss phase separation on the \emph{complement} of $\Ga_t$; however, if $\Ga_t$ develops interior, there is a fat set where the limiting behaviour of the model is unknown. We emphasize that formation of singularities does not imply the fattening of $\Ga_t$. For certain types of singularities, for example ``neck'' singularities, $\Ga_t$ will not develop interior \cite{CHH,CHHW}. In such cases, using generalized rather than classical MCF offers a great improvement in the analysis of phase separation.

\subsection{Statement of the Main Result} \label{s:abstract}

We have now introduced all relevant hypotheses to state our main result. We fix two constants $a<b$, with $a,b\in [0,1]$, and set $\mu:=2^{-1}(b-a)$.

\begin{definition}\label{d.defint}
Let $p: \R^{\mathbf{d}}\rightarrow [a,b]$. We say that $p$ defines the interface $\Ga_{0}\subset \R^{\mathbf{d}}$ if $\Ga_{0}=\partial \Theta_{0}$ where $\Theta_{0}, \overline{\Theta_{0}}^{c}\subseteq \R^{\mathbf{d}}$ are  non-empty open sets, with 
\begin{equation*}
\Theta_{0}=\left\{p(\cdot)<\mu\right\}, \quad\text{and}\quad \overline{\Theta_{0}}^{c}=\left\{p(\cdot)>\mu\right\}. 
\end{equation*}
\end{definition}
Observe that by definition, $\Ga_{0}$ is the boundary of an open set, and thus contains certain minimal topological properties (i.e. it is closed and nowhere dense). 
\begin{definition}\label{d.lumcf}
Let $\Ga_{0}$ be as above in Definition \ref{d.defint}. We say that a sequence of functions $\left\{u^{\eps}\right\}_{\eps>0}$ defined on $[0, \infty)\times \R^{\mathbf{d}}$ converges locally uniformly to $(a,b)$-generalized mean-curvature flow (MCF) started from $\Ga_{0}$ if $u^{\eps}(0,\cdot): \R^{\mathbf{d}}\rightarrow [a,b]$ defines $\Ga_{0}$ in the sense of Definition \ref{d.defint} above, and 
\begin{equation*}
 \lim_{\eps\to 0} u^{\eps}(t,x)=\begin{cases} b&\text{locally uniformly in $\bigcup_{t\in (0, \infty)} \left\{t\right\}\times \left\{u(t,\cdot)>0\right\}$},\\
 a&\text{locally uniformly in $\bigcup_{t\in (0,\infty)} \left\{t\right\}\times \left\{u(t, \cdot)<0\right\}$},
 \end{cases}
 \end{equation*}
where $u$ is the unique viscosity solution of \eqref{eq:levelset1}. 
\end{definition}
In particular, we highlight that ``away from the interface,'' $\left\{u^{\eps}\right\}_{\eps>0}$ converges locally uniformly in space-time to the values of $a$ or $b$. We also remark that the ``$(a,b)$'' from ``convergence to $(a,b)$-generalized mean-curvature flow'' relates only to the limiting values of $u^{\eps}(t,x)$ off of the interface, but that the sub- and super-level sets of $u$ where these limiting values are attained do not depend on $a$ and $b$.

We now state our main result: 
\begin{theorem}\label{t.realmain}
Suppose that $(w_t^\eps, t\geq 0)_{\eps > 0}$ is a family of stochastic spatial models such that the following holds:
\begin{itemize}
    \item There exists a family of approximate dual processes  $(\W^{\eps}_t, t\geq 0)_{\eps >0}$ such that \eqref{eq:approxdualdef} holds. 
    \item $(\W^{\eps}_t, t\geq0)_{\eps>0}$ are pure branching processes as in Definition~\ref{d.defprocess}, with branch rate $\gamma_\eps = \gamma \eps^{-2}$, and satisfy \eqref{a.1} and \eqref{a.2}.  \item The voting algorithm $\V$ is associated to a $g$-function satisfying \eqref{g.0}-\eqref{g.5}. 
\end{itemize} 
Then for any $p$ defining $\Ga_{0}\subseteq \R^{\mathbf{d}}$, as $\eps\to 0$, $\E^{\eps}_{p} [w^{\eps}_t]$ converges locally uniformly to $(a,b)$-generalized MCF started from $\Ga_{0}$.
\end{theorem}

\begin{remark} 
We highlight that the only assumption on $p$ is that it defines the interface $\Ga_{0}$ in the sense of Definition \ref{d.defint}.  In \cite{EFP2017, DH2021}, the authors assume that $\Ga_{0}=\left\{ p(\cdot)=\sfrac{1}{2}\right\}$, and they require the following:
\begin{list}{ (\thecscan)}
{
\usecounter{cscan}
\setlength{\topsep}{1.5ex plus 0.2ex minus 0.2ex}
\setlength{\labelwidth}{1.2cm}
\setlength{\leftmargin}{1.5cm}
\setlength{\labelsep}{0.3cm}
\setlength{\rightmargin}{0.5cm}
\setlength{\parsep}{0.5ex plus 0.2ex minus 0.1ex}
\setlength{\itemsep}{0ex plus 0.2ex}
}
\item \label{C1}$\Ga_{0}$ is $C^{\al}$ for some $\al>3$. 
\item\label{C2} $\Theta_{0}=\left\{p<\frac{1}{2}\right\}$ and $\overline{\Theta_{0}}^{c}=\left\{p>\frac{1}{2}\right\}$.
\item \label{C3} There exists $r,\eta\in (0,1)$ such that for all $x\in \R^{\mathbf{d}}$, $\left|p(x)-\sfrac{1}{2}\right|>\eta\dist(x, \Ga_{0} \wedge r)$. 
\end{list}
The assumption \eqref{C1} guarantees that classical MCF started from $\Gamma_0$ exists (and remains regular) until some $T^* >0$, whereas \eqref{C3} imposes that $p$ has some minimum ``slope'' near the initial interface. Both of these are used in the proofs in \cite{EFP2017, DH2021}. We are able to relax the assumptions \eqref{C1} and \eqref{C3} due to the fact that we work with generalized front propagation, and we note that \eqref{C2}, which agrees with Definition \ref{d.defint}, is in a sense optimal; any relaxation of this condition may cause the initial interface to change. The assumptions \eqref{C1}-\eqref{C3} originate from the work of Soner \cite{S2} in the analysis of a rescaled Ginzburg-Landau equation, and have also appear in other related works which yield convergence to generalized MCF (see \cite{KS1} for a discussion).

As we have discussed, the generalized MCF $\Gamma_t$ may develop interior. Theorem~\ref{t.realmain} does not address what happens to $\EE^\eps_p[w^{\eps}_{t}]$ in these regions. Still, there are types of singularities for which $\Gamma_t$ does not fatten, and in these cases Theorem~\ref{t.realmain} gives a complete convergence theorem. 

\end{remark}

\subsection{Strategy of the Proof: the Abstract Approach of Barles and Souganidis using the Approximate Dual}
\label{s.BS}

The approach of our proof relies on using an abstract approach to front propagation, introduced by Barles and Souganidis \cite{BS}, which we now describe. We begin with some basic definitions. 

\begin{definition} \label{def:hrlimit}
Given a collection of real-valued functions $\left\{u^{\eps}\right\}_{\eps>0}$, we define the ``half-relaxed'' limits
\begin{equation*}
\limsups_{\eps\rightarrow 0}u^{\eps}(t,x):=\sup\left\{\limsup_{\eps\rightarrow 0} u^{\eps}({t_{\eps}}, x_{\eps}): (t_{\eps}, x_{\eps})\rightarrow (t,x)\right\}
\end{equation*}

\begin{equation*}
\liminfs_{\eps\rightarrow 0}u^{\eps}(t,x):=\inf\left\{\liminf_{\eps\rightarrow 0} u^{\eps}({t_{\eps}}, x_{\eps}): (t_{\eps}, x_{\eps})\rightarrow (t,x)\right\}
\end{equation*}

\end{definition}

\begin{definition}
For $U\subseteq \R^{\mathbf{d}}$ and $f: U\to\R$, we define the upper (respectively lower) semicontinuous envelope by
\begin{equation*}
f^{*}(y):=\limsup_{z\to y}f(z)\quad\text{and}\quad f_{*}(y)=\liminf_{z\to y}f(z).
\end{equation*}
\end{definition}

The abstract approach of Barles and Souganidis \cite{BS} holds true for more general limiting flows which are characterized by a locally bounded function $F : \cS^{\mathbf{d} \times \mathbf{d}} \times (\R^\mathbf{d} \backslash \{0\})\rightarrow \R$, where $\cS^{\mathbf{d} \times \mathbf{d}}$ is the space of symmetric $\mathbf{d} \times \mathbf{d}$ matrices; the flow $F$ must satisfy certain properties (see Section \ref{s.app} and \cite{BS} for precise conditions), but in this discussion, we fix the function $F$ to be MCF, i.e.
\begin{equation}\label{e.Fdef}
F(M,p):=-\frac{1}{2}\tr\left[\left(\id-\frac{p\otimes p}{|p|^{2}}\right)M\right].
\end{equation}
We note that while $F(M, p)$ is not defined when $p=0$, we can consider the upper semicontinuous envelope of $F$, given by 
\begin{equation*}
F^{*}(M,p):=\begin{cases}-\frac{1}{2}\tr\left[\left(\id-\frac{p\otimes p}{|p|^{2}}\right)M\right]&\text{if $|p|\neq 0$},\\
-\frac{1}{2}\left[\tr(M)+\la_{\max}(M)\right]&\text{if $p=0$},
\end{cases}
\end{equation*}
where $\la_{\max}(M)$ is the largest eigenvalue of $M$. Similarly, we can consider the lower semicontinuous envelope
\begin{equation}\label{e.flsc}
F_{*}(M,p):=\begin{cases}-\frac{1}{2}\tr\left[\left(\id-\frac{p\otimes p}{|p|^{2}}\right)M\right]&\text{if $|p|\neq 0$},\\
-\frac{1}{2}\left[\tr(M)+\la_{\min}(M)\right]&\text{if $p=0$,}
\end{cases}
\end{equation}
where $\la_{\min}(M)$ is the smallest eigenvalue of $M$.

In \cite{BS}, Barles and Souganidis introduce four conditions under which they prove asymptotic phase separation for rescaled solutions of reaction-diffusion equations. As priorly mentioned, we will state these conditions on the following function which is defined according to the approximate dual process and voting algorithm.
For $p: \R^{\mathbf{d}}\rightarrow [a,b]$ as above, we define
\begin{equation}\label{e.uepdef}
u^{\ep}(t,x; p):=Q^{\ep}_{x}[\V(\W^{\eps}_{t}; p)=1],
\end{equation}
and recall from \eqref{eq:approxdualdef} that $u^\eps(t,x; p) = \E^{\eps}_{p}[w_t^\eps(x)]+o(1)$. While for certain models, $x$ may necessarily belong to the rescaled integer lattice, we may define the function $u^{\eps}(t, \cdot; p)$ on all of $\R^{\mathbf{d}}$ by piecewise constant extension. Thus, without loss of generality we work on spatial domain $\R^{\mathbf{d}}$.

We use the symbol $\heartsuit$ to denote parameters which are intrinsic to the stochastic spatial model and dual, such as dimension $\mathbf{d}$, number of children in branching events, branch rate, etc. These dependencies are made explicit later on. In addition, we recall at this time that for $k\in \mathbb{N}$, the $C^k$-norm on a domain $E\subseteq \R^{\mathbf{d}}$ is defined by
\begin{equation*}
\norm{f}_{C^{k}(E)}=\norm{f}_{L^\infty(E)}+\sum_{j=1}^{k} \norm{D^{j}f}_{L^\infty(E)}.
\end{equation*}

We now state the four following conditions which appear in \cite{BS}.
\begin{list}{ (\thejscan)}
{
\usecounter{jscan}
\setlength{\topsep}{1.5ex plus 0.2ex minus 0.2ex}
\setlength{\labelwidth}{1.2cm}
\setlength{\leftmargin}{1.5cm}
\setlength{\labelsep}{0.3cm}
\setlength{\rightmargin}{0.5cm}
\setlength{\parsep}{0.5ex plus 0.2ex minus 0.1ex}
\setlength{\itemsep}{0ex plus 0.2ex}
} 
\item \label{j.1} \textbf{The Semigroup Property:} For all $\eps>0$ and all $h>0$ and $(t,x)\in (0, \infty)\times\R^{\mathbf{d}},$
\begin{equation*}
u^{\ve}(t+h,x; p)=u^{\ve}(h, x; u^{\ve}(t,\cdot; p)).
\end{equation*} 

\item \label{j.2} \textbf{Monotonicity:} If $p(\cdot) \leq \hat{p}(\cdot)$, then for all $\eps>0$ and all $(t,x)\in (0, \infty)\times\R^{\mathbf{d}}$,
\begin{equation*}
u^{\ep}(t,x; p)\leq u^{\ve}(t,x; \hat{p}).
\end{equation*}

\item \label{j.3} \textbf{Existence of Equilibria:} There exists $a,b\in \R$ with $a<b$ such that for all $\eps>0$ and all  $(t,x)\in (0, \infty)\times\R^{\mathbf{d}}$, 
   \begin{equation*}
       u^{\eps}(t,x; a)\equiv a\quad\text{and}\quad u^{\eps}(t,x;b)\equiv b. 
   \end{equation*}
\item \label{j.4} \textbf{Flow Consistency:} 
\begin{enumerate}[(i)]
 \item For every $x_{0}\in \R^\mathbf{d}$, $r\in (0,1)$, and every smooth function $\phi : \R^\mathbf{d} \to \R$ satisfying $\{\phi\geq 0 \} \subseteq B(x_0,r)$ with $|D\phi(\cdot)| \neq 0$ on $\{\phi = 0\}$, there exists $\bar{\delta}\in (0,1)$ and $h_0=h_{0}(\norm{\phi}_{C^{4}(\overline{B(x_{0},r)})}, \heartsuit)\in (0,1)$ such that for every $\al\in (0,1)$, the following holds: for 
\begin{equation*}
L^{+}_{h,\al}:=\left\{x: \phi(x) - h\left(F^{*}(D^2\phi(x), D\phi(x)) + \alpha \right) >0\right\}\subseteq \R^{\mathbf{d}},
\end{equation*}
we have that for all $\delta \in (0, \bar{\delta}]$, for all $h\in (0, h_{0}]$ and for all $x\in L^{+}_{h, \al}\cap B(x_{0}, r)$, 
\begin{equation}
\liminfs_{\eps \to 0} \,\,u^{\eps}(t,x; p^{-}(\phi, \delta) )= b,
\end{equation}
where $p^-(\phi,\delta) :=  (b - \delta) \indc_{\{\phi \geq 0\}} + a \indc_{\{\phi< 0\}}$.

\item For every $x_0 \in \R^\mathbf{d}$, $r\in (0,1)$, and every smooth test function $\phi : \R^\mathbf{d}\to \R$ satisfying $\{\phi \leq 0 \} \subseteq B(x_0,r)$ with $|D\phi(\cdot)| \neq 0$ on $\{\phi= 0\}$, there exists $\bar{\delta}\in (0,1)$ and $h_0=h_{0}( \norm{\phi}_{C^{4}(\overline{B(x_{0},r))}}, \heartsuit)\in (0,1)$ such that for every $\alpha\in (0,1)$, the following holds: for 
\begin{equation*}
L^{-}_{h, \al}:=\left\{x: \phi(x) - h\left[ F_{*}(D^2\phi(x), D\phi(x)) - \alpha \right] <0\right\}\subseteq \R^{\mathbf{d}},
\end{equation*}
we have that for all $\delta \in (0, \bar{\delta}]$, for all $h\in (0,h_{0}]$, and for all $x\in L^{-}_{h,\al}\cap B(x_{0}, r)$, 
\begin{equation}
\limsups_{\eps \to 0} \,\,u^{\eps}(t,x; p^{+}(\phi, \delta)) = a,
\end{equation}
where $p^+(\phi,\delta) =  (a + \delta) \indc_{\{\phi \leq 0\}} + b \indc_{\{\phi> 0\}}$.
\end{enumerate}
\end{list}
The first result we state is a simple consequence of \cite[Theorem 3.1]{BS} and \eqref{eq:approxdualdef}.

\begin{theorem}\label{t.generaldual}
Assume \eqref{j.1}-\eqref{j.4} hold with $F$ as in \eqref{e.Fdef}, and \eqref{eq:approxdualdef} holds. Then for any $p : \R^{\mathbf{d}} \to [a,b]$ defining an interface $\Ga_{0}\subseteq \R^{\mathbf{d}}$, as $\eps\to 0$, $\E^{\eps}_{p} [w_t^\eps]$ converges locally uniformly to $(a,b)$-generalized MCF started from $\Ga_0$.\end{theorem}

Theorem~\ref{t.generaldual} follows easily from \cite[Theorem~3.1]{BS} and \eqref{eq:approxdualdef} (see the Appendix (Section \ref{s.app})). We give a short proof of Theorem \ref{t.generaldual} in the Appendix (Section \ref{s.app}).

\begin{remark}\label{r.j3} We state \eqref{j.3} in a simplified manner (compared to how it appears in \cite{BS}) which is more convenient for most of the models we consider. We also mention that the dependencies in \eqref{j.4} differ slightly from how they are stated in \cite{BS}. This does not impact the proof in \cite{BS}, but it is a bit easier to digest, and such dependencies hold in the present setting. 
\end{remark}

In light of Theorem \ref{t.generaldual}, our main approach to proving Theorem \ref{t.realmain} is to show that the hypotheses of this theorem imply that \eqref{j.1}-\eqref{j.4} hold. For this, we have a brief discussion of the challenges in verifying \eqref{j.1}-\eqref{j.4}, which motivate the hypotheses of Theorem \ref{t.realmain}. 

Conditions \eqref{j.2} and \eqref{j.3} are easily verified. As we will show, \eqref{j.1} will hold true for all dual processes which have a pure branching structure (i.e. \eqref{eq:approxdualdef}). The greatest challenge in all models will be to verify \eqref{j.4}. The meaning of \eqref{j.4} is as follows: one checks whether the level sets of smooth, monotone approximations of $\Ga_{t}$ (as in \eqref{e.gammatdef}) evolve according to the ``predicted'' limiting front propagation (in this case, MCF). If $\left\{\phi=0\right\}$ is smooth and $|D\phi| \neq 0$ on $\left\{\phi=0\right\}$, then in a short time period $h\in (0, h_{0}]$, one expects that the normal velocity has caused the hypersurface to evolve by a displacement of $F(D^{2}\phi, D\phi)h$. If this is true for all smooth hypersurfaces, one obtains a comparison between the level sets as in \eqref{e.gammatdef} and the ``predicted interface'' (i.e. an interface evolving classically according to the same normal velocity field). 

In Theorem \ref{t.realmain}, the fact that \eqref{eq:approxdualdef} holds with an approximate dual that is a pure branching process is important for two reasons. The first is that, in order to have the semigroup property \eqref{j.1}, we require that $(\X^{\eps}_{t}, t\geq 0)_{\eps>0}$ is a pure branching process (see Section \ref{s.J1} for a more detailed discussion). The second reason is related to the proof of convergence to MCF; the proof technique from \cite{EFP2017, DH2021}, which we use here to verify \eqref{j.4}, requires the branching property for the dual. Without the branching property, a key ingredient in the verification of \eqref{j.4}, Lemma~\ref{lemma:gfunMarkov}, fails.

We emphasize that Theorem \ref{t.generaldual} is stated for the flow $F$ as in \eqref{e.Fdef} for convenience and consistency with the rest of this paper. Indeed, the proof given in Section \ref{s.app1} holds for very general flows $F$, satisfying some mild hypotheses (see (i)-(iii) in that section). For other spatial stochastic models satisfying \eqref{eq:approxdualdef}, hypotheses \eqref{j.1}-\eqref{j.3} are nearly automatic, and the greatest challenge in applying Theorem \ref{t.generaldual} to obtain a convergence result will be in verifying \eqref{j.4}. In the case when $F$ is given by \eqref{e.Fdef} (i.e. MCF), while the verification of \eqref{j.4} still requires substantial work, it is streamlined thanks to some arguments used in the short-time convergence results (for smooth initial data) established in \cite{EFP2017} and \cite{DH2021}. 

\section{A Precise Framework for the Voting Algorithm and the Verification of (J2)-(J3)}\label{s:vote}
In this section, we give a precise formulation of the voting algorithm $\V(\cdot,p)$. We conclude the section by verifying that \eqref{j.2}-\eqref{j.3} hold under the assumptions of \eqref{eq:approxdualdef} and hypotheses $\eqref{g.0}-\eqref{g.5}.$

We begin with a brief discussion of basic tree notation we will use throughout this section. For a $N_{0}$-ary tree $\mathcal{T}$, the vertices are labeled according to Ulam-Harris notation. In particular, we let $\emptyset$ denote the root of the tree, and vertices of the tree are associated to a multi-index $\al$ belonging to $\mathcal{U}:=\cup_{n=0}^{\infty} \left\{1, \ldots, N_{0}\right\}^{n}$. For $\alpha = (\alpha_1,\dots,\alpha_n) \in \cU$, the parent of $\alpha$ is the individual with multi-index $(\alpha_1,\dots,\alpha_{n-1})$. For $\alpha,\beta \in \mathcal{U}$, we write $\alpha \vee \beta$ for their concatenation, and we say that $\beta$ is an ancestor of $\alpha$ if and only if $\alpha = \beta \vee \gamma$ for some $\gamma \in \cU$ with $\gamma \neq \emptyset$. Ancestry induces a partial order on $\cU$. Equipped with this basic notation, we can now describe the voting algorithm in a precise mathematical framework. 

\subsection{The Voting Algorithm on Trees}\label{ss.voted}
Let $\mathcal{T}$ be a finite $N_0$-ary tree, and let $\mathcal{L}(\mathcal{T})$ to be the set of leaves of $\mathcal{T}$. We denote by $v(L(\mathcal{T})) = \{v_\alpha : \alpha \in L(\mathcal{T})\}$ an assignment of $\{0,1\}$-valued votes to the leaves. We consider an algorithm $V(\mathcal{T},v(L(\mathcal{T})))$ which assigns a vote of $0$ or $1$ to all vertices of $\mathcal{T}$, with input data given by the (deterministic) vote inputs $v(L(\mathcal{T}))$ at the leaves. The algorithm may be deterministic or random, but since the deterministic algorithm (for the most part) fits into the framework of the random algorithm, we begin with a discussion of the random setting. For $\alpha \in \mathcal{T}$, we let $V_{\al}$ denote the vote assigned to vertex $\alpha$ by the algorithm $V(\mathcal{T},v(L(\mathcal{T})))$.
 
 For each $\alpha \in L(\mathcal{T})$, we set $V_\alpha = v_\alpha\in v(L(\mathcal{T}))$. In order to describe how the non-leaf vertices vote, we fix a function $\Theta: \{0,1\}^{N_0} \to [0,1]$; $\Theta$ is directly connected to the $g$-function, and this connection is described in detail in Section \ref{s:gfunction}. For each non-leaf vertex $\al\in \mathcal{T}$, the vote $V_\al$ is distributed according to a Bernoulli random variable whose parameter, conditional on the votes of its children, $V_{\alpha \vee 1}, \ldots, V_{\al\vee N_{0}}$, is $\Theta(V_{\alpha \vee 1}, \ldots, V_{\al\vee N_{0}})$. To encode this, we introduce the notation 
\[\Theta_\al := \Theta(V_{\alpha \vee 1}, \ldots, V_{\al\vee N_{0}})\]
to denote the parameter attached to vertex $\al$. Hence, $V_\al$ is Bernoulli($\Theta_{\al})$ and  
\begin{equation} \label{e.thetarel}
P[V_{\al}=1\mid V_{\alpha \vee 1}, \ldots, V_{\al\vee N_{0}}]= E[V_{\al}\mid V_{\alpha \vee 1}, \ldots, V_{\al\vee N_{0}}]=\Theta(V_{\alpha \vee 1}, \ldots, V_{\al\vee N_{0}}).
\end{equation}
The sole hypothesis we impose on $\Theta$ is that
\begin{equation}\label{e.theta}
\text{$\Theta$ is nondecreasing}.
\end{equation}
(The above is implicit in the assumption \eqref{g.0} we made on the $g$-function in Section~\ref{s.gint}.) The tree structure allows us to proceed in a directed fashion, from the leaves towards the root, since $\Theta_\alpha$ is measurable with respect to the votes of the vertices on the sub-tree descending from $\alpha$. We assign a vote to each vertex starting from the deterministic leaf inputs by applying the Bernoulli distribution described above at each vertex. In particular we can compute $V_\emptyset = V(\emptyset; \mathcal{T},v(L(\mathcal{T})))$, which is understood as the vote of the root $\emptyset$. More generally, for a subset $\Gamma$ of the vertices, we write $V(\Gamma; \mathcal{T}, v(L(\mathcal{T}))) := \{V_\alpha : \alpha \in \Gamma\}$ to denote votes of individuals with indices in $\Gamma$.

The discussion above is informal, but it is straightforward to construct a probability space, or enrich an existing probability space, to support the evaluation of the random algorithm; one can introduce a collection of IID Uniform$(0,1)$ random variables $(U_\alpha : \alpha \in \mathcal{T})$ and use them to iteratively generate the Bernoulli votes at each step in the algorithm using a standard argument. In fact, conditional on $(U_\alpha : \alpha \in \mathcal{T})$, we may define the algorithm deterministically, i.e. $V_\alpha = \indc_{\{\Theta(V_{\alpha \vee 1},\dots,V_{\alpha \vee N_0})\geq U_\alpha\}}$. In the sequel, the tree $\mathcal{T}$ may be random, but this can be handled in a universal way by taking the index set of the uniform random variables to be $\mathcal{U}$.

We can obtain the smaller class of strictly deterministic voting algorithms as a subclass of the algorithms defined here by restricting the range of $\Theta$ to $\{0,1\}$. In this case, given the votes of its children $V_{\alpha \vee 1},\dots, V_{\alpha \vee N_0}$, the vote $V_\alpha$ is simply {\it equal} to $\Theta(V_{\alpha \vee 1},\dots, V_{\alpha \vee N_0})$.

\subsection{Voting algorithm on the tree induced by the dual process}\label{ss.vt}
We now relate the voting algorithm described above to the spatial branching processes introduced earlier.

We begin with a discussion the tree structure which is associated to the pure branching process $(\W_{t}, t\geq 0)$ which arises from \eqref{eq:approxdualdef}. We consider finite sub-trees of $\mathcal{U}$ with ``full families'' in the sense that, if the tree contains one child of a given vertex, then it contains all the children of that vertex. For such a $N_0$-ary tree, denoted by $\mathcal{T}$, we introduce a time-labelled version by associating to each $\alpha \in \mathcal{T}$ a label $t_\alpha>0$, such that the label of each vertex is strictly larger than the label of its parent. The label $t_\alpha$ is understood as the death time of $\alpha$, i.e. $\alpha \sim s$ if and only if $t_\beta \leq s < t_\alpha$, where $\beta$ is the parent of $\alpha$. Given $t>0$, the historical process run until time $t$, i.e. $(\W_{s}, 0\leq s\leq t)$, traces out a time-labelled tree which records its genealogy and associated branch time of each branching event. The time label of each internal vertex is given by its branch time; for each leaf $\alpha$, we assign the label $t_\alpha = t$. We call the resulting time-labelled tree $\mathcal{T}(\W_{t})$. We observe that $L(\mathcal{T}(\W_t)) = N(t)$, where $N(t)\subset \mathcal{U}$ is the set of individuals alive at time $t$. We write $\alpha \sim s$ if $\alpha \in N(s)$. The state $X_s$ is an element of $(\R^\mathbf{d})^{N(s)}$.

For the historical process $(W_{t}, t\geq 0)$, conditional on $N(t)$, $\W_t$ belongs to $\mathbb{D}([0,t], \R^{\mathbf{d}})^{N(t)}$. For $\alpha \in N(t)$ and $s \in [0,t]$, $\W_s(\alpha)$ is the location at time $s$ of whichever ancestor of $\alpha$ was alive at time $s$. We can and will additionally assume that branch times along each lineage are encoded in $\W_t$, but we do not explicitly write this.

We now describe the voting algorithm associated to $\mathcal{T}(\X_{t})$. For the applications we have in mind, we impose that the input votes are \emph{random}, depending on the spatial information of $X_t$. Let $p : \R^{\mathbf{d}} \to [a,b]$ be measurable, and for ease of notation, let us define $L_{t}:=L(\mathcal{T}(\X_{t}))$. 
We consider the algorithm applied to the tree $\mathcal{T}(\X_t)$, with input votes as follows: conditional on $X_t$, for each $\alpha \in L_t$, the leaf vote $v_\alpha$ is an independently sampled Bernoulli$(p(X_t(\alpha)))$ random variable. We write $\mathcal{V}_p(L_t)$ to denote the random collection of votes assigned to the leaves in this way when the branching process $X_t$ is run until time $t$. In the notation we have now introduced, we have for $E_{x}$ the expectation operator corresponding to $Q_{x}$,
\begin{equation}\label{e:votetreeequivalence}
E_{x}[\V(\W_t;p)]=Q_x\left[\V(\W_t;p)=1\right] = Q_x\left[V(\emptyset; \mathcal{T}(\W_t), \mathcal{V}_p(L_t))=1\right].
\end{equation}
In the above, the expectation averages over the randomness of $\X_t$ and, in the case of a random algorithm, the randomness involved in the evaluation of the algorithm. The advantage of introducing this formulation is that it allows us to describe the voting algorithm $\V(\cdot, p)$ in a completely general sense, based on a voting algorithm $V$ operating on a time-labelled tree, with random inputs on the leaves. This formulation will be useful later in this paper when we prove \eqref{j.1} for general duals with the branching property.

\subsection{The $g$-function} \label{s:gfunction}
We now describe how the function $g:[0,1]^{N_0} \to[0,1]$ introduced in Section \ref{s.gint} arises in this precise voting algorithm. The notion of the $g$-function was introduced in \cite{DH2021}, and it is given by
\begin{equation}\label{e.gdef}
g(p_1,\dots,p_{N_0}) := E_{p_1,\dots,p_{N_0}}[\Theta(\mathsf{V}_1,\dots,\mathsf{V}_{N_0})],
\end{equation}
where $(\mathsf{V}_1,\dots,\mathsf{V}_{N_0})$ is a vector of independent Bernoulli($p_i$) random variables ($i \in [N_0]$) with law $E_{p_1,\dots,p_{N_0}}$, and $\Theta$ is as in Section \ref{ss.voted}.

The $g$-function defined above appears in a direct computation of the voting algorithm. From \eqref{e.thetarel}, we have that, conditional on the votes of the offspring of $\alpha \in \mathcal{T}$, the Bernoulli parameter of $V_\al$ is $\Theta(V_{\alpha \vee 1}, \dots, V_{\alpha \vee N_0}) = \Theta_\alpha$. The offspring votes $V_{\alpha \vee 1}, \dots, V_{\alpha \vee N_0}$ are themselves Bernoulli random variables by construction; if instead of conditioning on values of the votes, we simply condition on their parameters, we obtain that
\begin{equation}\label{e.bcomp} 
\begin{aligned}
P[V_\alpha = 1 \mid \Theta_{\alpha \vee 1}, \dots,\Theta_{\alpha \vee N_0}] &= E[\Theta_\alpha \, | \, \Theta_{\alpha \vee 1}, \dots,\Theta_{\alpha \vee N_0}] \\
&=  E_{\Theta_{\alpha \vee 1}, \dots,\Theta_{\alpha \vee N_0}}[\Theta(\mathsf{V}_1,\dots,\mathsf{V}_{N_0})]\\
&=g(\Theta_{\alpha \vee 1}, \dots,\Theta_{\alpha \vee N_0}).
\end{aligned}
\end{equation}
That is, $g$ computes the Bernoulli parameter of the parent's vote, given the Bernoulli parameters of the votes of its offspring. This encodes how the input randomness at the leaves is passed through the tree in the voting algorithm. Now that we have defined $g$, we remind the readers of the assumptions on $g$ introduced in Section~\ref{s.modelsduals}, \eqref{g.0}-\eqref{g.5}, under which we prove Theorem~\ref{t.generaldual}. 

The importance of the $g$-function derives from its appearance in a formula for the conditional expectation of $\V(\X_t ; p)$ obtained by applying the strong Markov property at the first branch time. We begin with a heuristic description of this property. Suppose that $(\X_t : t \geq 0)$ is a pure branching process in the sense of Definition~\ref{d.defprocess}. Let $\tau$ denote the first branch time of $(X_s, s\geq 0)$ and $\cF_\tau$ the $\sigma$-algebra of the filtration, up to time $\tau$.  For fixed $t>0$, $\V(\X_t ; p)$ is the vote of the root and thus is determined by the votes of its offspring. Conditional on the event $\left\{\tau \leq t\right\}$, the vote of each child is determined by the sub-tree rooted at that child. In particular, if the $i$th child is born at location $Z_i$ at time $\tau$, then by the branching property, its vote equals $1$ with probability $Q_{Z_i}[\V(\X_{t-\tau};p) = 1]$, independent of its siblings. By definition of $g$, if $\tau<t$, then the conditional probability that the root votes $1$ equals $g$ evaluated at these probabilities. Hence, we have
\begin{multline*}Q_x[\V(\X_t;p)=1 \, | \, \cF_\tau]\indc_{\{\tau \leq t\}} \\
= g(Q_{Z_1}[\V(\X_{t-\tau};p) = 1], \dots, Q_{Z_{N_0}}[\V(\X_{t-\tau};p) = 1])]\indc_{\{\tau \leq t\}}.
\end{multline*}
Recall that in Definition \ref{d.defprocess}, the displacements of the offspring are $\left\{\xi_i\right\}_{i\in [N_0]}$. For $t < \tau$, the position of the root individual is $X_t(\emptyset)$.
It is convenient to abuse notation and denote the trajectory of the root individual by $Y_t$, in which case the location of the parent at the first branch event is $Y_{\tau-}$. Since $Y$ is a Hunt process, it follows that $Y_{\tau-} = Y_\tau$ a.s., and hence the offspring positions are given by $Z_i = Y_\tau + \xi_i$. This leads to the more precise version stated below. 

\begin{lemma} \label{lemma:gfunMarkov}
For all $x$ and $t>0$,
\begin{align*}
Q_x[\V(\X_t;p) = 1 \, | \, \cF_\tau] &= g(\mathfrak{p}_1\dots,\mathfrak{p}_{N_0})\indc_{\{\tau \leq t\}} + Q_x [\V(\X_t;p) = 1, \tau > t], 
\end{align*}
where, for $i\in [N_0]$, $\mathfrak{p}_i$ is the $\cF_\tau$-measurable random variable
\[ \mathfrak{p}_i := Q_{Y_{\tau} + \xi_i}[\V(\W_{t-\tau};p) = 1].\]
\end{lemma}
As we will see, this lemma plays a key role in the verification of \eqref{j.4} in the proof of Theorem~\ref{t.generaldual}.

\subsection{Verification of (J2)-(J3)}

In light of the precise framework introduced in the prior sections, we may now verify that \eqref{j.2} and \eqref{j.3} hold under the assumptions of \eqref{g.0}-\eqref{g.5}.

\begin{lemma}\label{lem.j23} Suppose that $(\X^{\eps}_t, t \geq 0)$ is a dual process in the sense of Definition~\ref{d.defprocess} which is associated to a $g$-function satisfying \eqref{g.0}-\eqref{g.5}. Then \eqref{j.2} and \eqref{j.3} are satisfied.
\end{lemma}
\begin{proof}
Fix $x$ and $t>0$. We can suppress the dependence on $\eps$ since it plays no role in the argument. To prove \eqref{j.2}, suppose that $p(\cdot)\leq \hat{p}(\cdot)$. This implies that $p(X_{t}(i))\leq \hat{p}(X_{t}(i))$ for all $i\in N(t)$. Let $\al \in \mathcal{T}(\X_t)$ be in the second to last generation, i.e. its children are leaves. By a similar computation as in \eqref{e.bcomp}, we have that for any $x$, for any $t>0$,
\begin{align*}
E_{x}[E[\Theta_{\al}\mid \X_{t}]]&=E_{x}[E[P[v_{\al}=1]\mid \X_{t}]] \\
&=E_{x}[E_{p(X_{t}(\al\wedge 1)), \ldots, p(X_{t}(\al\wedge N_{0}))}[\Theta(\mathsf{V}_{1}, \ldots, \mathsf{V}_{N_{0}})]]\\
&=E_{x}[g(p(X_{t}(\al\wedge 1)), \ldots, p(X_{t}(\alpha \wedge N_{0})))]. 
\end{align*}
By \eqref{g.0}, it follows that, in expectation, the Bernoulli parameter of the parent $\Theta_{\al}$ will be larger with $\hat{p}$ than with $p$, thereby increasing the probability that any parent of the leaf children votes 1. By iterating this argument through the ancestral process $\X_{t}$, using a computation similar to the above, this implies that 
\begin{equation*}
Q_x[\V(\X_t; p) = 1]=E_{x}[\V(\X_{t}; p)]\leq E_{x}[\V(\X_{t}; \hat{p})]=Q_{x}[\V(\X_{t}; \hat{p})=1],
\end{equation*}
as desired.

To prove \eqref{j.3}, we claim that $a,b$ as in \eqref{g.1} are the desired equilibria for the conclusion of \eqref{j.3}.  Indeed, since $a$ is a fixed point of $g$ according to \eqref{g.1}, the above calculation demonstrates that when $p\equiv a$, for $\al$ the parent of any of the leaves, $E[\Theta_{\al}]\equiv a$. Again by iterative back propagation, this implies that $Q_x[\V(\X_t; a) = 1]=a$. An analogous argument can be made for $b$. 
\end{proof}

While this framework of the $g$-function is robust enough for most models we are interested in, we conclude this section by remarking that a significantly more general treatment of the $g$-function is necessary for certain models, such as the nonlinear voter model perturbation in Section \ref{s:nonlinearvoter}. This is due  to the non-vanishing impact of coalescences in the voting algorithm in that setting. We discuss this generalization in Section \ref{s:nonlinearvoter}.

\section{(J1) for duals of pure branching type}\label{s.J1}
In this section, we prove \eqref{j.1}, the special Markov property for dual processes which are of pure branching type. We suppress the dependence on the parameter $\eps>0$, since it plays no role in the argument. We remark that the results in this section are independent of the limiting flow and depend only on the branching structure of the dual. 

In the sequel, $(\X_t, t \geq 0)$ is a spatial branching process in the sense of Definition~\ref{d.defprocess}, and $\V$ is any voting algorithm within the framework of Section~\ref{ss.voted}. No assumptions on the $g$-function are required for the proof of the following. However, the branching structure of $(\X_t, t\geq 0)$ is essential, and the result will not hold in general without it. Thus, if the true dual process does not have a branching structure, \eqref{j.1} can only be verified for an approximate dual with the branching property.

\begin{proposition} \label{prop:J1}
For $p : \R^{\mathbf{d}} \to [a,b]$, let $\V(\W_t;p)$ denote the voting algorithm as described in Section \ref{s:vote}. Then for $0<s<t$, we have 
\begin{equation*}
Q_x[\V(\W_t; p)=1] = Q_x[\V(\W_s; q_{t-s})=1],
\end{equation*}
where $q_{t-s}(y) := Q_y[\V(\W_{t-s}; p)=1]$.
\end{proposition}

Before we give the proof of this result, we comment on an interesting point. The condition \eqref{j.1} is identical to the condition (H1) from \cite{BS}, and it is a natural semigroup property. However, if written in terms of $w^\eps_t$ instead of the dual, and ignoring the error term coming from approximate duality, one obtains the following rather exotic form of the Markov property: for every $t,h>0$ and every $x \in \R^{\mathbf{d}}$,
\begin{equation*}
\E^{\eps}_{\E^{\eps}_{p} [w^{\eps}_t]} [w^{\eps}_h(x)] = \E^{\eps}_{p}[w^{\eps}_{t+h}(x)].
\end{equation*} 
It is unclear in general how such a property should arise, except, as we prove momentarily, when $w^\eps_t$ is equipped with a dual process with the branching property. It is also remarkable that the condition as formulated in a purely analytic setting in \cite{BS} corresponds precisely, when applied to a stochastic model, to the notion that the dual process has the branching property. 

\begin{proof}[Proof of Propostion~\ref{prop:J1}]
As in Section \ref{s:vote}, we begin with a discussion based entirely on a fixed, time-labelled tree $\mathcal{T}$. Since $\mathcal{T}$ is a tree and the vote of each individual is determined by the votes of its offspring, the outcome $V\big(\mathcal{T},v(L(\mathcal{T}))\big)$ can by computed if we know the votes of all individuals at any given time height in the tree. Abusing notation slightly, we write $\mathcal{T}_s$ to denote the tree $\mathcal{T}$ with vertices $t_{v}\leq s$, for $s\leq \max_{v\in \mathcal{T}} t_{v}$, and we denote $L_{s}:=L(\mathcal{T}_{s})$. We then have
\begin{equation}\label{e_tree_alg_prop}
 V\big(\emptyset; \mathcal{T},v(L(\mathcal{T}))\big) = V\big(\emptyset; \mathcal{T}_s, V(L_s ; \mathcal{T},v(L(\mathcal{T})))\big), 
\end{equation}
where $V\big(L_s ; \mathcal{T},v(L(\mathcal{T}))\big)$ denotes the votes assigned to the leaves of $\mathcal{T}_s$, in the course of the evaluation of $V\big(\mathcal{T},v(L(\mathcal{T}))\big)$. 

From \eqref{e_tree_alg_prop}, we obtain that for $0<s<t$, with $\mathcal{T}(\W_{t})=\mathcal{T}_{t}$ and $\mathcal{T}(\W_{s})=\mathcal{T}_{s}$,  
\begin{equation}\label{e:lowertree}
Q_x\left[V(\emptyset; \mathcal{T}_{t}, \mathcal{V}_p(L_t))=1\right]=Q_x[V(\emptyset; \mathcal{T}_{s}, V(L_s; \mathcal{T}_{t}, \mathcal{V}_p(L_t)))=1], 
\end{equation}
where we recall that $\mathcal{V}_p(L_t)$ is the random collection of votes assigned to the leaves via the function $p$.  
That is, given the assignment of leaf votes $\mathcal{V}_p(L_t)$, the vote of the root is equal to the output of the algorithm on $\mathcal{T}_s$, when the input of the leaves $L_{s}$ is the assignment of votes of time $s$ individuals which are governed by the algorithm run on the orginal tree $\mathcal{T}_t$, with leaf votes $\mathcal{V}_p(L_t)$. 

The final point is the conditional distribution of $V(L_s; \mathcal{T}_t, \mathcal{V}_p(L_t))$ given $X_s$. Conditional on $X_s$, by the branching property, each individual $\alpha \in L_s$ is the root of an independent copy of $X$, started from location $X_s(\alpha)$. (That is, up to a relabeling.) In particular, because the vote $V_\alpha$ corresponding to individual $\al$ depends only on the subtree of height $t-s$ descending from the individual $\alpha$ at time $s$, we have the following: for $\al\in L_{s}$,
\begin{equation}
Q_x[V\big(\alpha ; \mathcal{T}_t, \mathcal{V}_{p}(L_t))\big) =1 \,| \, X_s]= Q_{X_s(\al)}[V(\emptyset; \mathcal{T}_{t-s},\mathcal{V}_p(L_{t-s})) =1]. 
\end{equation}
We furthermore remark that the right hand side above is equal to $q_{t-s}(X_s(\alpha))$. In particular, since $V$ can only take the states of $0$ or $1$, given $X_s$, we have
\begin{equation} \label{e:condvotedist}
\left\{V\big(\alpha;\mathcal{T}_t, \mathcal{V}_p(L_t))\big) : \alpha \in L_s\right\} \stackrel{d}{=} \bigotimes_{\alpha \in L_s} \text{Bernoulli}(q_{t-s}(X_s(\alpha))),
\end{equation}
where $\otimes$ denotes an independent product. 

Note that the above display is a characterization of the distribution of $V\big(L_s; \mathcal{T}_t, \mathcal{V}_p(L_t))\big)$, given $X_s$. Returning to \eqref{e:lowertree}, we condition on $X_s$ and apply \eqref{e:condvotedist}, to obtain  
\begin{align*}
Q_x[V(\emptyset; \mathcal{T}_t,\mathcal{V}_p(L_t))=1 \, | \, X_s] &= Q_x[V\big(\emptyset; \mathcal{T}_s, V(L_s; \mathcal{T}_t, \mathcal{V}_p(L_t))\big)=1 \, | \, X_s] \notag
\\ &= Q_x[V(\emptyset; \mathcal{T}_s, \mathcal{V}_{q_{t-s}}(L_s))=1\, | \, X_s].
\end{align*}
Taking expectations, we obtain
\begin{equation*}
Q_x[V(\emptyset; \mathcal{T}_t,\mathcal{V}_p(L_t))=1]= Q_x[V(\emptyset; \mathcal{T}_s,\mathcal{V}_{q_{t-s}}(L_s))=1].
\end{equation*}
Returning to the original notation via \eqref{e:votetreeequivalence}, this is equivalent to
\begin{equation*}
Q_x[\V(\W_t; p )=1] = Q_x[\V(\W_s ; q_{t-s})=1],
\end{equation*}
as asserted.\end{proof}

\section{Verification of (J4) in the mean curvature case} \label{s:genthm}
We now present a general framework and set of assumptions on the (approximate) dual process which will allow us to verify \eqref{j.4} with $F$ as in \eqref{e.Fdef}. Throughout this section, we assume that \eqref{eq:approxdualdef} is in place; in particular, there is an approximate dual with a branching/tree structure, and we will drop the term ``approximate'' and refer to it simply as the dual.

We consider a family of dual processes parameterized by the scaling parameter $\eps>0$. The dual $(X_t^\eps, t\geq 0)$ and its historical process are of the form given in Definition~\ref{d.defprocess}. The set-up and notation in this setting is summarized as follows: $(X^\eps_t : t \geq 0)$ is a branching Markov process, whose law when started from a single individual at $x$ is denoted $Q^\eps_x$, with branching and spatial motion as below.

\begin{itemize}
\item {\bf Spatial motion.} In between branch times, individuals evolve in space as independent copies of a Hunt process $(Y_t, t \geq 0)$  on $\R^{\mathbf{d}}$, with paths in the Skorokhod space $\mathbb{D}([0, \infty), \R^{\mathbf{d}})$, whose $\eps$-dependent law we denote by $P^{Y,\eps}_x$ when $Y_0 = x$.  
\item {\bf Branching.} With branch rate $\gamma_\eps := \gamma \eps^{-2}$, for $\gamma>0$, individuals branch into $N_0$ individuals. After giving birth, an individual is removed from the population. At the branch time, if the location of the parent is $y$, then the offspring displacements have joint law $\mu^\eps_y$; thus, for a branching event with parent location $y$, the offspring location vector is $(y+\xi_1,\dots, y+\xi_{N_0})$, where $(\xi_1,\dots,\xi_{N_0})$ is a sample from $\mu^\eps_y$. The map $y \mapsto \mu^\eps_y$ is assumed to be measurable for every $\eps>0$.
\end{itemize}

In order to refer back easily to these various constants, we collect all of them here as a set parameters which are universal to any model under consideration. We signify dependence on $\eqref{e.univconstants}$ as dependence on any of the following constants:
\begin{equation}\label{e.univconstants}
\begin{cases}\tag{$\heartsuit$}
\text{$\mathbf{d}$, the spatial dimension,}\\
\text{$\ga$, the branching rate parameter of the unscaled process $\W(t)$,}\\
\text{$\norm{g}_{C^{2}([0,1])}$, the $g$-function introduced in Section \ref{s:vote},}\\
\text{$\delta_{*}=\delta_{*}(\norm{g}_{C^{2}([0,1])})$ and $c_{0}\in (0,1)$, in \eqref{g.5}},\\
\text{$N_{0}$, the number of children in a branching event,}\\
\text{$k>1$, $\bar{C}, \bar{c}, \eta$ as in \eqref{a.1} and \eqref{a.2}}.
\end{cases}
\end{equation}

The main result of this section is the following. 

\begin{theorem} \label{thm_J4_general}
Suppose that $(w_t^\eps)_{\eps > 0}$ is a family of stochastic spatial models equipped with a family of (approximate) duals $(\W^{\eps}_t)_{\eps >0}$ satisfying \eqref{eq:approxdualdef}, where $(\W^{\eps}_t, t\geq0)$ has the form described above and satisfies \eqref{a.1}, \eqref{a.2}, and $\V$ is associated to a $g$-function satisfying \eqref{g.0}-\eqref{g.5}. Then \eqref{j.4} holds with $F$ defined as in \eqref{e.Fdef}.
\end{theorem} 

\begin{remark} (a) We do not claim that these are the sharpest conditions under which our result holds. Indeed, they can likely be relaxed. The method of proof requires that the dual trajectories can be approximated by Brownian motion up to some polynomial error in $\eps$, so the stretched exponential probability of exceeding the error can be relaxed to polynomial decay of sufficiently high order. However, these conditions are quite mild and indeed hold in the cases of interest. More precisely, in the scaling regimes where \eqref{eq:dualapproxdual} holds, the dual lineages can generally be coupled with a Brownian motion to within an exponentially small (or stretched exponentially small) error with (stretched) exponentially decaying probability of failure. Furthermore, stretched exponential decay holds over constant order time-scales. The assumption \eqref{a.2} only assumes that it holds over a time-scale of order $\eps^2|\log \eps|^2$; this could be further relaxed to $\eps^2 |\log \eps|$ (at least) but the proof is simplified with the assumption as given.

(b) The choice $\gamma_\eps = \gamma \eps^{-2}$ for the branching rate is for convenience; the proof works under the relaxed assumption that $\eps^2 \gamma_\eps \to \gamma$.
\end{remark} 

From this point on through the rest of this section, in order to lighten notational load, we suppress the dependence of $X_t^\eps$ and $\X_t^\eps$ on $\eps$ and simply write $X_t$ and $\X_t$. The dependence on $\eps$ is still apparent through the probability measure $Q^\eps_x$.

We remark that although \eqref{j.4} consists of two claims, we will only prove (ii). The proof of (i) is completely analogous and hence omitted. Hereafter when we discuss the proof of \eqref{j.4}, it is understood that we are referring to the proof of \eqref{j.4}(ii). 

We now fix such a function and some relevant notation which is in effect for the remainder of this section. Let $(\phi,B(x_{0},r))=(\phi,B)$ be a pair which satisfies
\begin{equation}\label{e.phidef}
\begin{cases}
\text{$\phi : \R^{\mathbf{d}} \to \R$ is a smooth function satisfying $\{\phi \leq 0\} \subseteq B$,}\\
\text{$D\phi(\cdot) \neq 0$ on $\{\phi = 0\}$.}
\end{cases}
\end{equation}

In order to prove \eqref{j.4}(ii), we will show that given $(\phi, B)$ as above, there exists $h_0=h_{0}( \norm{\phi}_{C^{4}(\overline{B})}, \text{\ref{e.univconstants}})\in (0,1)$ such that for every $\alpha\in (0,1)$, the following holds: for 
\begin{equation*}
L^{-}_{h, \al}:=\left\{x: \phi(x) - h\left[ F_{*}(D^2\phi(x), D\phi(x)) - \alpha \right] <0\right\}\subseteq \R^{\mathbf{d}},
\end{equation*}
for all $\delta \in (0, \delta_{*}]$ where $\delta_{*}$ is defined in \eqref{g.5}, for all $h\in (0,h_{0}]$, and for all $x\in L^{-}_{h,\al}\cap B$, 
\begin{equation*}
\limsups_{\eps \to 0} \,\,u^{\eps}(t,x; p^{+}(\phi, \delta)) = a,
\end{equation*}
where 
\begin{equation}\label{e.p0+def}
p^+(\phi,\delta) =  (a + \delta) \indc_{\{\phi \leq 0\}} + b \indc_{\{\phi> 0\}}.
\end{equation}

We note that in the verification of \eqref{j.4}(ii), we take $\bar{\delta} = \delta_*$, as defined in \eqref{g.5}, although this is mainly for convenience; the desired claims should hold for all $\delta < \avg - a$.

The proof of the above claim (and hence Theorem~\ref{thm_J4_general}) follows the same two-step procedure as in \cite{EFP2017} and \cite{DH2021}; we first show that an interface forms on a time-scale of order $\eps^2 |\log \eps|$, and then we show that the interface propagates like MCF over constant order time-scales. These arguments are given in Sections~\ref{ss:formation} and \ref{ss:propagation}, after which we complete the proof of Theorem~\ref{thm_J4_general} in Section~\ref{ss:J4holds}.

Before diving into the proof of \eqref{j.4}(ii), we begin with an analysis of the signed distance function to the boundary of $L^{-}_{h,\al}$, which we present in Section \ref{s:distance}. We establish further preliminaries on a one-dimensional BBM model in Section \ref{ss.bbm1}.

\subsection{Distance functions and the linearized level set equation} \label{s:distance}
 We fix $\al\in (0,1)$ throughout this section. Our approach will be based on analyzing the signed distance function to level sets of the function $\psi=\psi_{\al}: [0, \infty)\times \R^{\mathbf{d}}\rightarrow \R$, which we now define. 
 
 Fix $(\phi,B)$ satisfying \eqref{e.phidef}, and let 
\begin{equation}\label{e.psialdef}
\psi(t,x)=\psi_\alpha(t,x) := \phi(x) - t\left[ F_*(D^2 \phi(x),D\phi(x)) -\alpha \right],
\end{equation} 
where $F_{*}$ is defined in \eqref{e.flsc}. 

Since $\phi$ is smooth and $|D\phi(\cdot)|\neq 0$ on $\left\{\phi=0\right\}$, there exists a constant $h_{0}\in (0,1)$ such that $\left\{\psi(t, \cdot)=0\right\}\subseteq \R^{\mathbf{d}}$ is a smooth (at least $C^{1}$) $(\mathbf{d}-1)$-dimensional hypersurface for all $t\in [0, h_{0}]$. In particular, $|D\psi(t,\cdot)|\neq 0$ on $\left\{\psi(t, \cdot)=0\right\}$ for all $t\in [0, h_{0}]$. We note that $h_{0}$ depends on $\norm{\phi}_{C^{2}(\overline{B})}$, and for reasons which will become clear later in the proof, we will allow $h_{0}=h_{0}(\norm{\phi}_{C^{4}(\overline{B})})$. 

For each $t\geq 0$, we define three sets associated to $\psi$:
\begin{equation}\label{e.Ldefs}
L^0_{t} = \{\psi(t,\cdot) = 0\},\quad  L^+_{t} = \{\psi(t,\cdot) > 0\}, \quad L^-_{t} = \{\psi(t,\cdot) < 0\}. 
\end{equation}
In order to identify the normal velocity with which the $(\mathbf{d}-1)$-dimensional interface $\left\{\psi(t, \cdot)=0\right\}$ evolves, we note that by definition, for $t\in [0, h_{0}],$
\begin{align*}
\partial_{t}\psi=-F_{*}(D^{2}\phi, D\phi)+\al=-\frac{F(D^{2}\phi, D\phi)-\al}{|D\psi|}|D\psi|, 
\end{align*}
and hence by \eqref{e.normvel}, the normal velocity of this interface in the time interval $[0,h_{0}]$ must be given by
\begin{equation}\label{e.psinv}
V_{n}=\frac{F(D^{2}\phi, D\phi)-\al}{|D\psi|}.
\end{equation}

For each $t>0$, let $d(t, \cdot)=d_{\al}(t, \cdot)$ denote the signed distance function to $L^0_{t}$, with the convention that $d$ has the same sign as $\psi$, i.e. $d(t,\cdot)< 0$ on $L^-_{t}$ and $d(t,\cdot) > 0$ on $L^+_{t}$. We now present several properties of $d(t, \cdot)$ which we will use throughout our analysis.
\begin{proposition} \label{p.distprop} Let $\al, (\phi, B), \psi, d$ be as above. There exists a constant $h_{0}=h_{0}(\norm{\phi}_{C^{4}(\overline{B})})\in (0,1)$ such that the following holds: 
\begin{enumerate}[(i)]
\item For every $t\geq 0$, 
\begin{equation}\label{e.eikonal}
\begin{cases}
|Dd(t, \cdot)|=1&\text{in $\R^{\mathbf{d}}\setminus L^{0}_{t}$},\\
d(t,\cdot)=0&\text{in $L^{0}_{t}$}, 
\end{cases}
\end{equation}
pointwise almost everywhere and in the viscosity sense.
\item There exists $r_{0}=r_{0}(\norm{\phi}_{C^{4}(\overline{B})})\in (0,1)$ such that $d\in C^{1,2}(Q_{h_{0}, r_{0}})$ where $Q_{h_{0}, r_{0}}:=\left\{(t,x): t\in (0, h_{0}), |d(t,x)|< r_{0}\right\}$, and
\begin{equation} \label{e_distanceheatbd}
\partial_{t}d(t,x) - \frac{1}{2} \Delta d(t,x) \geq \frac{\alpha}{4|D\psi(t,x)|} \quad \text{in $Q_{h_{0}, r_{0}}$}.
\end{equation}
\item There exists a constant $C_{1}=C_1(\norm{\phi}_{C^{4}(\overline{B})})\in [1, \infty)$ and $\tau_0\in (0,h_{0}/2)$ such that for all $t \in [0,h_0 - \tau_0]$ and $s \in [t,t+\tau_0]$,
\begin{equation*}
\sup_{y\in \R^{\mathbf{d}}} |d(t,y) - d(s,y)| \leq C_1|t-s|.
\end{equation*}
In particular, there exists a constant $C_{1}$ (possibly relabeled) such that 
\begin{equation} \label{e_lipschitz_distance}
\sup_{y\in \R^{\mathbf{d}}}\max_{t\in [0, h_{0}]}|\partial_{t}d(t,y)|\leq C_{1}.
\end{equation}
\end{enumerate}
\end{proposition}

\begin{remark}
In both \cite{EFP2017, DH2021}, the authors prove some similar properties for $\hat{d}(t, \cdot)$, the signed distance function to the set $\hat{\Ga}_{t}$ which corresponds to $\Ga_{0}$ evolved according to the classical MCF up to time $t$, for all times $t\leq \mathrm{T}$, the first time when MCF develops singularities. 

We highlight property (ii) in Proposition \ref{p.distprop} yields that $d(t,\cdot)$, the signed distance function to the level sets of $\psi(t, \cdot)$, is a strict supersolution to the heat equation. In the case of $\hat{d}$, it is a consequence (see for example \cite[Section 2.3]{EFP2017}) that there exists constants $\ga_{0}, c_{0}>0$ such that 
\begin{equation*}
\partial_{t}\hat{d}(t,x)-\frac{1}{2}\Delta \hat{d}(t,x)\geq -c_{0}\quad \text{in $Q_{\mathrm{T}, \ga_{0}}$},
\end{equation*}
whereas \eqref{e_distanceheatbd} gives us a sharper bound (which will allow us to simplify some of our analysis in the latter parts of the argument). 
\end{remark}

\begin{proof} The proof of (i) is well-known and follows from the definition of the signed distance function. 

For (ii), the regularity is a consequence of  \cite[Proposition 67]{M}. Furthermore, the functions $d$ and $\psi$ must necessarily share the same $0$-level sets (which are smooth $(\mathbf{d}-1)$-dimensional hypersurfaces for $h_{0}$ chosen sufficiently small), and hence these level sets share the same normal vectors and normal velocities. This implies that 
\begin{equation}\label{e.dspace}
Dd=\frac{D\psi}{|D\psi|}\quad\text{and}\quad D^{2}d=D\left(\frac{D\psi}{|D\psi|}\right)\quad\text{on $\bigcup_{t\in (0, h_{0}]} \left\{t\right\}\times L^0_{t}$,}
\end{equation}
and by \eqref{e.psinv} and (i), 
\begin{equation}\label{e.dtime}
\partial_{t}d=-\frac{F_{*}(D^{2}\phi, D\phi)-\al}{|D\psi|}\quad\text{on $\bigcup_{t\in (0, h_{0}]} \left\{t\right\}\times L^0_{t}$.}
\end{equation}
Combining \eqref{e.dspace}, \eqref{e.dtime}, and \eqref{e.Fdef} yields 
\begin{align*}
\partial_{t}d-\frac{1}{2}\Delta d&=\partial_{t}d-\frac{1}{2}\di\left(Dd\right)=-\frac{F(D^{2}\phi, D\phi)-\al}{|D\psi|}-\frac{1}{2}\di\left(\frac{D\psi}{|D\psi|}\right)\\
&=-\frac{F(D^{2}\phi, D\phi)-\al}{|D\psi|}+\frac{F(D^{2}\psi, D\psi)}{|D\psi|}\geq \frac{\al}{4|D\psi|}, 
\end{align*}
for all $t\in (0, h_{0}]$ for $h_{0}=h_{0}(\norm{\phi}_{C^{4}(\overline{B})})$ sufficiently small. We highlight that this calculation demonstrates that the choice of $h_{0}$ depends precisely on $\norm{\phi}_{C^{4}(\overline{B})}$, since $D^{2}\psi$ has $D^{4}\phi$ in its expression. 

For (iii), we note that first, if $(t,y)\in Q_{h_{0}, \ga_{0}}$, then since $d\in C^{1,2}(Q_{h_{0},\ga_{0}})$, \eqref{e_lipschitz_distance} holds by the regularity of $d$. For $(\cdot, y)\notin Q_{h_{0}, \ga_{0}}$, since we are only interested in small time increments, we may assume that $d(\cdot, y)$ keeps the same sign in this entire time interval, and without loss of generality, we will assume it is positive. Recall that for all $s\in (0, h_{0}]$ for $h_{0}=h_{0}(\norm{\phi}_{C^{4}(\overline{B})})$ sufficiently small, $|D\psi(s,\cdot)|\neq 0$ on $L^0_s$, and hence there exists $M=M(\norm{\phi}_{C^{4}(\overline{B})})\in (0, \infty)$ such that $|V_{n}|\leq M$ for all times $s\in (0, h_{0}]$. This implies that the interface can be displaced by at most sets with Hausdorff distance governed by the constant velocity $M$.  In particular, for all $t, s\in [0, h_{0}]$, with $s<t$, 
\begin{equation*}
-M(t-s)+d(s,y)\leq d(t,y)\leq M(t-s)+d(s,y), 
\end{equation*}
and this yields the claim.
\end{proof}

Finally, we can use It\^{o}'s lemma to couple the values of the distance function $(d(t-s, \bmd_s), 0\leq s\leq t)$ with $(\bmd_{s}, s\geq 0)$ a $\mathbf{d}$-dimensional Brownian motion with a one-dimensional Brownian motion $(\bmo_{s}, 0\leq s\leq t)$.  
\begin{corollary} \label{corollary:coupling}
Let $\alpha\in (0,1)$, and $\phi, \psi, h_{0}$ and $Q_{h_{0}, \ga_{0}}$ as above. Let $t \in (0,h_0]$, and let $(\bmd_{s}, s\geq 0)$ denote a Brownian motion in $\R^{\mathbf{d}}$ started at $x\in \R^{\mathbf{d}}$. Define $\La := \inf \{ r \geq 0 : (t-r,\bmd_r) \not \in Q_{h_0, \ga_{0}}\}$. Then there is a one-dimensional Brownian motion $(\bmo_s, s\geq 0)$ started at $0$ such that for all $s \in (0,t)$,
\begin{equation}
d(t-(s \wedge \La), \bmd_{s \wedge \La}) \leq d(t,x) + \bmo_{s\wedge \La} - \frac{\alpha (s \wedge \La) }{4L},
\end{equation}
where $L = L(\norm{\phi}_{C^{4}(\overline{B})}) := \,\sup \{|D\psi(t,x)| : (t,x) \in Q_{h_0, \ga_{0}} \} > 0$. 
\end{corollary}

\begin{proof}
Let $t \in (0,h_0]$ and $x$ be an interior point of $Q_{h_0, \ga_{0}}$, as otherwise $\La=0$ almost surely and the statement of is trivial. Under $P^{\bmd}_x$, the space-time process $((r\wedge \La, \bmd_{r \wedge \La}), r \geq 0)$ remains in $Q_{h_0, \ga_{0}}$. We may therefore apply It\^{o}'s Lemma to the function $f(r, \bmd_r) = d(t-r, \bmd_r)$ to obtain that
\begin{multline*}
d(t-(s \wedge \La), \bmd_{s \wedge \La}) - d(t, x) \\
= \int_0^{s \wedge \La}D d(t-r, \bmd_r) \cdot d\bmd_r+ \int_0^{s \wedge \La} \left[-\partial_{t}d + \frac{1}{2}\Delta d\right](t-r, \bmd_r)\, dr. 
\end{multline*}
Since $(t-r,\bmd_r) \in Q_{h_0, \ga_{0}}$ for all $r \in [0,s\wedge \La]$, by Proposition~\ref{p.distprop}(ii) the integrand in the second term of the right hand side is bounded above by $-\alpha(4|D\psi(t-r,\bmd_{r})|)^{-1}\leq -\alpha(4L)^{-1}$, and hence
\begin{align}
&d(t-(s \wedge \La), \bmd_{s \wedge \La}) - d(t, x) \leq \int_0^{s \wedge \La} D d(t-r, \bmd_r) \cdot d\bmd_r - \frac{\alpha (s\wedge \La)}{4L}. \nonumber
\end{align}
Moreover, by Proposition~\ref{p.distprop}(i), $|D d(t-r,\bmd_r)| = 1$. This implies that the stochastic integral above is a continuous martingale with quadratic variation equal to its time parameter, and hence is a Brownian motion (stopped at $\La$) by L\'evy's characterization of Brownian motion. This completes the proof.
\end{proof}

\subsection{Branching Brownian motion in one dimension}\label{ss.bbm1}
In this section, we review some properties of the $N_{0}$-ary one-dimensional branching Brownian motion (BBM), which will be used later as a comparison process to the general multidimensional models. We highlight that this one-dimensional BBM has the advantage that the movement is given by Brownian motion and children are born at the exact location of their parent at birth. The results in this subsection are all proved in \cite[Section 3.1]{DH2021}.

Let $(B_t, t\geq 0)$ be an $N_0$-ary branching Brownian motion on $\R$ which branches at rate $\gamma \eps^{-2}$ (and $(\B_t, t\geq 0)$ its historical process). We denote the law and expectation of this process, when started from a single particle at $z \in \R$, by $P^\eps_z$ and $E^\eps_z$. Given $p : \R \to [a,b]$, we may compute the vote $\V(\B_t ; p)$ of a realization of $\B_t$ using the same voting algorithm $\V$ as on the dual process $\W_t$. Since the voting algorithm is the same and is independent of the dimension, the one-dimensional model inherits several properties of the dual. For instance, as a consequence of Lemma \ref{lemma:gfunMarkov} and \eqref{g.0}-\eqref{g.5}, 
\begin{equation} \label{eq:1dBBMstable}
P^\eps_z[ \V(\B_t; a) = 1] = a, \quad P^\eps_z[ \V(\B_t; b) = 1] = b,
\end{equation}
and if $p_1, p_2 : \R \to [0,1]$ with $p_1 \leq p_2$, then
\begin{equation} \label{eq:1dBBMmono}
P^\eps_z[ \V(\B_t; p_1) = 1] \leq P^\eps_z[ \V(\B_t; p_2) = 1].
\end{equation}
For this one-dimensional process, we hereafter restrict our attention to the particular voting function
\begin{equation}\label{e.p*def}
p_* (x)= a \indc_{(-\infty, 0)}(x)+ b\indc_{[0, \infty)}(x).
\end{equation}
We therefore simply write $\V(\B_t) \equiv \V(\B_t; p_*)$. For a non-branching Brownian motion on $\R$ which branches at rate $\ga\eps^{-2}$, we will simply write $\bmo_t$, which satisfies $\bmo_0 = z$ under $P^\eps_z$. (We think of $\bmo_t$ as being the trajectory of the first individual up to its first branch time.) We now proceed to state the properties of this model which we will use.

First, we remark that as a consequence of \eqref{eq:1dBBMstable}, \eqref{eq:1dBBMmono}, and \eqref{e.p*def}, we have,
\begin{equation} \label{eq:1diminterval}
P^\eps_z[ \V(\B_t) = 1] \in [a,b] \quad \text{for all} \,\, z \in \R, t > 0.
\end{equation}
Moreover, the one-dimensional system is monotone, in the sense that for $z_1 \leq z_2$,
\begin{equation} \label{e:BBMmonotone}
P^\eps_{z_1} [ \V(\B_t) = 1] \leq P^\eps_{z_2} [ \V(\B_t) = 1].
\end{equation}
This is a consequence of \eqref{eq:1dBBMmono}, the translation invariance of the process, and the fact that $p_*$ is non-decreasing. 

For the one-dimensional model, the special version of the strong Markov property from Lemma~\ref{lemma:gfunMarkov} takes the following form. If $\tau$ denotes the first branch time, then
\begin{equation}\label{eq:SMP_browniancomparison}
P^\eps_z[ \V(\B_t) = 1] = E^\eps_z [ g(P^\eps_{B_{\tau}}[\V(\B_{t-\tau}) = 1]) \indc_{\{\tau \leq t\}}]+ P^\eps_z [\V(\B_t) = 1, \tau > t].
\end{equation}

The following result from \cite{DH2021} demonstrates that for the one-dimensional model with initial voting function $p_{*}$, the interface remains ``stable.'' 
\begin{theorem}\cite[Theorem 3.6]{DH2021} \label{thm:bbminterface}
Fix $\La\in (0,\infty)$. For all $\ell \in \N$, there exists $c_1=c_{1}(\ell, \text{\ref{e.univconstants}})\in [1, \infty)$ and $\eps_{1}=\eps_{1}(\ell, \text{\ref{e.univconstants}}, \La)\in (0,1)$ such that for all $t \in [0,\La]$ and $\eps<\eps_{1}$, 
\begin{itemize}
\item For $z \geq c_1 \eps|\log\eps|$, $P^\eps_z[\V(\B_t) = 1] \geq b - \eps^{\ell}$.
\item For $z \leq -c_1 \eps|\log\eps|$, $P^\eps_z[\V(\B_t) = 1] \leq a  + \eps^{\ell}$.
\end{itemize}
\end{theorem}

In addition, we will need the following result concerning the slope of the interface. The parameter $\delta_*$ is defined in \eqref{g.5}.
\begin{lemma} \cite[Corollary~3.8]{DH2021}\label{lemma:bbmslope} Fix $\La \in (0,\infty)$. There are constants $c_2=c_{2}\eqref{e.univconstants}\in [1, \infty)$ and $\eps_{2}=\eps_{2}(\text{\ref{e.univconstants}}, \La)\in (0,1)$ such that the following holds: if $t\in [0,\Lambda]$ and $\rho \in \R$ satisfy
\begin{equation*}
\big|P^\eps_\rho [ \V(\B_t)=1] - \avg \big| \leq b - \avg - \delta_*,
\end{equation*}
then for any $\rho' \in \R$ with $|\rho-\rho'| \leq c_2 \eps |\log \eps|$ for $\eps<\eps_{2}$, we have
\begin{equation*}
\big|P^\eps_\rho [ \V(\B_t) = 1] - P^\eps_{\rho'}[\V(\B_t) = 1] \big| \geq \frac{\delta_*|\rho-\rho'|}{c_2 \eps|\log \eps|}.
\end{equation*}
\end{lemma}
The parameter $\delta_*$ may differ from the analogous quantity which appears in \cite{DH2021}, but this has no impact on the statement or proof of the above.

\subsection{Formation of the interface} \label{ss:formation}
Throughout this section, we will need a few constants which we remind the reader of here:
\begin{equation}\label{e.constants1}
\begin{cases}
\text{$c_\gamma=c_{\ga}(k):= k\gamma^{-1}$, for $k$, $\ga$ as in \eqref{e.univconstants},}\\
\text{$h_0 = h_0(\norm{\phi}_{C^{4}(\overline{B})})\in (0,1)$ introduced in Proposition \ref{p.distprop}}\\
\end{cases}
\end{equation}

The main result of this section is that with voting function $p^+(\delta,\phi)$ as defined in \eqref{e.p0+def}, a sharp interface forms on a time-scale of order $\eps^2|\log \eps|$: 
\begin{proposition} \label{prop:interformgen}
There exist constants $\sigma_1=\sigma_{1}\eqref{e.univconstants}\in[1, \infty)$, $K_1=K_{1}\eqref{e.univconstants} \in [1, \infty)$, and $\eps_{3}=\eps_{3}(\text{\ref{e.univconstants}}, \norm{\phi}_{C^{4}(\overline{B})})\in (0,1)$ such that for $\eps<\eps_{3}$, $\delta<\delta_{*}$, and points $(t,x)$ with $t \in [\sigma_1 \eps^2 |\log \eps| , \sigma_2 \eps^2 |\log \eps|]$, where $\sigma_2 = \sigma_1 +c_\ga$ for $c_\ga$ as defined in \eqref{e.constants1}, and $d(t,x) \leq - K_{1}\eps|\log \eps|$, 
\begin{align*}
Q^\eps_x [\V(\W_t ; p^+(\delta,\phi)) = 1] \leq a + \eps^{k}. 
\end{align*}
\end{proposition}

Before pursuing the proof of Proposition \ref{prop:interformgen}, we first recall some preliminaries. Let $g^{(n)}$ denote the $n$-fold composition of $g$ with itself. It is elementary to see from \eqref{g.5} and the mean value theorem that for all $\delta< \delta_{*}$, 
\begin{equation*}
g(a+\delta)\leq g(a)+(1-c_{0})\delta=a+(1-c_{0})\delta. 
\end{equation*}
Iterating this estimate using the fact that $g$ is increasing by \eqref{g.3}, this implies that there exists $A_1 = A_1(\text{\ref{e.univconstants}}, \delta_{*})\in [1, \infty)$ such that for $n \geq A_1 |\log \eps|$, 
\begin{equation} \label{eq:gcompose}
g^{(n)}(a+\delta) \leq a + \eps^{k+1}.
\end{equation}
A slightly stronger version of this estimate is used in \cite[Lemma 3.2]{DH2021}. 

Following the notation of Section~\ref{s:vote}, we write $\mathcal{T}(\W_t)$ to denote the labelled tree associated $\W_t$. For $K \in \N$, let $T^{\text{reg}}_K$ denote the $N_0$-regular tree of height $K$; that is, the root has $N_0$ children, as do each of its children, and so on, and all of the individuals in the $K$th generation are leaves. If $K$ is not an integer, then the same notation is used to denote the $N_0$-regular tree of height $\lfloor K \rfloor$. For two labelled trees, $T_1$ and $T_2$, we write $T_1 \subseteq T_2$ if $T_2$ contains $T_1$ as a (labelled) subgraph.
 
The following lemma states that with high probability, for times $t$ of order $\eps^2 |\log \eps|$, $\mathcal{T}(\W_t)$ contains and is contained in $N_0$-regular trees with heights of order $|\log \eps|$. 
\begin{lemma}\cite{DH2021,EFP2017} \label{lemma:regtree}
There exists constants $A_1<A_2$, both depending on \eqref{e.univconstants}, $\sigma_{1}=\sigma_{1}\eqref{e.univconstants}\in [1, \infty)$ with $\sigma_{2}=\sigma_{1}+c_\ga$, and $\eps_{4}=\eps_{4}\eqref{e.univconstants}\in (0,1)$ such that for any $\eps<\eps_{4}$,
\begin{equation*}
Q^\eps_x \left[ \mathcal{T}(\W_t) \supseteq T^{\text{reg}}_{A_1  |\log \eps|}\right] \geq 1 - \eps^{k+1}\quad\text{for all $t \geq \sigma_1 \eps^2 |\log \eps|$},
\end{equation*}
and
\begin{equation*}
Q^\eps_x \left[ \mathcal{T}(\W_t) \subseteq T^{\text{reg}}_{A_2|\log \eps|}\right] \geq 1 - \eps^{k+1}\quad\text{for all $t \leq \sigma_2 \eps^2 |\log \eps|$.}
\end{equation*}
\end{lemma}
We note that $A_{1}$ in the above statement is the same constant $A_{1}$ appearing in \eqref{eq:gcompose}. The first claim is proved in \cite[Lemma~3.3]{DH2021}. The second is proved for a particular ternary branching process in \cite[Lemma~3.16]{EFP2017}, but the proof for general $N_0$ and branch rate follows along the same lines and we omit it.

The key tool in the proof of Proposition~\ref{prop:interformgen} is the following uniform displacement bound on $(\W_t, t\geq 0)$ for small times. 

\begin{lemma} \label{lemma:unifdisplacement} There exists constants $C=C\eqref{e.univconstants}\in [1, \infty)$ and $\eps_{5}=\eps_{5}\eqref{e.univconstants}\in (0,1)$ such that for any $x$ and any $t \in [\sigma_1 \eps^2 |\log \eps|, \sigma_2 \eps^2 |\log \eps|]$ (for $\sigma_{1}, \sigma_{2}$ as in Lemma \ref{lemma:regtree}) with $\eps<\eps_{5}$, 
\begin{align*}
Q^\eps_x \big[ \exists\, \alpha \in N(t) \text{ such that } |X_t(\alpha) - x|> C \eps | \log \eps| \big] \leq 4\eps^{k+1}.
\end{align*}
\end{lemma}

\begin{proof}
Fix $t \in [\sigma_1 \eps^2 |\log \eps|, \sigma_2 \eps^2 |\log \eps|]$. Let $\mathcal{B}=\mathcal{B}_{t}$ denote the $\sigma$-algebra associated to the non-spatial behaviour of $(\W_s, 0\leq s<t)$, i.e. the branching process without spatial motion up to time $t$. We remark that we can construct $\W_t$ by first sampling the branching structure then adding spatial information. We write $\mathcal{T}$ to denote the time-labelled tree associated to the branching structure and remark that, once spatial information is added, $\mathcal{T} = \mathcal{T}(\W_t)$. 

Recall the index set $N(t)$ of individuals alive at time $t$ and note that $N(t) \in \mathcal{B}$. Define the event $\mathcal{R} := \{ \mathcal{T}(\W_t) \subseteq T^{\text{reg}}_{A_2 |\log \eps|} \}$ where $A_{2}$ is defined in Lemma \ref{lemma:regtree}. For each $\alpha \in N(t)$, we decompose the location $X_t(\alpha)$ as the sum of the increments between branch times; this constitutes copies of the Hunt process $Y$ which evolve between branch times, and the displacements introduced at branch times. For $\beta \in \mathcal{T}$, let $t_\beta^{\text{birth}}$ and $t_\beta^{\text{death}}$ respectively denote the birth and death time of $\beta$ (with $t_\beta^{\text{death}} := t$ if $\beta \in N(t)$, and $t^{\text{birth}}_\emptyset := 0$). 

Given $\mathcal{T}$, starting with the root we may inductively sample the paths and branching displacements along the entire tree up to time $t$. In particular, we denote the trajectory of the root before its branch time by $(Y^\emptyset_s , 0\leq s < t^{\text{death}}_\emptyset)$ with law $P^Y_x$. Let $(\xi^\emptyset_i)_{i\in [N_{0}]}$ denote the displacements of the $N_{0}$ offspring from $Y^\emptyset_{t_{\emptyset}^{\text{death}}-}$. We then inductively define the trajectories of each individual that lives until time $t$. If $\alpha$ is the $i$th child of $\beta$, then given $Y^{\beta}_{t_{\beta}^{\text{death}}}$ and $\xi^\beta_i$, the motion of $\alpha$ during its lifetime may be defined as $(Y^\alpha_s: t^{\text{birth}}_\alpha\leq s \leq t^{\text{death}}_\alpha)$. The displacements of the offspring of $\alpha$ are then $\xi^\alpha_i$, $i=1,\dots,N_0$, sampled according to $\mu^\eps_{Y^{\al}_{t_{\al}^\text{death}-}}$. In this fashion, we may decompose the locations of all individuals in $N(t)$ as follows: for $\alpha \in N(t)$ in generation $|\alpha|$, let $\alpha_0, \dots, \alpha_{|\alpha|}$ be the sequence of ancestors of $\alpha$, so that $\alpha_0 = \emptyset$, $\alpha_{|\alpha|} = \alpha$, and $\alpha_{i+1}$ is the $j_i$th child of $\alpha_i$, for some $j_i \in [N_0]$, for $i = 0,\dots,|\al|-1$. Then by the construction above,
\begin{align} \label{eq:trajdecomp}
X_t(\alpha) &=x +  \left(Y^\alpha_t - Y^\alpha_{t^{\text{birth}}_\alpha} \right) + \sum_{i=0}^{|\alpha|-1} \left(\xi^{\alpha_i}_{j_i} + Y^{\alpha_i}_{t^{\text{death}}_{\alpha_i}-} - Y^{\alpha_i}_{t^{\text{birth}}_{\alpha_i}} \right) \notag
\\ &= x + \sum_{i=0}^{|\alpha|-1} \xi^{\alpha_i}_{j_i} + \sum_{i=0}^{|\alpha|} \left( Y^{\alpha_i}_{t^{\text{death}}_{\alpha_i}} - Y^{\alpha_i}_{t^{\text{birth}}_{\alpha_i}} \right),
\end{align}
where in the final expression we recall that $t^{\text{death}}_\alpha := t$ for $\alpha \in N(t)$, and $t^{\text{birth}}_\emptyset = 0$. We have also changed $Y^{\alpha_i}_{t^{\text{death}}_{\alpha_i}-}$ to $Y^{\alpha_i}_{t^{\text{death}}_{\alpha_i}}$. This is permitted because $Y$ is a Hunt process, and conditionally on $\mathcal{B}$, we may view the death times as fixed times, and hence $Y$ is left continuous at these times.

By \eqref{a.1}, conditional on $\mathcal{B}$, for each $\beta \in \mathcal{T}$, we may couple the associated increment of $Y^\beta$ to a Brownian motion as follows: there is a Brownian motion $(W^{\mathbf{d},\beta}_t , 0\leq t \leq t_*(\beta))$  started at $0$, where $t_*(\beta) = t^{\text{death}}_\beta - t^{\text{birth}}_\beta$, such that, for 
\begin{equation*}
\Delta_\beta := W^{\mathbf{d},\beta}_{t_*(\beta)} - \left( Y^{\beta}_{t^{\text{death}}_{\beta}} - Y^{\beta}_{t^{\text{birth}}_{\beta}} \right), 
\end{equation*}
condition \eqref{a.1} yields
\begin{equation} \label{eq:Deltabetabd}
Q^\eps_x [|\Delta_\beta| > \eps^{k+2} \, | \, \mathcal{B}]  \leq \bar{C}e^{-\bar{c}\eps^\eta}.
\end{equation}
Notice that this was possible because $t_{*}(\beta)<t\leq \sigma_{2}\eps^{2}|\log \eps|<\eps^{2}|\log \eps|^{2}$ for $\eps$ sufficiently small.

From \eqref{eq:trajdecomp}, we obtain
\begin{align} \label{eq:trajdecomp2}
|X_t(\alpha) - x| \leq \sum_{i=0}^{|\alpha|-1} |\xi^{\alpha_i}_{j_i}| + \sum_{i=0}^{|\alpha|} |\Delta_{\alpha_i}| + \left| \sum_{i=0}^{|\alpha|} W^{\mathbf{d},\al_{i}}_{t_*(\alpha_i)} \right|.
\end{align}
We will now argue that all the terms in \eqref{eq:trajdecomp2} are small with high probability, simultaneously for all $\alpha \in N(t)$. 

Let $C \geq 1$ be a constant whose value will be fixed later, and may vary from line to line in the following computations. We remark that, for each $\alpha \in N(t)$, the sum of Brownian increments in \eqref{eq:trajdecomp2}, i.e.
\[\sum_{i=0}^{|\alpha|} W^{\mathbf{d},\al_{i}}_{t_*(\alpha_i)}\]
is equal in law to a Brownian motion run to time $t \leq \sigma_2 \eps^2|\log \eps|$. On $\mathcal{R}$, $|N(t)| \leq N_0^{A_2 |\log \eps|}$, and hence there are at most that many Brownian motions run to time at most $\sigma_2 \eps^2|\log \eps|$. By a union bound, this implies that for all $\eps<\eps_{4}$
\begin{align} \label{eq:unifBMbd}
Q^\eps_x\bigg[ \exists \alpha \in N(t) :  \Big| \sum_{i=0}^{|\alpha|} W^{\mathbf{d},\al_{i}}_{t_*(\alpha_i)} \Big|& > C \eps|\log \eps|  \, \bigg| \, \mathcal{B} \bigg] \indc_{\mathcal{R}} \notag \\
&\leq N_0^{A_2 |\log \eps|} P^{\bmd}_0[|\bmd_{\sigma_2 \eps^2 |\log \eps|}| \geq C\eps|\log \eps|] \notag
\\ &=  N_0^{A_2 |\log \eps|} P^{\bmd}_0[|\bmd_{1}| \geq C|\log \eps|^{1/2}]\notag
\\ &\leq N_0^{A_2 |\log \eps|} \frac{1}{\sqrt{2\pi}} e^{-C |\log \eps|/2} \notag
\\&  \leq \eps^{k+1},
\end{align} 
where the equality in the middle line uses the scaling of Brownian motion, and the final inequality holds by choosing a sufficiently large value of $C$, in terms of $N_{0}$, $k$, and $\eps_{4}$. In particular, $C=C\eqref{e.univconstants}$. 

Recall that $(\xi^\beta_1,\dots,\xi^\beta_{N_0})$ is distributed according to $\mu^\eps_{y_\beta}$, where $y_\beta = Y^\beta_{t^{\text{death}}_\beta}$. By \eqref{a.2}, irrespective of the parents' location, we have $|\xi^\beta_i| \leq \eps^{k+2}$ for all $i\in [N_0]$ with probability at least $1 - \bar{C}e^{-\bar{c}\eps^\eta}$. Recalling that $\mathcal{T}(\W_t)$ has height at most $\lfloor A_2 |\log \eps| \rfloor$ on $\mathcal{R}$, we have
\begin{align} \label{eq:unifdisplacementbd}
Q^\eps_x\big[ \exists \beta \in \mathcal{T}(\W_t), i \in [N_0] : &|\xi^\beta_i| > \eps^{k+2} \, \big| \, \mathcal{B} \big]\indc_{\mathcal{R}}\notag \\
&\leq \left( 1 + N_0 + \cdots + N_0^{\lfloor A_2|\log \eps| \rfloor - 1} \right) \bar{C}e^{-\bar{c}\eps^\eta} \notag
\\& \leq \left(\frac{1}{1-N_0^{-1}}\right) N_0^{A_2 |\log \eps|}  \bar{C}e^{-\bar{c}\eps^\eta} \notag
\\& \leq \eps^{k+1},
\end{align}
where the first inequality follows from the union bound and the number of internal vertices of $T^{\text{reg}}_{A_2 |\log \eps|}$, and the second simply uses the value of the geometric series. The last inequality is obvious by taking $\eps_{4}$ sufficiently small, depending on all of the constants in \eqref{a.2} and $N_{0}$. 

Finally we handle the $\Delta_\beta$ terms in \eqref{eq:trajdecomp2}. On $\mathcal{R}$, by a calculation similar to \eqref{eq:unifdisplacementbd} there are at most $(1-N_0^{-1})^{-1} N_0^{A_2 |\log \eps| + 1}$ vertices in $\mathcal{T}(\W_t)$. From  \eqref{eq:Deltabetabd} and a union bound, we obtain 
\begin{align} \label{eq:unifincbd}
Q^\eps_x\big[ \exists \beta \in \mathcal{T}(\W_t) : |\Delta_\beta| > \eps^{k+2}\, \big| \, \mathcal{B} \big]\indc_{\mathcal{R}} &\leq  (1-N_0^{-1})^{-1} N_0^{A_2 |\log \eps|+1}   \bar{C}e^{-\bar{c}\eps^\eta} \notag
\\ &\leq \eps^{k+1}
\end{align}
for $\eps$ sufficiently small, exactly as above. 

Now suppose that, conditional on $\mathcal{B}$, for all $\beta \in \mathcal{T}(\W_t)$, $|\Delta_\beta| \leq \eps^{k+2}$ and $|\xi^\beta_i| \leq \eps^{k+2}$ for all $i \in [N_0]$. Then on the event $\mathcal{R}$, for every $\alpha \in N(t)$, since $|\alpha| \leq A_2|\log \eps|$, for $C$ as above, 
\begin{align*}
&\sum_{i=0}^{|\alpha|-1} |\xi^{\alpha_i}_{j_i}| \leq |\alpha| \eps^{k+2} \leq A_2 |\log \eps| \eps^{k+2} \leq C \eps|\log \eps|, \, \text{and} \\
&\sum_{i=0}^{|\alpha|} |\Delta_{\alpha_i}| \leq (|\alpha| + 1)\eps^{k+2} \leq (A_2|\log \eps| + 1)\eps^{k+2} \leq C\eps|\log \eps|.
\end{align*}
Combining this with \eqref{eq:unifdisplacementbd} and \eqref{eq:unifincbd}, this implies that 
\begin{equation}\label{e.bd2}
\begin{cases}
&Q^{\eps}_{x}\Big[\exists \al\in N(t): \sum_{i=0}^{|\alpha|-1} |\xi^{\alpha_i}_{j_i}|>C\eps|\log\eps|\Big| \mathcal{B}\Big]\indc_{\mathcal{R}} \leq \eps^{k+1}\\
&\quad\quad\quad\quad\text{and}\\
&Q^{\eps}_{x}\Big[\exists \al\in N(t): \sum_{i=1}^{|\alpha|} |\Delta_{\alpha_i}|>C\eps|\log\eps|\Big| \mathcal{B}\Big]\indc_{\mathcal{R}}\leq \eps^{k+1}.
\end{cases}
\end{equation} 
From the above, and using \eqref{eq:trajdecomp2} combined with \eqref{eq:unifBMbd} and \eqref{e.bd2}, we conclude that 
\begin{align*}
Q^\eps_x\big[ \exists \alpha \in N(t) : |X_t(\alpha)-x| > 3C\eps|\log \eps|\, \big| \, \mathcal{B} \big]\indc_{\mathcal{R}} \leq 3\eps^{k+1}.
\end{align*}
Finally, by Lemma~\ref{lemma:regtree}, if $\eps<\eps_{4}$ is taken sufficiently small to satisfy all of the above estimates, then $Q^\eps_x[ \mathcal{R}^c] < \eps^{k+1}$, and the result follows with constant $C' = 3C=C'\eqref{e.univconstants}$.
\end{proof}

Equipped with these results, we are now ready to prove Proposition \ref{prop:interformgen}.
\begin{proof}[Proof of Proposition~\ref{prop:interformgen}] Fix $t \in [\sigma_1 \eps^2 |\log \eps| , \sigma_2 \eps^2 |\log \eps|]$, and let $A_{1}$ be as in \eqref{eq:gcompose}. Define the events
\begin{align*}
E_1:=\{\mathcal{T}(\W_t) \supseteq T^{\text{reg}}_{A_1  |\log \eps|}\} \,\,\, \text{and}\,\,\, E_2:=\{\forall \, \alpha \in N(t), |X_t(\alpha) - x| \leq C \eps| \log \eps|\},
\end{align*}
where $C$ is the constant from Lemma~\ref{lemma:unifdisplacement}. Let $K_1 = 2C$. Suppose that $d(t,x) \leq - 2C\eps|\log \eps|$. Then on $E_2$, by Proposition~\ref{p.distprop}(i), we have
\begin{equation*}
d(t,X_t(\alpha)) \leq - 2C \eps|\log \eps| + |X_t(\alpha) - x| \leq - C\eps|\log \eps|
\end{equation*}
for all $\alpha \in N(t)$. Since $t \leq \sigma_2\eps^2|\log \eps|$, by Proposition~\ref{p.distprop}(iii), for small enough $\eps$ depending on $h_{0}$, we have
\begin{equation*}
d(0,X_t(\alpha)) \leq -C \eps|\log \eps| + C_1 \sigma_2 \eps^2|\log \eps|.
\end{equation*}
Recall that $p^+(\delta,\phi) = (a + \delta)\indc_{\{\phi \leq 0\}} + b \indc_{\{\phi > 0\}}$, with $\delta \in (0, \delta_{*})$. The above implies that for small enough $\eps$, depending on $C$ and $C_{1}$ (hence depending on $\norm{\phi}_{C^{4}(\overline{B})}$, $k$ and $N_{0}$), the right hand side of the above equation display is negative, and hence we have
\begin{equation} \label{eq:formationevaluation}
p^+(\delta,\phi)(X_t(\alpha)) = a+\delta\, \text{ for all } \alpha \in N(t).
\end{equation}
Observe that since $g(p)<p$ for all $p \in (a,a+\delta]$,  the sequence $g^{(n)}(a+\delta)$ is decreasing in $n$. This implies that if $T_1 \subseteq T_2$, then $V(\emptyset; T_1, a+\delta)$ stochastically dominates $V(\emptyset; T_2, a+\delta)$, where $V(\emptyset; T,a+\delta) := V(\emptyset; T,v(L(T)))$ with $v(L(T))$ chosen to be iid $\text{Bernoulli}(a+\delta)$. In particular, by \eqref{eq:formationevaluation}, on $E_1 \cap E_2$ and conditional on the event $\mathcal{T}(\W_t) = T$ with $T \supseteq T^{reg}_{A_{1}|\log \eps|}$, we have
\[\V(\W_t; p^+(\delta,\phi))\indc_{E_{1}\cap E_{2}}\indc_{\left\{\mathcal{T}(\W_{t})=T\right\}} = E[V(\emptyset; T, a + \delta)] \leq E[V(\emptyset; T^{\text{reg}}_{A_1|\log \eps|}, a+ \delta)],\]
where the expectation is over a collection of independent Bernoulli$(a+\delta)$ leaf votes. We remark that by \eqref{eq:gcompose},
\begin{align*}
E[V(\emptyset; T^{\text{reg}}_{A_1|\log \eps|}, a+ \delta)] = g^{(\lfloor A_1|\log \eps|\rfloor)}(a+\delta) \leq a + \eps^{k+1}.
\end{align*}
In particular, since on $E_{1}\cap E_{2}$, we must have $\mathcal{T}(\W_t) = T$ with $T\supseteq T^{reg}_{A_{1}|\log \eps|}$, this implies 
\begin{equation*}
Q^\eps_x [\V(\W_t; p^+(\delta,\phi)) = 1, E_1 \cap E_2] \leq a + \eps^{k+1}.
\end{equation*}
By Lemmas~\ref{lemma:regtree} and \ref{lemma:unifdisplacement}, $Q^\eps_x[E_1^c \cup E_2^c] \leq 5\eps^{k+1}$. Hence, 
\begin{equation*}
Q^\eps_x [\V(\W_t; p^+(\delta,\phi))= 1] \leq a + 6\eps^{k+1},
\end{equation*}
and the proof is complete. 
\end{proof}

\subsection{Propagation of the interface} \label{ss:propagation}
We again collect here some constants and display their dependencies:
\begin{equation}\label{e.constants}
\begin{cases}
\text{$\al\in (0,1)=$ the parameter in the statement of \eqref{j.4},}\\
\text{$L = L(\norm{\phi}_{C^{4}(\overline{B})})\in [1, \infty)$ from Corollary \ref{corollary:coupling},}\\
\text{$h_0 = h_0(\norm{\phi}_{C^{4}(\overline{B})})\in (0,1)$ introduced in Proposition \ref{p.distprop},}\\
\text{$c_{0}=c_{0}(\norm{g}_{C^{2}([0,1])})$ as in \eqref{g.5}},\\
c_\gamma=c_{\ga}(k):= k\gamma^{-1},\\
\text{$c_{1}=c_{1}(k+1)$ from Theorem~\ref{thm:bbminterface}, where the bounds hold with $k+1$,}\\
\text{$\sigma_{1}\in [1, \infty), \sigma_{2}=\sigma_{1}+c_\ga$ as defined in Proposition \ref{prop:interformgen}.}
\end{cases}
\end{equation}
We also define the new constants
\begin{equation}\label{e.newconstants}
\begin{cases}
m_0=m_{0}(\norm{\phi}_{C^{4}(\overline{B})},\al):= 16L/\alpha,\\
C_*=C_{*}(\norm{\phi}_{C^{4}(\overline{B})},\al):= 2(3+m_0 \gamma)/c_{0},\\
m_1 = m_{1}(\norm{\phi}_{C^{4}(\overline{B})},\al)=8C_*Lc_2/(\delta_*\alpha).
\end{cases}
\end{equation}

The next two results are the key results which establish that the interface asymptotically propagates according to MCF. The first is a crucial lemma which allows us to compare the one-dimensional dual process started from different initial conditions which are calculated in terms of Brownian motions in different dimensions. It is similar to \cite[Lemma~2.18]{EFP2017} and \cite[Lemma~3.12]{DH2021}. The second result, Proposition~\ref{prop:propagationgen}, is also analogous in spirit to the approach of \cite{EFP2017, DH2021}. Those papers use a proof by contradiction to argue that the interface propagates according to MCF, whereas we have reformulated this as an induction.
\begin{lemma} 
\label{lemma.gfunctioncoupling.gen}
There exists $\eps_6=\eps_{6}(\norm{\phi}_{C^{4}(\overline{B})}, \al, \text{\ref{e.univconstants}}) \in (0,1)$ such that for every $x$, $t \in (c_\gamma \eps^2 |\log \eps|,h_0]$, constant $K=K\eqref{e.univconstants}>0$, and $s \in [m_0 \eps^{k+2}, c_\gamma \eps^2 |\log \eps|]$, if $\eps<\eps_6$, 
\begin{align} \label{eq.gfunctioncoupling_gen}
&E^{\bmd}_x\left[g\left( P^\eps_{d(t-s,\bmd_s) + K \eps|\log \eps| + 2\eps^{k+2}}\left[\V(\mathbf{B}_{t-s}) = 1\right] + C_*\eps^{k} \right)\right] \notag
\\& \hspace{2 cm} \leq E^{\bmo}_{0}\left[g\left( P^\eps_{d(t,x) + K \eps|\log \eps| +\bmo_s}\left[\V(\mathbf{B}_{t-s}) = 1\right]\right)\right] + (1-c_0/2)C_*\eps^{k} \notag
\\ &\hspace{2 cm} \quad  + C_*\eps^{k} \cdot \norm{g}_{C^{1}([0,1])} \indc_{\left\{s 
\leq m_1 \eps^{k+1}|\log \eps| \right\}},
\end{align}
where the constants $c_{\ga}, h_{0}, m_{0}, C_{*}, m_{1}$ are defined in \eqref{e.constants} and \eqref{e.newconstants}. 
\end{lemma}
\begin{remark}
The constant $K$ in the statement of Lemma \ref{lemma.gfunctioncoupling.gen} is in preparation for how we will apply the result subsequently in the argument. Assuming that $K=K(\text{\ref{e.univconstants}})$, this introduces the dependence of $\eps_{6}$ on \eqref{e.univconstants}.
\end{remark}

\begin{proof}[Proof of Lemma~\ref{lemma.gfunctioncoupling.gen}] 

Recall the constant $c_1=c_{1}(\text{\ref{e.univconstants}},k+1)$ in \eqref{e.constants} from Theorem~\ref{thm:bbminterface}. We fix a constant $c_* \geq 3(c_1 + K)$, whose value we may adjust as necessary.

We begin by stating a displacement bound for Brownian motion. For $s \leq c_\gamma \eps^2 |\log \eps|$, by scaling, we have
\begin{align*}	
P^{\bmo}_0 \left[\sup_{r \leq s} |\bmo_{r}| > 2^{-1}c_* \eps |\log \eps|\right] &\leq P^{\bmo}_0 \left[\sup_{r \leq 1} \,(\eps^2|\log\eps|)^{1/2}|\bmo_{r}| > 2^{-1}c_* \eps |\log \eps|\right]
\\ &= P^{\bmo}_0 \left[\sup_{r \leq 1} |\bmo_r| \geq 2^{-1}c_* |\log \eps|^{1/2} \right]
\\& \leq 2(2\pi)^{-1/2} \exp (- 8^{-1}(c_*)^{2}|\log \eps|), 
\end{align*}
where the last line follows from the reflection principle. A similar computation holds in higher dimensions by considering the coordinates individually. We may enlarge $c_*$ as needed (depending on dimension $\mathbf{d}$ and $k$) to obtain that, for sufficiently small $\eps>0$, depending on $\mathbf{d}$, $k$, $c_{1}$, and $K$, for all $s \leq  c_\gamma \eps^2|\log \eps|$,
\begin{equation} \label{e_Brownianmaxbd_gen}
\begin{aligned}
&P^{\bmo}_0\Big[\sup_{r \leq s}| \bmo_r| > 2^{-1}c_* \eps |\log \eps| \Big] \leq \eps^{k+1}\\
&P^{\bmd}_x\Big[\sup_{r \leq s} |\bmd_r - x| > 2^{-1}c_* \eps |\log \eps|\Big] \leq \eps^{k+1}.
\end{aligned}
\end{equation}

Now define
\[G_\eps(t,x) := E^{\bmd}_x\left[g\left( P^\eps_{d(t-s, \bmd_s) + K\eps|\log \eps| + 2\eps^{k+2}}\left[\V(\mathbf{B}_{t-s}) = 1\right] + C_*\eps^{k} \right)\right].\]
We consider the three following cases separately:
\begin{itemize}
\item[(i)] $d(t,x) \leq -c_*\eps |\log \eps|$,
\item[(ii)] $d(t,x) \geq c_* \eps |\log \eps|$,  
\item[(iii)] $|d(t,x)| \leq c_* \eps |\log \eps|$. 
\end{itemize}
Let us first consider case (i), so that $d(t,x) \leq - c_* \eps |\log \eps|$. In this case, we will verify that \eqref{eq.gfunctioncoupling_gen} holds by proving that
\begin{equation}\label{e.case1}
G_{\eps}(t,x)\leq a+ (1-c_0/2)C_* \eps^{k}.
\end{equation}
Indeed, by \eqref{eq:1diminterval} and the fact that $g$ is increasing by \eqref{g.3}, the right hand side of \eqref{eq.gfunctioncoupling_gen} is clearly larger than the right hand side of the above.

By Proposition~\ref{p.distprop}(i) and (iii), for $\eps>0$ sufficiently small such that $s\leq c_\gamma \eps^2 |\log \eps| \leq \tau_{0}=\tau_{0}(\norm{\phi}_{C^{4}(\bar{B})})$, for $C_{1}=C_{1}(\norm{\phi}_{C^4(\overline{B})})$,
\begin{align}\label{e_distancebd_triangle_gen}
|d(t-s,\bmd_s) - d(t,x)| &\leq |\bmd_s - x| + \sup_y |d(t-s,y) - d(t,y)| \notag
\\ &\leq |\bmd_s - x| + C_1 s.
\end{align}
Using $g\leq 1$, $g$ is increasing, the monotonicity of %$P^{\eps}_{(\cdot)}$
$z \mapsto P^{\eps}_z[\V(\B_t)=1]$, i.e. \eqref{e:BBMmonotone}, and \eqref{e_Brownianmaxbd_gen}, we have
\begin{align} \label{e:glemmapf1_gen} 
&G_\eps(t,x) \notag 
\\ & \leq  E^{\bmd}_x\Big[g\big(P^\eps_{d(t-s,\bmd_s) + K\eps|\log \eps| + 2\eps^{k+2}}\left[\V(\mathbf{B}_{t-s})=1\right] + C_*\eps^{k} \big)\indc_{\left\{|\bmd_s - x| \leq 2^{-1}c_* \eps |\log \eps|\right\}}\Big]\notag
\\&  \quad+ P^W_x\left[|\bmd_s - x| > 2^{-1}c_* \eps |\log \eps|\right]\notag
\\& \leq g\left(P^\eps_{d(t,x) + (2^{-1}c_*+K) \eps|\log \eps|+ C_1 s + 2\eps^{k+2}}\left[\V(\mathbf{B}_{t-s})=1\right] + C_* \eps^{k}	\right)+ \eps^{k+1}.
\end{align} 
Next, because $d(t,x)\leq -c_* \eps|\log\eps|$ with $c_* \geq 3(c_1 + K )$ and $s \leq c_\gamma \eps^2 |\log \eps|$, we have that for sufficiently small $\eps$, depending on $C_{1}, c_{\ga}$,
  \begin{eqnarray*}
  \lefteqn{
d(t,x) + \left( \frac{c_*}{2} + K\right) \eps|\log \eps|+C_1 s+ 2\eps^{k+2}}\\
& & \leq  - \left(\frac{3}{2}c_1 + \frac{1}{2} K \right)\eps |\log \eps| + C_1 \cdot c_\gamma \eps^2 |\log \eps| + 2\eps^{k+2} \\
& &  \leq -c_1\eps |\log \eps|.
\end{eqnarray*}
Hence by Theorem~\ref{thm:bbminterface} and \eqref{e:BBMmonotone},
\begin{equation*}
P^\eps_{d(t,x) + (2^{-1}c_*+K) \eps|\log \eps|+ C_1 s + 2\eps^{k+2}}\left[\V(\mathbf{B}_{t-s})=1\right] \leq a + \eps^{k+1}. 
\end{equation*}
Continuing from \eqref{e:glemmapf1_gen} and using that $g$ is increasing, we obtain that
\begin{align*}
G_\eps(t,x) \leq g(a + \eps^{k+1} + C_*\eps^{k}) + \eps^{k+1}.
\end{align*}
By \eqref{g.5} and the mean value theorem, for sufficiently small $\eps$ depending on $c_{0}$ and $C_*$, we have
\begin{equation*}
g(a + \eps^{k+1} + C_*\eps^{k}) \leq a + (1-c_0)(\eps^{k+1}+ C_*\eps^{k}) \leq a + (1-2c_0/3)C_* \eps^{k}, 
\end{equation*}
and combining this with the above yields \eqref{e.case1}. 

Next we handle case (ii), when $d(t,x)\geq c_* \eps |\log \eps|$. In this case, we will verify that \eqref{eq.gfunctioncoupling_gen} holds by proving that
\begin{equation}\label{e.case2}
b + (1-c_0)C_*\eps^{k}\leq E^{\bmo}_0\left[g(P^\eps_{d(t,x) + K\eps|\log \eps| + \bmo_s}[\V(\mathbf{B}_{t-s})=1])\right]+ (1-c_0/2)C_*\eps^{k} .
\end{equation}
Indeed, by the monotonicity of $g$ and \eqref{eq:1diminterval}, the left hand side of \eqref{eq.gfunctioncoupling_gen} is at most $g(b+C_{*}\eps^{k})$. By \eqref{g.5} and the mean value theorem, $b+(1-c_{0})C_{*}\eps^{k}\geq g(b+C_{*}\eps^{k})$ for $\eps$ sufficiently small, depending on $\delta_{*}$ and $k$, and \eqref{eq.gfunctioncoupling_gen} follows from \eqref{e.case2}. 

To prove \eqref{e.case2}, we begin by observing that
\begin{multline*}
E^{\bmo}_0\left[g(P^\eps_{d(t,x) + K\eps|\log \eps| + \bmo_s}[\V(\mathbf{B}_{t-s})=1])\right] 
\\ 
\begin{aligned}
& \geq E^{\bmo}_0\left[g(P^\eps_{d(t,x) +\bmo_s}[\V(\mathbf{B}_{t-s}) = 1]) \indc_{\left\{|\bmo_s| \leq 2^{-1}c_* \eps|\log \eps|\right\}}\right]\\
 & \geq E^{\bmo}_0\left[g(P^\eps_{\frac{3}{2}c_1\eps|\log\eps|}[\V(\mathbf{B}_{t-s}) = 1]) \indc_{\left\{|\bmo_s| \leq 2^{-1}c_* \eps|\log \eps|\right\}})\right]
\\ & \geq g\left(P^\eps_{\frac{3}{2}c_1\eps|\log\eps|}\left[\V(\mathbf{B}_{t-s}) = 1\right]\right)(1 - \eps^{k+1}),
\end{aligned}
\end{multline*}
where we used monotonicity \eqref{e:BBMmonotone}, the assumption of case (ii), the fact that $c_{*}>3c_{1}$, and \eqref{e_Brownianmaxbd_gen}. By Theorem~\ref{thm:bbminterface}, the argument of $g$ in the last line above is at least $b - \eps^{k+1}$. In particular, by \eqref{g.5}, we have
\begin{align*}
E^{\bmo}_0\left[g(P^\eps_{d(t,x) + K\eps|\log \eps| + \bmo_s}[\V(\mathbf{B}_{t-s})=1])\right] &\geq g(b-\eps^{k+1})(1 - \eps^{k+1})\\
&\geq (b-(1-c_0)\eps^{k+1})(1 - \eps^{k+1})\\
&\geq b - (1+b)\eps^{k+1} \geq b-2\eps^{k+1}. 
\end{align*}
Hence, we obtain that for $\eps$ sufficiently small, depending on $c_{0}$ and $C_{*}$,
\begin{align*}
E^{\bmo}_0&\left[g(P^\eps_{d(t,x) + K\eps|\log \eps| + \bmo_s}[\V(\mathbf{B}_{t-s})=1])\right]+ (1-c_0/2)C_*\eps^{k}\\
 &\quad\quad\quad\geq  b - 2\eps^{k+1}+ (1-c_0/2)C_*\eps^{k}\\
&\quad\quad\quad \geq b+(1-c_{0})C_{*}\eps^{k}
\end{align*}
as desired for \eqref{e.case2}.

Finally, consider case (iii), when $|d(t,x)| \leq c_* \eps|\log \eps|$. Let $r_0=r_{0}(\norm{\phi}_{C^{4}(\bar{B})})>0$ and $Q_{h_0, r_{0}}$ be as in Proposition~\ref{p.distprop}, and note that for sufficiently small $\eps$, depending on $c_{1}$ and $K$, $(t,x) \in Q_{h_0,r_0}$. For a Brownian motion $\bmd_r$ started at $x$, define the  event \[A_{r_0,s} = \left\{|d(t-r,\bmd_r)| \leq r_0 \text{ for all } r \in [0,s]\right\}.\]
Again arguing as in \eqref{e_distancebd_triangle_gen}, and using the fact that $s \leq c_\gamma \eps^2 |\log \eps|$, it follows that for sufficiently small $\eps$, $\{\sup_{r \leq s} |\bmd_r -x| \leq 2^{-1}c_* \eps |\log \eps|\} \subseteq A_{r_0,s}$. Hence, by \eqref{e_Brownianmaxbd_gen},
\begin{align*}
G_\eps(t,x) \leq  E^{\bmd}_x\Big[g\big(P^\eps_{d(t-s,\bmd_s) + K \eps|\log \eps| + 2\eps^{k+2}}\left[\V(\mathbf{B}_{t-s})=1\right] + C_*\eps^{k} \big)\indc_{A_{r_0,s}}\Big]  + \eps^{k+1}.
\end{align*}
By Corollary~\ref{corollary:coupling}, there exists a one-dimensional Brownian motion $(\bmo_r, r \geq 0)$ such that, on $A_{r_0,s}$,
\begin{equation*}
d(t-r,\bmd_r) \leq d(t,x) + \bmo_r - \frac{\alpha r}{4L}\,\, \text{ for all } r \in [0,s]. 
\end{equation*}
By monotonicity as in \eqref{e:BBMmonotone}, it follows that, on $A_{r_0,s}$,
\begin{align*}
P^\eps_{d(t-s,\bmd_s) + K \eps|\log \eps| + 2\eps^{k+2}}&\left[\V(\mathbf{B}_{t-s})=1\right]\\
&\leq P^\eps_{d(t,x) + K \eps|\log \eps| + \bmo_s - \frac{\alpha s}{4L} + 2\eps^{k+2}}\left[\V(\mathbf{B}_{t-s})=1\right].
\end{align*}
Since the right hand side of the above only depends on the one-dimensional Brownian motion, we may apply this bound and write the resulting expression as an expectation under one-dimensional Brownian motion, i.e.
\begin{align*}
G_\eps(t,x) \leq  E^{\bmo}_{0}\left[g\left(P^\eps_{d(t,x) + K\eps|\log \eps| + \bmo_s  - \frac{\alpha s}{4L} + 2\eps^{k+2}}\left[\V(\mathbf{B}_{t-s})=1\right] + C_*\eps^{k} \right)\right]  + \eps^{k+1}.
\end{align*}
By assumption, $s \geq 16L \eps^{k+2}/\alpha$, so in particular we have 
\begin{equation} \label{e.xyz_gen}
G_\eps(t,x) \leq E^{\bmo}_0[g(p_\eps + C_*\eps^{k})] + \eps^{k+1},
\end{equation}
where
\[p_\eps := P^\eps_{d(t,x) +K\eps|\log\eps| + \bmo_s - \frac{\alpha s}{8L}}\left[\V(\mathbf{B}_{t-s})=1\right].\]
We further subdivide into two cases: 
\begin{equation*}
\text{(a)}\,\,  |p_\eps - \avg| \leq \avg - a - \delta_* \quad\text{or}\quad \text{(b)}\,\,  |p_\eps - \avg| >  \avg - a - \delta_*.
\end{equation*}
First consider (a), in which case, since $s \leq c_{\ga}\eps^{2}|\log\eps|$, and $b-\mu=\mu-a$, we can apply Lemma~\ref{lemma:bbmslope}, with 
\[\rho = d(t,x) +K\eps|\log\eps| +\bmo_s - \frac{\alpha s}{8L} , \quad \rho' = d(t,x) + K\eps|\log\eps| + \bmo_s.\]
We obtain that for sufficiently small $\eps$, 
\begin{equation*}
p_\eps \leq P^\eps_{d(t,x) +K \eps|\log \eps| +\bmo_s}\left[\V(\mathbf{B}_{t-s}) = 1\right] - \frac{\delta_* \alpha s}{8 L c_2 \eps | \log \eps|} .
\end{equation*}
Hence, by monotonicity and Lipschitz continuity of $g$, we have 
\begin{align*}
&g(p_\eps + C_*\eps^{k})
\\ &\quad\leq g \Big(P^\eps_{d(t,x) +K\eps|\log \eps| + \bmo_s}\left[\V(\mathbf{B}_{t-s}) = 1\right] - \frac{\delta_* \alpha s}{8L c_2 \eps | \log \eps|} + C_* \eps^{k} \Big)
\\ &\quad\leq g\left(P^\eps_{d(t,x) + K\eps|\log \eps| + \bmo_s}\left[\V(\mathbf{B}_{t-s}) = 1\right]\right)  + C_*\eps^{k} \cdot \|g'\|_\infty  \indc_{\left\{\frac{\delta_* \alpha s}{8L c_2 \eps | \log \eps|} \leq C_* \eps^{k} \right\}}.
\end{align*}
The above combined with \eqref{e.xyz_gen} implies \eqref{eq.gfunctioncoupling_gen} in case (iii)(a).

Now consider case (iii)(b). By \eqref{eq:1diminterval}, $p_\eps \in [a,b]$. Given this and \eqref{g.1}, the condition $|p_\eps - \avg| > \avg - a - \delta_*$ is equivalent to having either $p_\eps \in [a,a+\delta_*)$ or $p_\eps \in (b-\delta_*,b]$. For all such $p_\eps$, by \eqref{g.5}, we have that for $\eps$ sufficiently small, in terms of $\delta_{*}$, 
\begin{equation*} \label{eq:caseiiig}
g(p_\eps + C_*\eps^{k})\leq g(p_\eps) + (1-2c_0/3)C_* \eps^{k}. 
\end{equation*}
By substituting the above into \eqref{e.xyz_gen} and applying monotonicity, we obtain 
\begin{align*}
G_{\eps}(t,x)&\leq E^{\bmo}_{0}[g(p_{\eps})]+ (1-2c_0/3)C_* \eps^{k}+\eps^{k+1}\\
&\leq E^{\bmo}_{0}\left[g\left(P^\eps_{d(t,x) +K\eps|\log\eps| + \bmo_s}\left[\V(\mathbf{B}_{t-s})=1\right]\right)\right]+(1-c_0/2)C_* \eps^k,
\end{align*}
which implies \eqref{eq.gfunctioncoupling_gen}. Thus \eqref{eq:caseiiig} holds for all the claimed values of $p_\eps$, and the lemma is proved.
\end{proof}

We may now prove the following inductive result, which allows us to bootstrap our bounds over time. 
\begin{proposition} \label{prop:propagationgen} There exists $\eps_{7}=\eps_{7}(\text{\ref{e.univconstants}}, \eqref{e.constants})\in (0,1)$ such that the following holds. Let $t_0 \in (0,h_0)$ and $K=K\eqref{e.univconstants}>0$ and suppose that there exists $\bar{\eps}\in (0,1)$ such that $t_0 > c_\gamma \bar{\eps}^2 |\log \bar{\eps}|$, and for all $\eps \leq \bar{\eps}$, and $t \in [t_0 - c_\gamma \eps^2 |\log \eps|,t_0)$, 
\begin{equation}
\label{prop_propagation_inequality_gen}
\sup_x \left(Q^\eps_x\left[\V(\X_t; p^+(\phi,\delta)) = 1\right]-P^\eps_{d(t,x) + K \eps|\log\eps|}\left[\V(\mathbf{B}_t)=1\right]\right)\leq C_*\eps^{k},
\end{equation}
for $c_{\ga}$ as in \eqref{e.constants} and $C_{*}$ as in \eqref{e.newconstants}. Then for all $\eps \leq \bar{\eps} \wedge \eps_7$, \eqref{prop_propagation_inequality_gen} holds for all $t \in [t_0, (t_0 + m_0 \eps^{k+2}) \wedge h_0)$, where $m_{0}$ is defined in \eqref{e.newconstants}. 
\end{proposition}

\begin{proof} For convenience, write
\begin{align*}
J_\eps(t,x) &:= Q^\eps_x\left[\V(\X_t; p^+(\phi,\delta)) = 1\right].
\end{align*}
Suppose that \eqref{prop_propagation_inequality_gen} holds for all $t \in [t_0 - c_\gamma \eps^2 |\log \eps|, t_0)$, for $\eps \leq \bar{\eps}$. We will show that for all $\ve$ sufficiently small, \eqref{prop_propagation_inequality_gen} also holds at $t_1 \in [t_0, (t_0 + m_0 \eps^{k+2}) \wedge h_0)$ with the same constant $C_{*}$. Thus, for $\eps$ smaller than a threshold whose value is chosen later, we fix $t_1 \in [t_0, (t_0 + m_0 \eps^{k+2}) \wedge h_0)$. 

Fix $x_0$ and let $I := [m_0 \eps^{k+2}, c_\gamma \eps^2 |\log \eps|]$. Under $Q^\eps_{x_{0}}$, we denote by $\tau$ the time of the first branching event and write $(Y_t, 0\leq t <\tau)$ to denote the trajectory of the root individual up to its branch time. We remark that since $t_0 > c_\gamma \bar{\eps}^2 |\log \bar{\eps}|$, $\{\tau \in I\} \subset \{\tau < t_1\}$. Consequently, from Lemma~\ref{lemma:gfunMarkov}, we have
\begin{align} \label{eq:propproofgfun}
Q^\eps_{x_0}\big[\V(\W_{t_1}; p^+(\phi,\delta)) = 1, &\tau \in I \, \big| \, \cF_\tau \big]\notag \\
&= Q^\eps_{x_0}\left[\V(\W_{t_1}; p^+(\phi,\delta)) = 1, \tau < t_1 \, \big| \, \cF_\tau \right] \indc_{\{\tau \in I \}} \notag \\
&=  g(\mathfrak{p_1},\dots,\mathfrak{p}_{N_0})\indc_{\{\tau \in I\}} 
\end{align}
where
\begin{equation*}
\mathfrak{p}_i := Q^\eps_{Y_{\tau} + \xi_i} [\V(\W_{t_1-\tau}; p^+(\phi,\delta)) = 1]
\end{equation*}
and $(\xi_1,\dots,\xi_{N_0})$ are the displacements of the root's offspring sampled from $\mu^\eps_{Y_\tau}$. In particular, 
\begin{align*}
Q^\eps_{x_0}\left[\V(\W_{t_1}; p^+(\phi,\delta)) = 1 \, \big| \, \cF_\tau \right] \leq \indc_{\{\tau \in I \}} E^\eps_{x_0}[g(\mathfrak{p_1},\dots,\mathfrak{p}_{N_0})] + \indc_{\{\tau \not \in I \}}.
\end{align*}
Recall that $\tau$ is exponential with parameter $\gamma \eps^{-2}$. Let $S$ denote an independent Exp($\gamma \eps^{-2}$) random variable with law and expectation $P^S$ and $E^S$. Taking the expectation of the above and using that $\V(\cdot)\leq 1$, we have
\begin{align} \label{eq:firstbranchdecomp_gen}
J_\eps(t_{1},x) & = Q_{x_0}^\eps \otimes E^S \left[ \V(\W_{t_1}; p^+(\phi,\delta)) = 1\right] \notag
\\ & \leq E^\eps_{x_0} \otimes E^S \left[ g(\mathfrak{p_1},\dots,\mathfrak{p}_{N_0}) \indc_{\{S \in I \}}\right] + P^S[S \not \in I] \notag
\\ & \leq E^\eps_{x_0} \otimes E^S \left[ g(\mathfrak{p_1},\dots,\mathfrak{p}_{N_0}) \indc_{\{S \in I \}}\right] + (1 + m_0 \gamma)\eps^{k},
\end{align}
where the last inequality follows because $P^{S}(S \geq c_\gamma \eps^2 |\log \eps|) = \eps^{k}$ (by definition of $c_\gamma$ in \eqref{e.constants}) and $P^S[S < m_0 \eps^{k+2}] \leq m_0 \gamma \eps^{k}$. We abuse notation slightly; in the above expression, $\mathfrak{p}_i$ is redefined with $S$ replacing $\tau$, that is \[\mathfrak{p}_i = Q^\eps_{Y_S + \xi_i} [\V(\W_{t_1-S}; p^+(\phi,\delta)) = 1].\] If $S \in I$, then $t_1 - S \in [t_0 - c_\gamma \eps^2| \log \eps|, t_0]$, and hence by assumption we have
\begin{equation} \label{eq:prop_onedimbd}
\mathfrak{p}_i \leq P^\eps_{d(t_1 - S, Y_S + \xi_i) + K \eps|\log \eps|} [ \V(\B_{t_1 - S}) = 1] + C_*\eps^{k}
\end{equation}
for each $i$. By \eqref{a.1}, we may couple $Y$ with a Brownian motion $\bmd$ started from $x_0$ such that
\begin{equation*}
Q^\eps_{x_0}[E_{1}]:=Q^\eps_{x_0}[|\bmd_s - Y_s| \geq \eps^{k+2}] \leq \bar{C}e^{-\bar{c}\eps^\eta}
\end{equation*}
for all $s \leq c_\gamma \eps^2 |\log \eps|$. Furthermore, by \eqref{a.2},
\[ Q^\eps_{x_0}[E_{2}]:=Q^\eps_{x_0} [|\xi_i|> \eps^{k+2} \, \text{ for all } \, i=1,\dots,N_0] \leq \bar{C}e^{-\bar{c}\eps^\eta}.\]
On $E_{1}^{c}\cap E_{2}^{c}$, we have 
\[ |(Y_{S} +\xi_i) - \bmd_S| \leq 2\eps^{k+2}\]
for all $i \in [N_0]$, and by Proposition~\ref{p.distprop}(i), this implies
\[|d(t_1-S,Y_{S} +\xi_i) - d(t_1-S,\bmd_S)| \leq 2\eps^{k+2}.\]
Hence, by monotonicity, 
\begin{equation*}
\mathfrak{p}_i \leq P^\eps_{ d(t_1-S,\bmd_S)+ K \eps|\log \eps|+2\eps^{k+2}} [ \V(\B_{t_1 - S}) = 1] + C_*\eps^{k}.
\end{equation*}
Hence, from the above and \eqref{eq:prop_onedimbd}, for any $S\in I$, 
\begin{align*}
&Q^\eps_{x_0}\Big[\mathfrak{p}_i >  P^\eps_{d(t_1 - S, \bmd_S) + K \eps|\log \eps| + 2\eps^{k+2}} [ \V(\B_{t_1 - S}) = 1] + C_* \eps^{k}\, \text{ for some } i\in [N_0] \Big] \\&\quad\leq Q^\eps_{x_0}[E_{1}]+Q^\eps_{x_0}[E_{2}]<\eps^{k},
\end{align*}
where we assume that $\eps$ is small enough so that $2\bar{C}e^{-\bar{c}\eps^\eta} < \eps^{k}$. It then follows from \eqref{g.0} and \eqref{eq:firstbranchdecomp_gen} that 
\begin{align} \label{eq:firstbranchdecomp2_gen}
&J_\eps(t_{1},x) 
\\& \leq  E^{\bmd}_{x_0} \otimes E^S \Big[ g \left( P^\eps_{d(t_1-S,\bmd_S) + K\eps|\log \eps| + 2\eps^{k+2}}[\V(\B_{t_1-S})=1]  + C_*\eps^{k} \right)\notag\\
&\qquad \qquad \qquad \indc_{\{m_{0}\eps^{k+2}\leq S \leq c_{\ga} \eps^2 |\log \eps| \}} \Big] \notag
\\ &\qquad + \left(2 + m_0\gamma \right) \eps^{k}. \notag
\end{align}
In the above, we have written the expectation with respect to the $\mathbf{d}$-dimensional Brownian motion $\bmd$, which is justified since the only remaining randomness from $Q^\eps_{x_0}$ is through the Brownian motion coupled to $Y$. We now expand the expectation with respect to $S$ and apply Lemma~\ref{lemma.gfunctioncoupling.gen} as follows:
\begin{align*}
&J_\eps(t_{1},x)
\\ &\leq \int_{m_0 \eps^{k+2}}^{c_\gamma \eps^2|\log \eps|} \gamma \eps^{-2} e^{-\gamma \eps^{-2}s}\times \\
&\qquad \qquad \times E_{x_0}^{\bmd} \left[ g \left( P^\eps_{d(t_1-s,\bmd_s) + K\eps|\log \eps| + 2\eps^{k+2}}[\V(\B_{t_1-s})=1]  + C_*\eps^{k} \right)\right]ds \\& \quad+ \left(2 + m_0\gamma\right) \eps^{k} 
\\ &\leq \int_{m_0 \eps^{k+2}}^{c_\gamma \eps^2|\log \eps|} \gamma \eps^{-2} e^{-\gamma \eps^{-2}s} \left( E^{\bmo}_0 \left[ g \left(P^\eps_{d(t_1,x_0) + K\eps|\log \eps| + \bmo_s} [\V(\B_{t_1-s})=1]\right) \right]   \right)ds
 \\& \quad  + \int_{m_0 \eps^{k+2}}^{c_\gamma \eps^2|\log \eps|} \gamma \eps^{-2} e^{-\gamma \eps^{-2}s} \left((1-c_0/2) C_*\eps^{k}  + C_* \eps^{k} \cdot \|g'\|_\infty \indc_{\left\{s \leq m_1 \eps^{k+1}|\log \eps|\right\}}\right) ds
 \\ &\quad +\left(2 + m_0 \gamma \right) \eps^{k}
 \\ &\leq \int_{m_0 \eps^{k+2}}^{c_\gamma \eps^2|\log \eps|} \gamma \eps^{-2} e^{-\gamma \eps^{-2}s} \left( E^{\bmo}_0 \left[ g \left(P^\eps_{d(t_1,x_0) + K\eps|\log \eps| + \bmo_s} [\V(\B_{t_1-s})=1]\right) \right] \right) ds
 \\&\quad + C_*\|g'\|_\infty \gamma m_1 \cdot \eps^{2k-1}|\log \eps|+ \left((1-c_0/2)C_* + 2 + m_0 \gamma \right) \eps^{k}. 
\end{align*}
Since $k >1$, the first term in the last line is at most $\eps^{k}$ for sufficiently small $\eps$, depending only on constants in \eqref{e.univconstants}. In the integral in the final expression, we may re-interpret $s$ as the first branch time $\tau$ of the one-dimensional BBM and apply the strong Markov property, i.e. Lemma~\ref{lemma:gfunMarkov}, to this process. We obtain that for small $\eps$,
\begin{align*}
J_\eps(t,x) &\leq P^\eps_{d(t_1,x_0) + K\eps|\log \eps|} \left[\V(\B_{t_1})=1,  \tau \in I \right] + \left((1-c_0/2)C_* + 3 + m_0 \gamma \right) \eps^{k} 
 \\ &\leq P^\eps_{d(t_1,x_0) + K\eps|\log \eps|} \left[\V(\B_{t_1})=1 \right] + \left((1-c_0/2)C_* + 3 + m_0 \gamma \right) \eps^{k}.
\end{align*}
Since $C_*$ defined in \eqref{e.newconstants} satisfies $c_0 C_*/2 = 3 + m_0 \gamma$, we have shown that
\[Q^\eps_{x_0}[\V(\W_{t_{1}}; p^+(\phi,\delta)) = 1] \leq P^\eps_{d(t_{1},x) + K \eps|\log \eps|}[\V(\B_{t_{1}}) = 1] + C_*\eps^{k}.\]
This holds for all $x_0$ and $t_1 \in [t_0, t_0 + m_0 \eps^{k+2}]$, and the proof is complete. \end{proof}

\subsection{Remainder of the proof of (J4)} \label{ss:J4holds}
The following is now an easy consequence of Proposition~\ref{prop:interformgen} and Proposition \ref{prop:propagationgen}.

\begin{proposition}\label{prop:onedcouplefulltime}
There exists $K_2 =K_{2}\eqref{e.univconstants}\in [1, \infty)$ and $\eps_{8}=\eps_{8}(\text{\ref{e.univconstants}}, \eqref{e.constants})\in (0,1)$ such that for $\eps<\eps_{8}$, for all $t \in [\sigma_1 \eps^2 |\log \eps|, h_0]$, we have
\begin{equation*}
\sup_x \left(Q^\eps_x[\V(\W_t ; p^+(\delta,\phi)) = 1] -P^\eps_{d(t,x) + K_2 \eps|\log \eps|}[\V(\B_t) =1]\right)\leq  C_* \eps^{k}.
\end{equation*}
\end{proposition}

Before proving the above, we show how it proves Theorem~\ref{thm_J4_general}.

\begin{proof}[Proof of Theorem~\ref{thm_J4_general}] Let $t \in (0,h_0]$ and recall the sub-level set
\[L^-_{t,\alpha} = \left\{x : \phi(x) - t\left[F_*(D^2\phi(x),D\phi(x))-\alpha \right] < 0 \right\}=\left\{\psi_{\al}(t,\cdot)<0\right\}.\]
We will show that for each $x \in L^-_{t,\alpha}$,
\begin{equation} \label{eq:J4limsuppf}
\limsups_{\eps \to 0} Q^\eps_x[\V(\W_t;p^+(\delta,\phi)) = 1] = a,
\end{equation}
where we recall the definition of the half-relaxed limit $\limsups_{\eps \to 0}$ from Definition~\ref{def:hrlimit}. Recall that $d(t,\cdot)$ is the signed distance function to the zero level set $L^0_{t,\alpha}$, chosen to have the same sign as $\psi_\alpha(t,\cdot)$. In particular, $x \in L^-_{t,\alpha}$ is equivalent to $d(t,x) < 0$. If $d(t,x)<0$, let us suppose $d(t,x)=-\eta$ for some $\eta>0$. By Proposition~\ref{p.distprop}, if $t\in (0, h_{0}]$, there exists $\delta=\delta(\eta, h_{0})>0$ such that for all $|t'-t| + |x'-x| \leq \delta$, $d(t',x') \leq -\eta/2$, and in particular, $d(t',x') + K_2\eps|\log \eps| \leq -\eta/4$ for sufficiently small $\eps$. Hence, by Proposition~\ref{prop:onedcouplefulltime}, if $|t'-t| + |x'- x| \leq \delta$, and $t\in (0, h_{0}]$,
\begin{align*}
Q^\eps_{x'} [\V(\W_{t'};p^+(\delta,\phi)) = 1] \leq P^\eps_{-\eta/ 4}[V(\B_{t'}) = 1] + C_*\eps^{k}.
\end{align*}
By Theorem~\ref{thm:bbminterface}, since $-\eta/4\leq -c_{1}\eps|\log\eps|$ for $\eps$ sufficiently small, we conclude that $P^\eps_{-\eta /4} [\V(\B_{t'}) = 1] \leq a+ \eps^{k}$ for sufficiently small $\eps$. In particular, $Q^\eps_{x'}[\V(\W_{t'};p^+(\delta,\phi)) = 1]$ converges uniformly to $a$ on $\{(t',x') : |t'-t| + |x'-x| \leq \delta\}$. This implies \eqref{eq:J4limsuppf}, and the proof is complete.
\end{proof}

We conclude with the proof of Proposition \ref{prop:onedcouplefulltime}. 
\begin{proof}[Proof of Proposition~\ref{prop:onedcouplefulltime}]
We recall the constants $c_1$ from \eqref{e.constants} and $K_1$ from Proposition~\ref{prop:interformgen}, respectively, and we define $K_2 = K_1 + c_1$. Let $t \in [\sigma_1 \eps^2 |\log \eps|, \sigma_2 \eps^2|\log \eps|]$ where $\sigma_{1}<\sigma_{2}$ appears in \eqref{e.constants}. By Proposition~\ref{prop:interformgen}, 
\begin{align} \label{eq:J4aux1}
\text{If $d(t,x) \leq -K_1\eps|\log \eps|$, then }\,\, Q^\eps_x[ \V(\W_t; p^+(\delta,\phi)) = 1] \leq a+\eps^{k}.
\end{align}
Now suppose that $d(t,x)> - K_1 \eps|\log \eps|$. In this case, $d(t,x) + K_2 \eps|\log \eps| > c_1 \eps |\log \eps|$, and in particular, by Theorem~\ref{thm:bbminterface}, we have $P^\eps_{d(t,x) + K_2\eps|\log \eps|} [\V(\B_t) = 1] \geq b - \eps^{k+1}$. In other words, for sufficiently small $\eps$, depending on $C_{*}$,
\begin{align}\label{eq:J4aux2}
\text{If $d(t,x) > -K_1\eps|\log \eps|$, then $P^\eps_{d(t,x) + K_2 \eps |\log \eps|}[\V(\B_t) =1] \geq b-C_* \eps^{k}$}.
\end{align}
Note that by the definition of $p^+(\delta, \phi)$ in \eqref{e.p0+def} and monotonicity, it is immediate that $Q^\eps_x[\V(\W_t;p^+(\delta,\phi)) = 1]\leq b$. Combining \eqref{eq:J4aux1}, \eqref{eq:J4aux2},  \eqref{eq:1diminterval}, and the above observation, this implies that for all $t \in [\sigma_1 \eps^2|\log \eps|, \sigma_2 \eps^2|\log \eps|]$,
\begin{align*} \label{eq:J4aux2}
\sup_{x} \bigg( Q^\eps_x[\V(\W_t;p^+(\delta,\phi)) = 1] -  P^\eps_{d(t,x) + K_2 \eps|\log \eps|}&[\V(\B_t) =1] \bigg) \\&\leq  (a+\eps^{k}-a)\vee (b-b+C_{*}\eps^{k})\\
&\leq C_* \eps^{k} .
\end{align*}
We may now iteratively apply Proposition~\ref{prop:propagationgen} to conclude that the above holds for all $t \in [\sigma_1\eps^2|\log \eps| ,h_0]$. Indeed, taking $t_0 = \sigma_2 \eps^2|\log \eps|$ and remarking that $\sigma_2 - \sigma_1 = c_\gamma$, Proposition~\ref{prop:propagationgen} implies that \eqref{eq:J4aux2} holds for all $t \in [\sigma_1 \eps^2|\log \eps|, \sigma_2 \eps^2 |\log \eps| + m_0 \eps^{k+2}]$. One then applies the result with $t_0 =  \sigma_2 \eps^2 |\log \eps| + m_0 \eps^{k+2}$, and so on. This completes the proof.
\end{proof}

\section{Applications to Stochastic Spatial Models Considered in [17]} \label{s.bbmslfv}

In this section, we present two new results which extend the work of Etheridge, Freeman, and Penington \cite{EFP2017}, which are simple consequences of Theorem \ref{t.realmain}. For brevity, we do not introduce the ``forward processes'' in detail; we focus instead on giving a precise description of the approximate dual process, the relation \eqref{eq:approxdualdef}, and the $g$-function. This allows us to present the new convergence results which are obtained through Theorem \ref{t.realmain}. 

\subsection{Ternary Branching Brownian Motion Subject to Majority Voting}\label{s.tbbm}

%The main model of interest in \cite{EFP2017} is the spatial $\Lambda$-Fleming-Viot model with selection, which we consider in the next subsection. However, in that work the authors first consider a simplified probabilistic model, which they remarkably show is dual to the Allen-Cahn equation. For completeness, we briefly discuss how our results can be applied to this model, but note that the result we state is not new, as the convergence of the Allen-Cahn equation to generalized MCF has been proved in the PDE literature, for example in \cite{BS} and some of the references therein.

The first model is a (multidimensional) ternary Branching Brownian motion (BBM) subject to majority voting, which was introduced in \cite{EFP2017}. Let $(X^{\eps}_{t}, t\geq 0)$ be a BBM where the particles evolve according to Brownian motion run at rate $1$, and after an Exp$(\eps^{-2})$-distributed random time, a particle splits into 3 particles who then undergo independent ternary BBMs. 
%As before, the ancestral process will be denoted by $(\W^{\eps}_{t}, t\geq 0)$, and we let $Q^{\eps}_{x}$ denote the distribution of the ancestral process started from $x\in \R^{\mathbf{d}}$. 

The voting algorithm on the space-time tree $\mathcal{T}_{t}=\mathcal{T}(\W^{\eps}_{t})$ is as follows. At time $t$, all children that are alive vote 1 with probability $p(X^{\eps}_{t})$ and 0 otherwise, and hence the random inputs of the leaves are distributed according to Bernoulli($p(X^{\eps}_{t}))$. Using the notation introduced in Section \ref{s:vote},
the voting algorithm we will consider is the deterministic ``majority vote'' function 
\begin{equation}\label{e.Thetadef}
\Theta(v_{1}, v_{2}, v_{3})=\begin{cases} 1&\text{if at least two arguments are 1,}\\
0&\text{otherwise}.
\end{cases}
\end{equation}
It follows that 
\begin{align*}
g(p_{1}, p_{2}, p_{3})&=E_{p_{1}, p_{2}, p_{3}}[\Theta(\mathsf{V}_{1}, \mathsf{V}_{2}, \mathsf{V}_{3})]\\
&=p_{1}p_{2}p_{3}+
p_{1}p_{2}(1-p_{3})+p_{1}(1-p_{2})p_{3}+(1-p_{1})p_{2}p_{3}
\end{align*}
is the probability that the majority of voters out of $(p_1, p_{2}, p_{3})$ vote 1, and for the univariate version, 
\begin{equation}\label{e.gbbmdef}
g(p)=3p^{2}-2p^{3}.
\end{equation}
In \cite{DH2021}, it is shown that this $g$ function satisfies \eqref{g.0}-\eqref{g.5}. As observed in \cite{EFP2017}, the function, 
\begin{equation}\label{e.acdual}
u^{\eps}(t,x):=Q^{\eps}_{x}[\V(\textbf{X}^{\eps}_{t};p)=1]
\end{equation}
solves the celebrated Allen-Cahn equation
\begin{equation*}
\begin{cases}
u^{\eps}_{t}-\frac{1}{2}\Delta u^{\eps}=\frac{1}{\eps^{2}}u^{\eps}(1-u^{\eps})\left(\frac{1}{2}-u^{\eps}\right)&\text{in $(0, \infty)\times \R^{\mathbf{d}}$},\\
u^{\eps}(0,x)=p(x)&\text{in $\R^{\mathbf{d}}$}.
\end{cases}
\end{equation*}
It is clear in this setting that  hypotheses \eqref{a.1} and \eqref{a.2} are automatically satisfied. Thus, by Theorem \ref{t.realmain}, we obtain the following result:
\begin{theorem}\label{t.bbm}
For any $p: \R^{\mathbf{d}}\rightarrow [0,1]$ which defines an initial interface $\Ga_{0}\subseteq \R^{\mathbf{d}}$, as $\eps \to 0$, $u^{\eps}$ defined by \eqref{e.acdual} converges locally uniformly to $(0,1)$-generalized MCF started from $\Ga_{0}$.  
\end{theorem}

While the result is immediate from the priorly mentioned results of \cite[Theorem 4.1]{BS}, it also serves as a simple illustration of the power of Theorem \ref{t.realmain}. 
%we highlight that Theorem \ref{t.bbm} is also immediate consequence from Theorem \ref{t.realmain}. In the case of ternary BBM subject to majority voting, hypotheses \eqref{a.1} and \eqref{a.2} are automatically satisfied, and hence Theorem \ref{t.realmain} implies that \eqref{j.1}-\eqref{j.4} hold. 

\subsection{The spatial $\Lambda$-Fleming-Viot model} \label{s.SLFV}
We apply our result to the spatial $\Lambda$-Fleming Viot with selection (SLFVS), in which the selection mechanism is chosen to model selection against heterozygosity. This model was considered in detail in \cite[Section 1.2]{EFP2017} and we refer the reader there for the motivation, derivation, and complete description of the model. 

To keep our presentation consistent with \cite[Section 1.3]{EFP2017}, we parametrize the rescaled models by $n \in \N$ so that the rescaled SLFVS is denoted by $w^{n}_{t}: \R^{d}\rightarrow [0,1],$ which we interpret as a measure-valued process (i.e. the density of a measure). The initial condition $w^{n}_{0}=p,$ where $p: \R^{d}\rightarrow [0,1]$ will define an initial interface. We note that in the definition of the rescaled process, there is a family of parameters $\left\{\eps_{n}\right\}_{n\geq 0}$ which is any sequence tending to $0$ such that $\eps_{n}(\log n)^{1/2}\to \infty$.

Our main result regarding SLFVS is the following:
\begin{theorem}\label{t.slfvs}
Let $(w^{n}_{t}, t\geq 0)_{n>0}$ denote the rescaled SLFVS as defined in \cite[Section 1.3]{EFP2017}. Then for any $p: \R^{\mathbf{d}}\rightarrow [0,1]$ which defines an initial interface $\Ga_{0}\subseteq \R^{\mathbf{d}}$, $\E^{n}_{p}[w^{n}_{t}(x)]$ converges locally uniformly to $(0,1)$-generalized MCF started from $\Ga_0$.
\end{theorem}

We next describe the dual, the approximate dual, and the associated $g$ function. For details, we refer to Section~3 of \cite{EFP2017}. Fix $n\in \NN$. The dual is a branching/coalescing process $(\hat{X}^{n}_{t})_{t\geq 0}$ which is a $\cup_{\ell\geq 1}(\RR^{\mathbf{d}})^{\ell}$-valued Markov process, with $\hat{X}^n_{0}=x$, and $\hat{X}^{n}_{t}=(\hat{X}^{n}_t(\alpha): \alpha \in N(t) )$. The dynamics of $\hat{X}^{n}_t$ may be defined via a point process $\Pi^{n}$ which drives $w_t^n$. Each event in $\Pi$ independently marks individuals within its radius with probability $u_n$; in non-selective events, all marked individuals coalesce into a single offspring, whereas in selective events, all marked individuals coalesce and are replaced by three offspring whose locations are drawn independently and uniformly from a ball of radius $r$ centered at the parent.

It is shown in \cite{EFP2017} that $\hat{X}^n_t$ is approximated by a branching process. For any $\Lambda>0$, with high probability, by time $\Lambda$, no event marks more than one individual at a time. Conditional on no event marking multiple individuals up to time $\Lambda$, non-selective events are just random walk steps for the single marked individual, and selective events are ternary branches for the single marked individual. (Events with no marked individuals have no effect.) The approximate dual $X^n_t$ is defined as having precisely these dynamics and it is a branching dual in the sense of Section~\ref{s:vote}. It is a continuous time branching random walk with branch rate $\gamma \eps_n^{-2}$ for an explicit constant $\gamma$ (see (3.5) of \cite{EFP2017}). As before, we denote by $\hat{\X}^n_t$ and $\X^n_t$ the historical processes associated to $\hat{X}^n_t$ and $X^n_t$. Without loss of generality, we may construct coupled versions of $\hat{\X}^n_t$ and $\X^n_t$ on a common probability space whose law, when the initial individual starts at $x \in \R^{\mathbf{d}}$, we denote by $Q^n_{x}$. It is a consequence of \cite[Lemma~3.12]{EFP2017} that the coupling may be chosen so that for any $\Lambda > 0$,
\begin{equation} \label{e.SLFVaproxdual1}
    \sup_x \sup_{t \in [0,\Lambda]} Q^n_x[ \hat{\X}^n_t \neq \X^n_t] = o(1).
\end{equation}

For a description of the voting procedure $\V$ on the true dual $\hat{\X}^n_t$, see \cite{EFP2017}. Restricted to $\X^n_t$, the voting algorithm $\V$ is simply (ternary) majority voting and corresponds to univariate $g$-function $g(p) = 3p^2-2p^3$ as in \eqref{e.gbbmdef}. Since the SLFVS is a measure-valued process, defining the dual relation requires integration against test functions. The dual relation (see \cite[Theorem 3.4]{EFP2017}) is as follows: for every $\psi\in C(\R^{\mathbf{d}})\cap L^{1}(\R^{\mathbf{d}})$, 
\begin{equation}\label{e.slfvscouple}
\begin{aligned}
\E^{n}_{p}\left[\int_{\R^{\mathbf{d}}}\psi(x)w^{n}_{t}(x)\, dx\right]&=\int_{\R^{\mathbf{d}}}\psi(x)\hat{E}^{n}_{x}\left[\V(\hat{\W}^{n}_{t}; p)\right]\, dx\\
&=\int_{\R^{\mathbf{d}}}\psi(x)\hat{Q}^{n}_{x}\left[\V(\hat{\W}^{n}_{t}; p)=1\right]\, dx.
\end{aligned}
\end{equation}

We now verify the hypotheses of Theorem \ref{t.realmain} needed in order to prove Theorem \ref{t.slfvs}.
As a consequence of \eqref{e.SLFVaproxdual1}, for every $\psi\in C(\R^{\mathbf{d}})\cap L^{1}(\R^{\mathbf{d}})$, and $\Lambda>0$,
\begin{equation*}
\sup_{t\in [0, \La]} \left|\E^{n}_{p}\left[\int_{\R^{\mathbf{d}}}\psi(x)w^{n}_{t}(x)\, dx\right]-\int_{\R^{\mathbf{d}}}\psi(x)Q^{n}_{x}\left[\V(\W^{n}_{t}; p)=1\right]\, dx\right|=o(1). 
\end{equation*}
This guarantees a version of \eqref{eq:approxdualdef} holds for the purely branching dual $(\W^{n}_{t}, t\geq0)$. In particular, this implies that for any compact set $K\subseteq (0, \infty)\times \R^{\mathbf{d}}$ such that $Q^{n}_{x}\left[\V(\W^{n}_{t}; p)=1\right]$ converges to 1 or 0 uniformly on $K$, we have that $\E^{n}_{p}[w^{n}_{t}(x)]$ must also converge to the same value for a.e. $(t,x)\in K$.

As in the prior section, the $g$ function satisfies \eqref{g.0}-\eqref{g.5}. The SLFVS was not considered in \cite{DH2021}, but all of the estimates relevant to Section \ref{ss:formation} are proven in \cite[Section 3.2.3]{EFP2017}, with error bounds in terms of $\eps_{n}$. In particular, hypotheses \eqref{a.1} and \eqref{a.2} for this model take on the following form:
\begin{list}{ (\theaprimescan)}
{
\usecounter{aprimescan}
\setlength{\topsep}{1.5ex plus 0.2ex minus 0.2ex}
\setlength{\labelwidth}{1.2cm}
\setlength{\leftmargin}{1.5cm}
\setlength{\labelsep}{0.3cm}
\setlength{\rightmargin}{0.5cm}
\setlength{\parsep}{0.5ex plus 0.2ex minus 0.1ex}
\setlength{\itemsep}{0ex plus 0.2ex}
}
\item \label{aprime.1} \textbf{Lineages converge to Brownian motion.} For every $x$, $(Y_t, t\geq0)$ started at $x$ can be coupled with a Brownian motion $(\bmd_t, t\geq 0)$ started from $x$, such that for all $n>n_0$,
\begin{equation*}
\sup_x \sup_{s \in (0, \eps_{n}^2 |\log \eps_{n}|^2]} P^{Y, n}_x [ |Y_s - \bmd_s| > \eps_{n}^{k+2}]\leq \bar{C}e^{-\bar{c} \eps_{n}^\eta}.
\end{equation*}

\item \label{aprime.2} \textbf{Offspring dispersion concentration.} For $n>n_{0}$, we have 
\begin{equation*}
 \sup_y \mu^{n}_y[ \xi: \max_{1\leq i\leq N_0} |\xi_i| > \eps_{n}^{k+2} ] \leq \bar{C}e^{-\bar{c}\eps_{n}^\eta}.
\end{equation*}
\end{list}
These estimates are of the correct order given that the branch rate of $X_t^n$ is of order $\eps_n^{-2}$. Property \eqref{aprime.1} is a direct consequence of \cite[Lemma 3.8]{EFP2017}, and Property \eqref{aprime.2} follows from the scaling of the model. Theorem \ref{t.slfvs} is now immediate from Theorem \ref{t.realmain}.

\section{Interacting particle systems}\label{s.ips}
We present new convergence results for several interacting particle systems considered in the work of Huang and Durrett \cite{DH2021}. These results all follow as consequence of the arguments in Theorem \ref{t.realmain}. We discuss additional modifications which are necessary for each specific model in this section.

Following \cite{DH2021}, we discuss two different perturbations of the linear voter model and one particle system called the sexual reproduction model with fast stirring. As in the previous section, we do not include the full details of the forward process; our discussion primarily focuses on the approximate dual $(\X_{t}, t\geq 0)$, the relation \eqref{eq:approxdualdef}, and the $g$-function associated to the voting algorithm. 

Each interacting particle system is a rescaling of a $\{0,1\}^{\Z^\mathbf{d}}$-valued Markov process, characterized in terms of its flip rates. 
The state of the process at time $t\in [0, \infty)$ can be represented by the function $\xi_{t}: \Z^{\mathbf{d}}\rightarrow \left\{0,1\right\}$. The initial state $\xi_{0}$ will be determined by a function $p: \R^{\mathbf{d}}\rightarrow [0,1]$; if $(\xi_t :t \geq 0)$ has initial state $p$, then $\xi_0(x)$ is a Bernoulli$(p(x))$ random variable, and different sites' initial states are independent. 

We now briefly describe the scaling regime considered, but refer the reader to \cite{DH2021} for more details for each model. Our scaling parameter is denoted $\eps >0$, but the rescaled process lives on $\eta \Z^{\mathbf{d}}$, where $\eta = \eps \delta$ for some $\delta = \delta(\eps)$, that is, we consider $\xi^\eps_t \in \{0,1\}^{\eta \Z^{\mathbf{d}}}$. This process arises after two scalings. The first scaling defines a process on scale $\delta$, i.e. with values in $\{0,1\}^{\delta \Z^{\mathbf{d}}}$. The rates are tuned to obtain a ``motion'' resembling a random walk on scale $\delta$ with jump rate of order $\delta^{-2}$ and nonlinear effects/interactions with rate of order $O(1)$. An additional diffusive rescaling $(t,x) \mapsto (\eps^2 t,\eps x)$ amplifies the effect of the nonlinearity over $O(1)$ timescales. To achieve a necessary separation of time scales between diffusion and nonlinearity, one then chooses $\delta = \exp(-\eps^{-3})$. In fact, it suffices to take $\delta = \exp(-\eps^{-(2+c)})$ for any $c > 0$, but the proofs will break down if $c=0$. We fix $c>0$ and take $\delta = \exp(-\eps^{-(2+c)})$ in the sequel. We denote the law and expectation associated to $\xi^\eps_t$ with initial conditions generated by $p$ by $\mathbb{P}^\eps_p$ and $\mathbb{E}^\eps_p$. Finally, we extend $\xi^\eps_t$ to $\mathbb{R}^{\mathbf{d}}$ by setting $\xi^\eps_t(x) = \xi^\eps_t(\eta \lfloor \eta^{-1} x \rfloor)$ for $x \in \mathbb{R}^{\mathbf{d}}$, where $\lfloor \cdot \rfloor$ is is the integer part rounded down in each coordinate.

We now state three convergence theorems to generalized MCF for three different models, and discuss the verification of the hypotheses of Theorem \ref{t.realmain} in each setting.

\subsection{The Voter Model and the Lotka-Volterra Perturbation}  \label{s.voteLVvote}
Here we consider a model known as the Lotka-Volterra perturbation of the voter model, or simply the Lotka-Volterra voter model. A model for two competing species, its flip rates combine the standard flip rates of the voter model with the perturbation flip rates, in which two individuals collaborate to change the vote of a third individual. It was first studied in \cite{NeuhauserPacala}; see also \cite{CoxDurrettPerkins} and \cite[Section 1.3]{DH2021} for a full description of the model. The main result concerning the Lotka-Volterra perturbation is the following.

\begin{theorem}\label{t.LV}
Let $(\xi^\eps_t, t\geq 0)_{\eps>0}$ denote the rescaled Lotka-Volterra voter model with $\mathbf{d} \geq 3$. For any $p:\R^{\mathbf{d}}\rightarrow [0,1]$ defining an initial interface $\Ga_{0}\subseteq \R^{\mathbf{d}}$, as $\eps \to 0$, $\E^{\eps}_{p}[\xi^{\eps}_{t}]$ converges locally uniformly to $(0,1)$-generalized MCF started from $\Ga_{0}$. 
\end{theorem}

The dual process $(\hat{\X}^{\eps}_{t}, t\geq 0)$ of $\xi^{\eps}_t(\cdot)$ is a branching-coalescing random walk with ternary branching; see \cite[Section 2.2]{DH2021} for a full description. In the case when $\mathbf{d} \geq 3$, the approximate dual $(\X^\eps_t, t\geq 0)$ is obtained by first restricting the time intervals during which branching and coalescence is permitted, which is shown to have no effect on the dynamics with high probability, and then identifying individuals which coalesce within time $\eta^{1/2}$ of their birth. Subsequently, it can be shown that branching events followed by coalescences may be ignored altogether. This gives $(\X^\eps_t, t\geq 0)$ a ternary branching structure, and the voting algorithm $\V$ corresponds to majority voting. Thus, the $g$ function in this setting is again given by 
\begin{equation*}
g(p)=p^{3}-3p^{2}.
\end{equation*}
Since branching events in $(\hat{\X}^{\eps}_{t}, t\geq 0)$ which are followed by coalescences are ignored in $(\X^{\eps}_{t}, t\geq 0)$, if $\hat{\X}^{\eps}_{t}$ has branch rate $\eps^{-2}$, then $(\X^\eps_{t}, t\geq 0)$ is a pure branching process with a slowed-down branch rate given by $p_{3}\eps^{-2}$, where $p_{3}$ is the probability that the three particles, after a branching event, do not coalesce. In the sequel, we assume without loss of generality that $\X^\eps_t$ has branch rate $\eps^{-2}$, which corresponds to a slight reparametrization of the original process.

We now give a precise definition of the dynamics of $(\X^\eps_t, t\geq 0)$. Each particle undergoes a continuous time nearest-neighbour random walk on $\eta \Z^{\mathbf{d}}$ run at speed $\mathbf{d}\eta^{-2}$; that is, for each of the $2\mathbf{d}$ nearest neighbours $y$ of $x$ in $\eta \Z^{\mathbf{d}}$, an individual at $x$ jumps to $y$ at rate $\eta^{-2}/2$. Particles branch into three particles at rate $ \eps^{-2}$, and in this case,
\begin{equation}\label{e.mudeflv}
\mu^{\eps}_{y}(\cdot)= \delta_y \otimes K_{\eps}(\cdot-y) \otimes K_{\eps}(\cdot - y),
\end{equation}
where $K_\eps(y) = K(y/\eta)$ and $K$ is a uniform probability kernel on $[-L,L]^{\mathbf{d}} \cap \Z^{\mathbf{d}}$. Thus, one of the offspring is placed at the location of the parent and the other two are independently sampled from $K_{\eps}$, centered at the location of the parent.
 
It is a consequence of \cite[Lemma~2.3]{DH2021} and the fact that $\mathbf{d}\geq 3$ that \eqref{eq:approxdualdef} holds, that is
\begin{equation}
\E^\eps_{p}[\xi^\eps_t(x)] = Q^{\eps}_x[\V(\X^\eps_t; p) = 1] + o(1).
\end{equation}
Since $(\X^\eps_t, t\geq 0)$ has the branching property, Proposition~\ref{prop:J1} and the above imply that \eqref{j.1} holds. $\V$ is associated to the $g$-function $g(p) = 3p^2 - 2p^3$, which satisfies \eqref{g.0}-\eqref{g.5} by the same arguments as in the previous sections.

To apply Theorem~\ref{thm_J4_general}, it remains to verify that \eqref{a.1} and \eqref{a.2} hold. To do so, we first rewrite the above in the notation of our paper. The single lineage process $Y^{\eps}_t$ associated to $\X^\eps_t$, whose law we denote by $P_x^{Y,\eps}$, is simply a continuous time simple random walk on $\eta \Z^{\mathbf{d}}$ with jump rate $\mathbf{d} \eta^{-2}$. Branching events are ternary, and hence $N_0 = 3$, and they occur at rate $p_3\eps^{-2}$. The offspring are distributed according to $\mu^{\eps}_{y}$ as defined in \eqref{e.mudeflv}.
Since $K_\eps$ is supported on $[-L\eta,L\eta]^d$ and $L \in \N$ is fixed, it is immediate that $\mu_y^\eps(\cdot)$ satisfies \eqref{a.2}, so we need only verify \eqref{a.1}. This amounts to the verification that a simple random walk can be coupled with a Brownian motion with a sufficiently small error. This is handled in \cite{DH2021} using Skorokhod embedding. We point out that \cite[Lemma~2.4]{DH2021} as stated is not quite strong enough, but this is mostly because of an unnecessary bound used in the last line of its proof. One can follow the proof of that result, making small changes as necessary, to obtain the following.

\begin{lemma}\label{lemma:LV_BM}
There exists a constant $C \geq 1$ and a coupling of $Y^\eps_t$ with a Brownian motion $\bmd_t$ such that for any $T>0$, for sufficiently small $\eps>0$,
\begin{equation}
\sup_x P^{Y,\eps}_x\left[\sup_{t \in [0,T]}|\bmd_{t} - Y^\eps_t| > \eta^{1/6} \right] \leq C(1+T)\eta^{1/2}.
\end{equation}
\end{lemma}
Since $\eta = \exp(-\eps^{-(2+c)})\eps$, we remark that the above is much stronger than \eqref{a.1}, and thus it implies that \eqref{a.1} holds for any $k \geq 3$. Upon verifying all of the hypothes, we can apply Theorem~\ref{t.realmain}, which yields Theorem~\ref{t.LV}.

\subsection{Nonlinear voter model perturbations}\label{s:nonlinearvoter}

We consider another class of voter models which is a continuous time version of the one considered in the work of Molofsky et al \cite{molo}, and has been considered in \cite{CoxDurrettPerkins} and \cite{DH2021}. Hereafter, we refer to it simply as the nonlinear voter model. For a precise description of the model, we refer to \cite{DH2021}. 

The flip rates in this setting are characterized by a parameter $L$ and Bernoulli parameters $a_{k}$, for $k = 0,\dots,5$, which satisfy the following:
\begin{equation}\label{e.arels}
a_{0}=0, a_{1}=1-a_{4}, a_{2}=1-a_{3}, a_{5}=1.
\end{equation}
The exact form of flip rates subject to these dynamics can be found in \cite[Example 1]{CoxDurrettPerkins}. As before,  we consider a family of processes $\xi^\eps_t$ parametrized by $\eps$. Our main result is the following convergence statement:

\begin{theorem}\label{t.nonlinearv}
Let $(\xi^\eps_t, t\geq 0)_{\eps>0}$ denote the rescaled nonlinear voter model with $\mathbf{d} \geq 3$. Then for sufficiently large $L \in \N$, there exist equilibria $a_0 = a_0(L)$, $b_0 = b_0(L)$ such that $0< a_0 < 1/2 < b_0 < 1$, and for any $p: \R^{\mathbf{d}}\rightarrow [a_0,b_0]$ defining an initial interface $\Ga_{0}$, $\E^{\eps}_{p}[\xi^{\eps}_{t}]$ converges locally uniformly to $(a_0,b_0)$-generalized MCF started from $\Ga_{0}$. 
\end{theorem}

\begin{remark} In \cite{DH2021}, the authors require that there exists $\la>0$ such that $p\in [\la, 1-\la]$. For simplicity, we have stated Theorem \ref{t.nonlinearv} with the stronger assumption that $p\in [a_0,b_0]$. With a relatively straightforward additional argument, this can be relaxed to match their assumption on $p$, but we do not pursue this here. 
\end{remark}

In the case of the nonlinear voter model perturbation, the dual $(\hat{\X}^{\eps}_{t}, t\geq 0)$ is given in terms of a branching/coalescing random walk which branches into five children, where one lineage is understood to be the parent particle. As shown in \cite[Section 2.2.1]{DH2021}, there exists a process $(\X^{\eps}_{t}, t\geq 0)$ which is only allowed to coalesce (and not branch) on time intervals of the length $\eta^{1/2}$ immediately after branching events. Branching occurs at rate $\eps^{-2}$. At a branching event, a particle gives birth to 
4 particles, chosen uniformly without replacement from $x+([-L\eta, L\eta]^{\mathbf{d}}\cap \eta \Z^{\mathbf{d}})$ for a sufficiently large $L \in \N$, and the particle itself remains alive (and stays in the same location). In the notation as in the beginning of Section \ref{s:genthm}, $\mu^{\eps}_{y}$ is the product measure of a Dirac centered at $y$ and the distribution of sampling 4 particles uniformly from $y+([-L\eta, L\eta]^{\mathbf{d}}\cap \eta \Z^{\mathbf{d}})$ without replacement.

To maintain a tree structure, siblings which coalesce within $\eta^{1/2}$ of their birth are identified. A priori this leads to a tree which is not necessarily $5$-ary, but we follow \cite[Section 2.3]{DH2021} in keeping the $5$-ary tree and encoding the coalescences with additional information. Doing this precisely requires some work; for the time being it suffices to note that $(\X^{\eps}_{t}, t\geq 0)$ is associated to a $5$-regular tree. In this model, the $g$-function will be defined by 
\begin{align} \label{e:NLVgfun}
&g(p)=(4a_{1}-a_{4})p(1-p)^{4}+(6a_{2}-4a_{3})p^{2}(1-p)^{3}-(6a_{2}-4a_{3})p^{3}(1-p)^{2} \notag\\
& \hspace{2 cm}-(4a_{1}-a_{4})p^{4}(1-p)+p, 
\end{align}
which corresponds exactly to the voting algorithm $\V$ if there were no coalescences. However, as described in \cite{DH2021}, in this model one must take into account the impact of coalescences in the evaluation of $\V$, which necessitates a more sophisticated set-up which we describe shortly. To this end, we introduce the $\Theta$-function associated to this model, where the role of $\Theta$ is described in Section \ref{s:vote}. From the description of the dynamics above, $\Theta : \{0,1\}^5 \to [0,1]$ is given by
\begin{equation}
\Theta(v_1,v_2,v_3,v_4,v_5)  = \sum_{k=0}^5 a_k \indc_{\{\sum v_i = k\}}.
\end{equation}

In order to pursue the verification of properties needed to prove Theorem \ref{t.nonlinearv}, a more general framework is needed. This is due to the fact that for the model at hand, the approximate dual relation \eqref{eq:approxdualdef} is more subtle than the other models considered; the coalescences of the dual cannot  simply be ignored in the computation of the voting algorithm. Instead, each coalescence event in which a given subset of lineages merge has a corresponding $g$-function, and the ``effective'' $g$-function is the weighted average of these $g$-functions, each of which corresponds to a coalescence event. The situation is further complicated by the fact that the coalescence probabilities, which yield the weights in the expression of the effective $g$-function, depend on the relative locations of the offspring created in a branching event. In the case when $L$ is large, it can be shown that the effective $g$-function is uniformly close to $g$ defined in \eqref{e:NLVgfun}, uniformly in $\eps$. This implies that for $L$ large, the voting algorithm from the dual, which takes into account information about coalescences, can be approximated by a modified voting algorithm which acts on the regular tree.

Therefore, to apply the framework from Section \ref{s:vote} to this setting, we must adapt the existing framework to allow trees which contain additional information. For a full treatment of this adaptation in the setting of the nonlinear voter model, we refer the reader to \cite{xtend}. Nevertheless, the convergence result Theorem \ref{t.nonlinearv} still follows essentially from Theorem \ref{t.realmain}.

\subsection{A sexual reproduction model with fast stirring} \label{s.SRMS}
The model in this section is a sexual reproduction model in the sense that pairs of nearby individuals produce new individuals. However, there are no sexual types; any two individuals can reproduce. Besides these sexual reproduction dynamics, individuals move in space by stirring dynamics in which neighbouring sites swap states. As in the previous sections, the model with scaling parameter $\eps$ is denoted by $\xi^\eps_t(x)$ and lives on $\eta \Z^{\mathbf{d}}$.

Due to a technical condition arising with the dual process, we state our convergence theorem for a slightly restricted class of initial conditions. In particular, we assume that the function $p$ takes only two values and that the regions where it takes these values are separated by an interface $\Gamma_0$ which has upper box dimension strictly less than $\mathbf{d}$. (For a definition of this definition and its properties, see Falconer \cite{Falconer}.) Since the dimension of $\Ga_{0}$ equals $\mathbf{d}-1$ when $\Gamma_0$ is smooth, this is not a very restrictive assumption. These assumptions are likely unnecessary but simplify a calculation in verification of the approximate dual property, as we discuss after stating the theorem.

\begin{theorem}\label{t.SRSM}
Let $(\xi^\eps_t, t\geq 0)_{\eps>0}$ denote the rescaled sexual reproduction model with stirring, with $\mathbf{d} \geq 2$. Suppose that $p: \R^{\mathbf{d}}\rightarrow \{q_1,q_2\}$ defines an initial interface $\Ga_0$, where $0\leq q_1 < 1/3 < q_2 \leq 2/3$ and $\Ga_0$ has upper box dimension less than $\mathbf{d}$. Then $\E^{\eps}_{p}[\xi^{\eps}_{t}]$ converges locally uniformly to $(0,2/3)$-generalized MCF started from $\Ga_0$.
\end{theorem}

It again suffices to verify that there is an approximate dual process satisfying the appropriate properties. For details, including a precise description of the dual, we refer to \cite{DH2021}, where most of the arguments are made. We focus our discussion on the one point which requires additional attention, which is also the cause of the additional assumptions in Theorem~\ref{t.SRSM}.

The true dual $\hat{\X}_t $ (where we have suppressed the dependence on $\eps$) is a branching-stirring system with collisions. Once one verifies that collisions occur with vanishing probability \cite[Lemma~2.1]{DH2021}, the resulting approximate dual, denoted by $\X_t$, is a branching-stirring system which satisfies \eqref{eq:approxdualdef}. The difference between this process and a branching random walk is that, because the motion is by stirring, the spatial motions of different individuals are correlated, since individuals at adjacent sites can swap positions. This causes the failure of Lemma~\ref{lemma:gfunMarkov}, and in particular \eqref{j.1} does not hold for $\X_{t}$, the branching-stirring system. However, one can take yet another approximate dual by replacing the motion-by-stirring with independent random walks to obtain a branching random walk. The fact that this perturbation is small follows more or less from the same argument used to prove that collisions do not occur with high probability. This additional approximation (replacing stirring by independent random walk motions) is not mentioned explicitly in \cite{DH2021}, but the argument is made for the same system in a different scaling regime in \cite{DN1994}. We denote by $\X_t^*$ the branching random walk approximation to $\X_t$, obtained by taking the same branching process but allowing particles to evolve like independent branching random walks. Thus, the set $N(t)$ of living individuals at time $t$ is the same for both processes for all $t$. We observe that the random walk dynamics are such that an individual jumps to one of its $2\mathbf{d}$ nearest neighbours uniformly at rate $2\mathbf{d} \eta^{-2}$. Hereafter, we assume that $\X_t$ and $\X_t^*$ are both defined, via a particular coupling, under $Q^\eps_x$, and refer to \cite{DN1994} for the details.  

The process $\X^*_t$ is a nearest neighbour branching random walk on a lattice with mesh $\eta$ and jump rate $\mathbf{d}\eta^{-2}$, and the displacements of the individuals at branch times are also to nearest neighbour sites. The branch rate equals $\eps^{-2}$. It is therefore trivial to see that \eqref{a.2} holds, and that it also satisfies a particularly strong version of \eqref{a.1}, which we return to shortly. As shown in \cite[Section 1.4]{DH2021}, the $g$-function associated to the voting algorithm is given by 
\begin{equation*}
g(p)=\frac{9}{11}[p+p^{2}-p^{3}],
\end{equation*}
which satisfies \eqref{g.0}-\eqref{g.5}. The only missing ingredient is therefore that it is actually an approximate dual to our system. Given that this is the case for $\X_t$, we need to verify that on compact sets $K\subseteq (0, \infty)\times \R^{\mathbf{d}}$,
\begin{equation} \label{e:approxdualSR}
\sup_{K}|Q_x^\eps[ \V(\X_t; p) = 1] - Q_x^\eps[\V(\X_t^*; p) =1 ] | = o(1).
\end{equation}
Unlike the coupling of $\X_t$ to $\hat{\X}_t$, the individuals in $\X_t$ and $\X_t^*$ can have different locations, which cause the input leaf votes in the evaluation of $\V$ to have different distributions. Our choice to restrict to piecewise constant voting functions considerably simplifies this, however. We now demonstrate that \eqref{e:approxdualSR} holds. 

Let $x \in \eta \Z^{\mathbf{d}}$ and let $X_t$ and $X_t^*$ denote the non-historical versions of $\X_t$ and $\X_t^*$, all under $Q^\eps_x$. Arguing as in Section 2.b of \cite{DN1994}, in particular the arguments culminating in equation (2.8), one can show that
\begin{equation} \label{eq:SR_approxdual1}
Q_x^\eps \left[ \sup_{\alpha \in N(t)} | X_t(\alpha) - X_t^*(\alpha)| > 2 \eta^{1/6} \right]  \leq C\eta^{5/3} e^{2\eps^{-2}t} \log(t \eta^{-2}) = o(1).
\end{equation}
Therefore, up to an event with vanishing probability, the locations of the leaves, i.e. the individuals in $X_t$ and $X_t^*$, are uniformly within $2\eta^{1/6}$ of each other. To conclude, we will argue that, again with high probability, there are no individuals in $X_t^*$ within this distance of $\Gamma_0$. In particular, for $\Theta_{0}$ as in Definition \ref{d.defint}, this implies that for all $\alpha \in N(t)$, $X_t^*(\alpha) \in \Theta_0$ if and only if $X_t(\alpha) \in \Theta_0$. By our choice of $p$, we then have $p(X_t(\alpha)) = p(X_t^*(\alpha))$ for all $\alpha \in N(t)$. Thus, we can choose the coupling of $\X^*_t$ and $\X_t$ such that, on the intersection of the event described above and the event in \eqref{eq:SR_approxdual1}, $\V(\X_t; p) = \V(\X_t^*; p)$, and this proves \eqref{e:approxdualSR}. As a result, it suffices to verify that, for any $t>0$, for any starting point $x\in \eta\Z^{\mathbf{d}}$, $\dist(X_{t}^{*}(\alpha), \Ga_{0})>2\eta^{1/6}$ for all $\alpha \in N(t)$, where $\dist$ denotes the unsigned distance.

As remarked previously, the spatial motion by a random walk on $\eta \Z^{\mathbf{d}}$ can be coupled with a Brownian motion to have a small error, giving a version of the condition \eqref{a.1} for the model. In fact, this is already given by Lemma~\ref{lemma:LV_BM}. Moreover, the nearest-neighbour displacements from $O(\eps^{-2})$ branching events along a trajectory up to time $t$ contribute an error of order $\eps^{-2} \eta \ll \eta^{1/6}$. In particular, there is a Brownian motion $\bmd_t$ such that
\begin{equation}
Q^\eps_x[ |X_t(\alpha) - \bmd_t| > \eta^{1/6} \, | \, \alpha \in N(t)] \leq C(1+t)\eta^{1/2}.
\end{equation}
For $r>0$ (and $\Gamma_0$ the initial interface), let $\Gamma_0^r = \{x \in \R^{\mathbf{d}} : \dist(\Gamma_0,x) \leq r\}$. Since $\Gamma_0$ is compact and has upper box dimension $\delta < \mathbf{d}$, there is a constant $C>0$ such that for all $r \in (0,1]$,
\begin{equation} \label{eq:SR_Lebinterface}
\text{Leb}(\Gamma_0^r) \leq C r^{\mathbf{d}-\delta},
\end{equation}
where $\text{Leb}(\cdot)$ denotes the $\mathbf{d}$-dimensional Lebesgue measure. We now conclude as follows. Let $p(t,y)$ denote the $\mathbf{d}$-dimensional Brownian heat kernel. We have
\begin{align*}
Q^\eps_x&[ X_t(\alpha) \in \Gamma_{0}^{2\eta^{1/6}} \, | \, \alpha \in N(t)] \\
&\leq P_x^{\bmd}[\bmd_t \in \Gamma_{0}^{\eta^{1/6}}] + Q^\eps_x[|X_t(\alpha) - \bmd_t| > \eta^{1/6} \, | \, \alpha \in N(t)]
\\ &\leq \left(\sup_{y \in \R^{\mathbf{d}}} p(t,y)\right) \text{Leb}(\Gamma_0^{\eta^{1/6}}) + C(1+t) \eta^{1/2}
\\ &\leq C t^{-\mathbf{d}/2} \eta^{(\mathbf{d}-\delta)/6} + \eta^{1/6}.
\end{align*}
In particular, by a union bound,
\begin{align*}
Q^\eps_x[ \exists \alpha \in N(t) : X_t(\alpha) \in \Gamma^{2\eta^{1/6}} ] &\leq E^\eps_x[\#\{\alpha \in N(t) : X_t(\alpha) \in \Gamma^{2\eta^{1/6}}\}]
\\& \leq e^{2\eps^{-2}t}[C t^{-\mathbf{d}/2} \eta^{(\mathbf{d}-\delta)/6} + \eta^{1/6}].
\end{align*}
The term $e^{\eps^{-2}t}$ is the expected number of individuals at time $t$ in a continuous time branching process with ternary branches at rate $\eps^{-2}$. Recalling that $\eta = e^{-\eps^{-(2+c)}}$ (and because $\delta < \mathbf{d}$), it follows that the above vanishes uniformly on $t \in [K^{-1}, K]$ for all $K>1$. In view of the previous discussion, this proves that $\X^*_t$ is an approximate dual to the system, and we are done.

\section{Appendix} \label{s.app}

\subsection{Generalized Flows and the Generalized Level-Set Method}\label{s.app1}
We provide an overview of the theory of generalized flows and the generalized level-set method. Everything in this section is contained in the work of Barles and Souganidis \cite{BS}. Consider a general PDE of the form 
\begin{equation}\label{e.geneq}
\begin{cases}
u_{t}+F(D^{2}u, Du)=0&\text{in $(0, \infty)\times\R^{\mathbf{d}}$,}\\
u(0,x)=u_{0}&\text{in $\R^{\mathbf{d}}$,}
\end{cases}		
\end{equation}
where $F: \mathcal{S}^{\mathbf{d}}\times \R^{\mathbf{d}}\rightarrow \R$ satisfies the following:
\begin{enumerate}[(i)]
\item There exists $G:\mathcal{S}^{\mathbf{d}}\times \R^{\mathbf{d}}\rightarrow \R$ such that  
\begin{equation*}
F(M,p)=|p|G\left(\frac{1}{|p|}\left(\id-\frac{p}{|p|}\otimes \frac{p}{|p|}\right)M, \frac{p}{|p|}\right)
\end{equation*}
\item For all $p\in \R^{\mathbf{d}}$, 
\begin{equation*}
F(M,p)\leq F(N,p)\quad\text{whenever}\quad M\geq N.
\end{equation*}
\item $F$ is locally bounded. 
\end{enumerate}
We note that the case of MCF is the special case when 
\begin{equation*}
F(M,p)=|p|G\bigg(\frac{1}{|p|}\Big(\id-\frac{p}{|p|}\otimes \frac{p}{|p|}\Big)M, \frac{p}{|p|}\bigg)=-\frac{1}{2}\tr \bigg(\Big(\id-\frac{p}{|p|}\otimes \frac{p}{|p|}\Big)M\bigg).
\end{equation*}

We first begin by defining what it means for a sequence of sets $\left\{\Theta_{t}\right\}_{t\geq 0}\subseteq \R^{\mathbf{d}}$ to be a generalized flow:

\begin{definition}\label{d.genflow}
\begin{enumerate}[(i)]
\item A family of open sets $\left\{\Theta_{t}\right\}_{t\geq0}\subseteq \R^{\mathbf{d}}$ is a generalized superflow with normal velocity $-F$ if and only if for every $(t_{0}, x_{0})\in (0, \infty)\times\R^{\mathbf{d}}$, $r\in (0,1)$, and every smooth function $\phi : \R^\mathbf{d} \to \R$ satisfying $\{\phi\geq 0 \} \subseteq B(x_0,r)$ with $|D\phi(\cdot)| \neq 0$ on $\{\phi = 0\}$, there exists $h_0=h_{0}(\norm{\phi}_{C^{4}(\overline{B(x_{0},r)})}, \text{\ref{e.univconstants}})\in (0,1)$ such that for every $\al\in (0,1)$, the following holds: for 
\begin{equation*}
L^{+}_{h,\al}:=\left\{x: \phi(x) - h\left(F^{*}(D^2\phi(x), D\phi(x)) + \alpha \right) >0\right\}\subseteq \R^{\mathbf{d}},
\end{equation*}
we have that for all $h\in (0, h_{0}]$, 
 \begin{equation*}
 L^{+}_{h}\cap \overline{B(x_{0},r)}\subseteq \Theta_{t_{0}+h}.
 \end{equation*}
 \item A family of open sets $\left\{\Theta_{t}\right\}_{t\geq0}\subseteq \R^{\mathbf{d}}$ is a generalized subflow with normal velocity $-F$ if and only if for every $(t_{0}, x_{0})\in (0, \infty)\times\R^{\mathbf{d}}$, $r\in (0,1)$, and every smooth function $\phi : \R^\mathbf{d} \to \R$ satisfying $\{\phi\geq 0 \} \subseteq B(x_0,r)$ with $|D\phi(\cdot)| \neq 0$ on $\{\phi = 0\}$, there exists $h_0=h_{0}(\norm{\phi}_{C^{4}(\overline{B(x_{0},r)})}, \text{\ref{e.univconstants}})\in (0,1)$ such that for every $\al\in (0,1)$, the following holds: for 
\begin{equation*}
L^{-}_{h}=L^{0}_{h}(\al, \phi):=\left\{x: \phi(x)-h[F_{*}(D^{2}\phi(x), D\phi(x))-\al]<0\right\},
\end{equation*}
we have that for all $h\in (0, h_{0}]$, 
\begin{equation*}
 L^{-}_{h}\cap \overline{B(x_{0},r)}\subseteq \overline{\Theta}_{t_{0}+h}^{c}.
 \end{equation*}
 \end{enumerate}
 We say that $\left\{\Theta_{t}\right\}_{t\geq 0}$ is a generalized flow with normal velocity $-F$ if and only if it is both a generalized sub and super flow. 
 \end{definition}

The main result we will invoke is the following:
\begin{theorem}\cite[Theorem 3.1]{BS}\label{t.bs1}, 
 Assume \eqref{j.1}-\eqref{j.4} hold. For $t>0$, let 
 \begin{equation*}
 \begin{aligned}
& \Theta^{b}_{t}:=\Big\{x\in \R^{\mathbf{d}}: \liminfs_{\eps\to 0} u^{\eps}(t,x;p)=b\Big\},\, \text{and}\\
& \Theta^{a}_{t}:=\left\{x\in \R^{\mathbf{d}}: \limsups_{\eps\to 0} u^{\eps}(t,x;p)=a\right\},
 \end{aligned}
 \end{equation*}
 where we recall the definition of $\liminfs, \limsups$ from Definition \ref{def:hrlimit}.
 Define 
 \begin{align*}
 \Theta^{b}_{0}=\bigcap_{t>0} \left(\bigcup_{0\leq h\leq t}  \Theta^{b}_{h}\right)\quad \text{and}\quad \Theta^{a}_{0}=\bigcap_{t>0} \left(\bigcup_{0\leq h\leq t}  \Theta^{a}_{h}\right)
  \end{align*}
  If $\Theta^{b}_{0}$ (resp $\Theta^{a}_{0}$) is nonempty, then $\Theta^{b}_{t}$ (resp. $\Theta^{a}_{t}$) is nonempty for sufficiently small $t$, and 
  \begin{align*}
  \text{$\left\{\Theta^{b}_{t}\right\}_{t\geq 0}$ is a generalized super-flow}\\
  \text{$\left\{\overline{\Theta^{a}_{t}}^{c}\right\}_{t\geq 0}$ is a generalized sub-flow}
  \end{align*}
with normal velocity $-F$. 
 \end{theorem}

\begin{remark}
    Observe that if $p:\R^{\mathbf{d}}\rightarrow [a,b]$, and $p$ defines an initial interface $\Ga_{0}$ as in Definition \ref{d.defint}, then the sets
    \begin{equation*}
        \Theta^{b}_{0}=\left\{p(\cdot)>\mu\right\}\quad\text{and}\quad \Theta^{a}_{0}=\left\{p(\cdot)<\mu\right\}. 
    \end{equation*}
  are open and nonempty.
\end{remark}
 
 It turns out that this definition of an abstract flow is closely related to the generalized level-set front propagation introduced in Section \ref{s:abstract}. Recall the solution $u$ solving 
\begin{equation*}
\begin{cases}
u_{t}+F(D^{2}u, Du)=0&\text{in $(0, \infty)\times \R^{\mathbf{d}}$},\\
u(0,x)=d(0,x)&\text{in $\R^{\mathbf{d}}$},
\end{cases}
\end{equation*}
where $d(0,x)$ is the signed distance function to $\Ga_{0}$. 

We will refer to the triplet $(\Ga_{t}, \Theta^{+}_{t}, \Theta^{-}_{t})$ as the generalized level-set evolution when 
\begin{align*}
\Ga_{t}:=\left\{x: u(t,x)=0\right\},\quad \Theta^{+}_{t}:=\left\{x: u(t,x)>0\right\},\quad \Theta^{-}_{t}:=\left\{x: u(t,x)<0\right\}.
\end{align*}

A corollary of this fact is to relate generalized flows to the generalized level-set evolution. 
\begin{corollary}\cite[Corollary 3.1]{BS}\label{c.bs2}
Assume \eqref{j.1}-\eqref{j.4} and that $\Theta^{b}_{0}=(\overline{\Theta^{a}_{0}})^{c}.$ Let $(\Ga_{t}, \Theta^{+}_{t}, \Theta^{-}_{t})_{t\geq 0}$ be the generalized level set evolution of $(\partial \Theta^{b}_{0}, \Theta^{b}_{0}, \Theta^{a}_{0})$ with normal velocity $-F$. Then for all $t\geq 0$, 
\begin{equation*}
\Theta^{+}_{t}\subseteq \Theta^{b}_{t}\subseteq \Theta^{+}_{t}\cup \Ga_{t}\quad\text{and}\quad \Theta^{-}_{t}\subseteq \Theta^{a}_{t}\subseteq \Theta^{-}_{t}\cup \Ga_{t}.
\end{equation*}

\end{corollary}
The proof of Theorem \ref{t.generaldual} is now a direct consequence Theorem \ref{t.bs1},  Corollary \ref{c.bs2} and \eqref{eq:approxdualdef}.
\begin{proof}[Proof of Theorem \ref{t.generaldual}]
By Theorem \ref{t.bs1} and  Corollary \ref{c.bs2}, for $u^{\eps}$ defined by \eqref{e.uepdef}, we have 
\begin{equation*}
\lim_{\eps\to 0}u^{\eps}(t,x; p)=\begin{cases} a&\text{locally uniformly in $\bigcup_{t>0} \left\{t\right\}\times \Theta^{-}_{t}$,}\\
b&\text{locally uniformly in $\bigcup_{t>0} \left\{t\right\}\times \Theta^{+}_{t}$}. 
\end{cases}
\end{equation*}
Now by \eqref{eq:approxdualdef}, $\E^{\eps}_{p}[w^\eps_{t}(x)]$ converges locally uniformally to generalized $(a,b)$-MCF started from $\Ga_{0}$(as in Definition \ref{d.lumcf}).

\end{proof}

\section*{Acknowledgements}
We thank Rick Durrett, Xiangying Huang, and Sarah Penington for helpful discussions throughout the completion of this work. TH was partially supported by NSERC Postdoctoral Fellowship 545729-2020 and NSERC Discovery Grant 247764. JL was partially supported by NSERC Discovery Grant 247764, FRQNT Grant 250479, and the Canada Research Chairs program. 

\small
\bibliographystyle{abbrv}
\bibliography{mcffp}

\end{document}